\documentclass{amsart}
\usepackage{cancel}
\usepackage{amsthm,amsmath,amscd}
\usepackage{amssymb}
\usepackage{graphicx}
\usepackage{psfrag}
\usepackage{latexsym, epsfig}
\usepackage[dvipsnames]{xcolor}
\usepackage[bbgreekl]{mathbbol}
\usepackage[refpage,notocbasic]{nomencl} 
\usepackage{xpatch,textcase}

\usepackage{helvet,bbm}

\usepackage{tikz}
\usepackage{fp}

\usepackage{hyperref}

 \newfont{\smc}{cmcsc10 at 10pt}
\newfont{\ind}{cmcsc10}
\newfont{\testo}{cmr10 at 10pt}
\newfont{\mail}{cmtt10}
\newfont{\tit}{cmbx10 at 18pt}

\usepackage[margin=1.8cm]{geometry}

\newtheorem{lemma}{Lemma}
\newtheorem{theorem}[lemma]{Theorem}
\newtheorem{corollary}[lemma]{Corollary} 
\newtheorem{proposition}[lemma]{Proposition} 
\newtheorem{definition}[lemma]{Definition} 
\newtheorem{remark}[lemma]{Remark} 
\newtheorem{example}[lemma]{Example} 

\newcommand{\inte}[1]{[\![ #1]\!]}
\newcommand{\norm}[1]{\left\lVert #1 \right\rVert}
\newcommand{\abs}[1]{\left\lvert #1 \right\rvert}
\newcommand{\pr}[1]{\left(#1 \right)}

\newcommand{\prq}[1]{\left[#1 \right]}
\newcommand{\sour}{\mathrm {source}}
\newcommand{\ac}{\mathrm {a.c.}}
\newcommand{\jump}{\mathrm {jump}}
\newcommand{\cont}{\mathrm {cont}}
\newcommand{\rin}{\mathrm {in}}
\newcommand{\out}{\mathrm {out}}
 \newcommand{\Cantor}{\mathrm {Cantor}}
 
 \newcommand{\uno}{\mathbbm{1}}

\newcommand{\Lf}{\mathfrak L}

\newcommand{\dH}{ \delta_{H}}

\newcommand{\dss}{{\delta}}
\newcommand{\dssb}{\overline{\delta}}

\newcommand{\wlim}[1]{\mathrm{weak^{*}-}\lim_{#1}}
\newcommand{\bel}[1]{\begin{equation}\label{#1}}
\def\bas#1\eas
{\begin{align*}#1\end{align*}}
\newcommand{\eal}{\end{align} }
\def\ba#1\ea
{\begin{align}#1\end{align}}
\newcounter{stepnb}
\newcommand{\firststep}{\setcounter{stepnb}{0}}
\newcommand{\step}[1]{{{\sc \addtocounter{stepnb}{1}\vskip.35\baselineskip\noindent $\circleddash$ Step \arabic{stepnb}:} #1.}} 
\newcommand{\vSC}[2]{{ P_{#1,#2}}}
\newcommand{\vSB}[2]{{ S_{#1,#2}}}
\newcommand{\ftS}[2]{{ S^{FT}_{#1,#2}}}
\newcommand{\ww}{{w}}
\newcommand{\uu}{{ u}}
\newcommand{\vv}{{ v}}

\newcommand{\curva}{{y}}
\newcommand{\wm}{{{\mathfrak\upsilon}}}

\newcommand{\evidenzia}[1]{\noindent \textbf{#1.}\ }

\newcommand{\be}{\begin{equation}}
\newcommand{\ee}{\end{equation}}

\newcommand{\clos}{\mathop {\rm clos}\nolimits}

\newcommand{\rarbound}{2\eps}
\newcommand{\elleuno}{{\mathbb{L}^{\strut 1}}}
\newcommand{\real}{\mathbb{R}}
\newcommand{\R}{\mathbb{R}}

\newcommand{\nat}{\mathbb{N}}
\newcommand{\N}{\mathbb{N}}

\newcommand{\Ll}{\mathcal{L}}
\newcommand{\Ha}{\mathcal{H}}

\newcommand{\Qg}{\mathcal{Q}}

\newcommand{\Z}{\mathbb{Z}}
\newcommand{\OL}{\mathcal{O}}

\newcommand{\I}{\mathcal{I}}
\newcommand{\J}{\mathcal{J}}
\newcommand{\E}{\mathcal{E}}

\newcommand{\NP}{\mathcal{NP}}

\newcommand{\eps}{\varepsilon}
\newcommand{\dds}{\frac{\rm d}{{\rm d}s}}

\newcommand{\TV}{\hbox{\testo Tot.Var.}}

\newcommand{\SBV}{\mathrm{SBV}}
\newcommand{\BV}{\mathrm{BV}}
\newcommand{\loc}{\mathrm{loc}}

\newcommand{\second}{{\prime\prime}}

\newcommand{\scalar}[2]{{\big\langle #1, #2 \big\rangle}}
\newcommand{\ddt}{\frac d{dt}}
\newcommand{\ddr}{\frac d{dr}}
\numberwithin{equation}{section}

\newcommand{\heart}{\ensuremath\heartsuit}

\DeclareMathOperator{\Graph}{Graph}
\DeclareMathOperator{\ri}{r.i.}

\makenomenclature

\begin{document}
%
%
\title{\uppercase{SBV regularity of Entropy Solutions for Hyperbolic Systems of Balance Laws with General Flux function}}

\author{
 {\scshape Fabio Ancona} \email{ancona@math.unipd.it}}
 
\author{{\scshape Laura Caravenna} \email{laura.caravenna@unipd.it}}

\author{{\scshape Andrea Marson} \email{marson@math.unipd.it} } 
 
 \address{Dipartimento di Matematica `Tullio Levi-Civita'
 Via Trieste, 63,
 35121 - Padova, Italy} 
\date{\today}

\begin{abstract}
We prove that vanishing viscosity solutions to smooth non-de\-ge\-ne\-ra\-te systems of balance laws, having small bounded variation, in one space dimension, must be functions of special bounded variation ($\SBV$).
For more than one equation, this $\SBV$-regularity for non-degenerate fluxes is new also in the case of systems of conservation laws outside the context of genuine nonlinearity.
For general smooth strictly hyperbolic systems of balance laws, this regularity fails, as known for systems of conservation laws: in such case we generalize the $\SBV$-like regularity of the eigenvalue functions of the Jacobian matrix of the flux from conservation to balance laws.
Proofs are based on extending Oleinink-type balance estimates, with the introduction of new source measure, a localization argument from~\cite{Robyr,BYTrieste}, and observations in real analysis.
\end{abstract}

\maketitle

2010\textit{\ Mathematical Subject Classification:} 35L45, 35L65

\keywords{hyperbolic systems; vanishing viscosity solutions; SBV regularity; balance laws}
 %
%
%
%
%
%

\tableofcontents

\markboth{F. Ancona, L. Caravenna, A. Marson}{$\SBV$ regularity of entropy solutions for hyperbolic balance laws}

\section[Introduction]{Introduction and simplified statements}
\label{sec:int}
\indent

Consider the Cauchy problem for a general hyperbolic system of $N$ quasilinear first-order PDEs in one space dimension
\begin{align}
\label{eq:sysnc}
& u_t+A(u) u_x =g(t,x,u)\,,\\
\noalign{\smallskip}
\label{eq:inda}
& u(0,x) =\overline u(x)\,.
\end{align}
Here, the vector $u=u(t,x)=\big(u_1(t,x),\dots, u_N(t,x)\big)$ is the unknown, and
$A=A(u)$ is a smooth matrix-valued function defined on an open domain
$\Omega\subseteq\real^N$. Solutions to Eqs.~\eqref{eq:sysnc}-\eqref{eq:inda}
are considered as limits in $L^1_{\loc}([0,T]\times\R;\Omega)$, for some $T>0$, of vanishing viscosity approximations
\begin{equation}
\label{eq:vv}
u^\eps_t+A(u^\eps) u^\eps_x =g(t,x,u^\eps) + \eps u^\eps_{xx}
\end{equation}
as $\eps\to 0+$, as derived in~\cite{BB,Ch1}. In case $A$ is the Jacobian matrix
of a \emph{flux function} $F:\Omega\rightarrow \real^N$, Equation~\eqref{eq:sysnc}
can be written as a system of balance laws, namely
\begin{subequations}\label{eq:syscfSh}\begin{equation}
\label{eq:syscf}
u_t+F(u)_x =g(t,x,u)\,.
\end{equation}
We assume that the
system in Equation~
\eqref{eq:sysnc} is strictly hyperbolic, i.e. that the
matrix $A(u)$ has $N$ real distinct eigenvalues
\begin{equation}
\lambda_1(u)<\dots<\lambda_N(u)\qquad\forall~u\in\Omega\,,
\label{eq:strhyp}
\end{equation}
\end{subequations}
and we will denote by
\begin{equation}
\label{eq:rle}
r_1(u),\ldots, r_N(u)\,, \qquad
\ell_1(u),\ldots,\ell_N(u)
\end{equation}
corresponding bases of, respectively, right and left
eigenvectors, normalized so that
\begin{equation}
\label{eq:norm}
\vert r_k(u)\vert\equiv 1\,,
\qquad
\scalar{r_k(u)}{\ell_h(u)} = \delta_{kh}\,,
\end{equation}
where $\scalar{\cdot}{\cdot}$ stands for the usual scalar product in $\R^N$,
and $\delta_{hk}$ is the usual Kronecker symbol.
The limits of vanishing viscosity approximations to~\eqref{eq:syscf}-\eqref{eq:inda}
turn out to be distributional solutions which are entropy admissible (see~\cite{Ch1}).
We clarify  the following assumption.

\begin{description}
\item[(G)]\label{Ass:G}
the function $g:[0,T]\times\R\times\Omega\rightarrow \real^N$ in Equation~
\eqref{eq:sysnc} is continuous in $t$ and Lipschitz
continuous w.r.t. $x$ and $u$, uniformly in $t$; moreover, there exists a
function $\alpha\in L^1(\R)$ such that $\vert g_x(t,x,u)\vert \leq
\alpha(x)$\label{alpha} for any $t\in[0,T]$, $u\in\Omega$ and a.e.~$x\in\R$.
\end{description}
\begin{remark}
Regarding the assumptions
on $g$, since the seminal papers~\cite{DafHsiao, TPLin} and
the first paper~\cite{CP} on the well-posedness
of the Cauchy problem,
many papers appeared in the past years dealing with several
existence results for first-order hyperbolic
inhomogeneous systems, both local and global in time, within the setting of small $\BV$ solutions. We refer to the brief review in~\cite{ACM1}.  Therein, Theorem 4.4 states existence, in general locally in time under our hypothesis, to the Cauchy problem on the real line.
See also~\cite[Theorem~16.1.3]{Daf_book}.

Our regularity result is of a local nature both in space and time: we could thus also have worked locally in space thus removing the assumption that $\alpha\in L^1(\R)$.
We feel it is not relevant: if $\overline \lambda$ is a bound on the speed of propagation, reduction arguments can embed a Cauchy problem at time $\overline t$ on an interval $(a,b)$ into a Cauchy problem on a real line having the same solution on a rectangle $(a-\overline \lambda\delta, b-\overline \lambda\delta)\times(\overline t,\overline t+\delta)$, for some $\delta>0$.
See, for example, Step 4 in the proof of Theorem~\ref{T:sbvhyp} below for a localization procedure.
\end{remark}\vspace{.1truecm}

This paper is concerned with the $\SBV$ and $\SBV$-like regularity of entropy weak solutions to Equation~\eqref{eq:syscf}
or of vanishing viscosity solutions of Equation~\eqref{eq:sysnc}.
We recall the following definition.
\nomenclature{$\SBV(I;\Omega)$}{See Definition~\ref{D:cantorPart1d}, when $\Omega\subset \R^{N}$ is open and $I= (a,b)\subseteq\R$}
\begin{definition}[Cantor part of the derivative of a $\BV$-function of one variable]
\label{D:cantorPart1d}
Suppose $u^{*}:(a,b)\to\R$ has bounded variation, thus $D_{x}u^{*}$ is a Radon measure: then we denote by $D^{\Cantor}_{x}u^{*}$ the Cantor part of its derivative, defined as the continuous part of the measure $D_{x}u^{*} $ which is not absolutely continuous in the Lebesgue $1$-dimensional measure:
\begin{equation}\label{E:decomposition}
D_{x}u^{*}=D^{\ac}_{x}u^{*}+D^{\Cantor}_{x}u^{*}+D^{\jump}_{x}u^{*}\,,
\end{equation}
where $D^{\ac}_{x}u^{*}$ is absolutely continuous in $\Ll^{1}$ and $D^{\jump}_{x}u^{*}$ is purely atomic.

If $u^{*}:(a,b)\to\Omega\subset \R^{N}$ is vectorial, we decompose every component and we obtain a vector measure.
We still write the same decomposition~\eqref{E:decomposition} with vectorial measures.

If $D^{\Cantor}_{x}u^{*}=0$ we say that $u^{*}$ is a special function of bounded variation and we denote $u\in\SBV((a,b);\Omega)$.
\end{definition}

We show that when the characteristic families are genuinely nonlinear~\cite{Lax}, or more generally \emph{non-degenerate}, see Definition~\ref{def:ND} in \S~\ref{S:SBVND}, on the line of what was introduced in~\cite{IguchiLeFloch,LeFlochGlass}, the vanishing viscosity solutions of the non homogeneous system of Eqs.~\eqref{eq:sysnc} starting with small $\BV(\R)$ data gains $\SBV(\R)$ regularity at all times except at most countably many. In particular, the solution is a special function of bounded variation on the strip where it is defined.
This means that solutions to non-degenerate smooth balance laws preserve the same regularity first obtained for the corresponding homogeneous system with $g=0$ when fields are genuinely nonlinear, thus extending to such non-homogeneous systems the results established for systems of conservation laws in~\cite{AmbrosioDeLellis,Robyr,DafSBV,BCSBV,BYTrieste}.
See also~\cite{AdiSBV} for a related result in the case of a strictly convex, but not uniformly convex, flux.

We stress that this regularity result is new also in the case of systems of conservation laws, which means when $g=0$, for fluxes that are non-degenerate but having fields that are not genuinely nonlinear.
This extension is particularly relevant because systems that are non-degenerate, but that do not satisfy the classical assumptions of genuine nonlinearity in the sense of Lax, may arise
in several contexts.

\begin{example}A first example is a
system of balance laws arising in modelling elasticity,
\begin{equation}
\label{eq:elas}
\begin{cases}
v_t-u_x=0\\
\noalign{\smallskip}
u_t-\sigma(v)_x = -\alpha u\,,
\end{cases}
\end{equation}
where the stress $\sigma=\sigma(v)$ satisfies $\sigma'(v)>0$.
This system has been studied diffusely (e.g., see~\cite{Daffricdam,DiPerna}). 
Its behavior resembles the $p$-system with dumping~\cite{Daffricdam,DafPan},
but its characteristic fields are not genuinely nonlinear, nor linearly degenerate.
\end{example}

\begin{example}
Another $2\times 2$ system of balance laws not fulfilling the classical
Lax assumptions on characteristic fields is the generalized
Cattaneo's model of heat conduction in high purity
crystals~\cite{RMScimento,RMS,SRM},
\begin{equation}
\label{eq:cattaneo}
\begin{cases}
\rho e_ t +q_x=0\\
\noalign{\smallskip}
(\alpha q)_t + \nu_x = -\dfrac{\nu'}{k} q\,.
\end{cases}
\end{equation}
\end{example}

In order to state this regularization result now, at least in the simple setting, we recall basic definitions.

\begin{definition}
\label{D:LD}
The $i$th-characteristic field is \emph{linearly degenerate} if $\nabla\lambda_{i}(u)\cdot r_{i}(u)\equiv0$ for all $ u\in\Omega$.
\end{definition}

\begin{definition}
\label{D:GN}
The $i$th-characteristic field is \emph{genuinely nonlinear} if $\nabla\lambda_{i}(u)\cdot r_{i}(u)\neq0$ for all $u\in\Omega$. We choose the direction of $r_{i}(u)$ so that $\nabla\lambda_{i}(u)\cdot r_{i}(u)>0$.
\end{definition}


\begin{theorem}[$\SBV$-regularity]
\label{Th:SBV}
Let $A(u)\in C^{1}(\R)$ satisfy the strictly hyperbolicity Assumption~\eqref{eq:strhyp} and suppose that the source $g$ satisfies Assumption~\textbf{(G)} at Page~\pageref{Ass:G}.
Suppose the vanishing viscosity solution $u$ to~\eqref{eq:sysnc}-\eqref{eq:inda}, constructed as in~\cite{Ch1}, is defined in $[0,T]\times\R$ for initial data with sufficiently small total variation.
When every characteristic field of system~\eqref{eq:sysnc} is genuinely nonlinear, the derivative $D_x u(t)$ of the vanishing viscosity solution at time $t\in(0,T]$, except at most countably many times, is the sum of a purely atomic measure and of a measure which is absolutely continuous with respect to the Lebesgue one-dimensional measure $\Ll^{1}$.

The same result holds also in the case of non-degenerate fluxes, intended as satisfying  the generic nonlinearity assumption in Definition~\ref{def:ND} of \S~\ref{S:SBVND}.
\end{theorem}

The case of of non-degenerate fluxes is more technical but relevant, because non-degenerate fluxes are generic, see \S~\ref{S:SBVND}.

In particular, Theorem~\ref{Th:SBV} rules out, at all times except at most countably many, the presence of a Cantor-like behavior. Cantor-like behaviors are present, for example, in monotone continuous functions that are not $W^{1,1}_{\loc}(\R)$: the derivative is continuous but not absolutely continuous compared to~$\Ll^{1}$.
This result generalizes the previous results in the scalar $1d$-case~\cite{AmbrosioDeLellis,Robyr}, respectively, for a uniformly convex conservation law and for a balance law with at most countably many inflection points, and the results for genuinely nonlinear, strictly hyperbolic $1d$-systems of conservation laws~\cite{DafSBV,BCSBV}, respectively in case of the Riemann and of the Cauchy problem on the real line.

In the case a single field is genuinely nonlinear, but the system is strictly hyperbolic and the assumptions on the source still hold, there is a similar regularity statement for the projection of $D_{x}u(t)$ on the direction $\ell_{i}(u)$, see Theorems~\ref{T:SBVGNbalance} in \S~\ref{S:SBVGNargument}. This is the main tool for proving Theorem~\ref{Th:SBV} in the genuinely nonlinear case.
The case of nondegenerate fluxes is then stated and proved as the separate Theorem~\ref{C:SBVGNbalance1den} in \S~\ref{S:SBVND}.

Unfortunately, when some field is linearly degenerate, there is no hope for such regularization even in the context of conservation laws: see Example~\ref{Ex:counterBVlike} reported from~\cite{BYTrieste}. Nevertheless, the Cantor-like behavior, if present, must really be concentrated on the region where linear degeneracy indeed happens; this was proved for conservation laws in~\cite{BYTrieste}.
See Theorem~\ref{T:sbvhyp} for the precise statement: here, we rephrase its essential meaning.

\begin{theorem}[$\SBV$-like regularity]
\label{Th:SBVlike}
Let $A(u)\in C^{1}(\R;\Omega)$ satisfy the strictly hyperbolicity Assumption~\eqref{eq:strhyp} and suppose $g$ satisfies Assumption~\textbf{(G)} at Page~\pageref{Ass:G}. 
Suppose the vanishing viscosity solution $u$ to~\eqref{eq:sysnc}-\eqref{eq:inda}, constructed as in~\cite{Ch1}, is defined in $[0,T]\times\R$ for initial data with sufficiently small total variation.
Let $i\in \{1,\dots,N\}$. Then, there exists a $\sigma$-compact set
\[
 K\subset \{\text{continuity points } (t,x) \text{ of } u\ : \nabla\lambda_i(u(t,x))\cdot r_i(u(t,x))=0\}
 \]
having zero $\Ll^{2}$-measure such that the following holds: for $t\in[0,T]$, except at most countably many times, the derivative $D_{x}u(t)\cdot \ell_{i}(u(t))$ restricted on $\{x\ :\ (t,x)\not\in K\}$ is the sum of a purely atomic measure and of a measure absolutely continuous with respect to the Lebesgue one-dimensional measure.
\end{theorem}

This generalizes the previous result~\cite{BYTrieste} related to the case of strictly hyperbolic $1d$-systems of conservation laws.
Classical interesting examples where this $\SBV$-like reg\-u\-la\-ri\-ty applies are the following.
\begin{example}
A model from traffic flow which has a degenerate flux, where thus still the result of $\SBV$-like regularity applies, is~\cite{AwRascle,Colombo}: in the unknown $(\rho,q)$ the system is
\begin{equation*}
\label{eq:Colombo}
\begin{cases}
\rho_ t +(\rho v)_x=0\\
\noalign{\smallskip}
q_t + ((q-q_*)v)_x = 0\,.
\end{cases}
\end{equation*}
Above, $\rho=\rho(t,x)\in[0,\rho_M]$ is the car density, $\rho_M$ depends on the street, $v=v(\rho,q)=(\rho^{-1}-\rho_M^{-1})q$ is the car speed, and $v$ is an auxiliary variable.

In the original ARZ model, since drivers operating at speed $v$ in traffic with local density $\rho$ never adjust their vehicles’ operation to achieve the maximum allowable speed, one accounts for that with a source term:
\begin{equation*}
\label{eq:AwRm}
\begin{cases}
\rho_ t +(\rho v)_x=0\\
\noalign{\smallskip}
\partial_{t}\left((v+p(\rho))\rho  \right) +((v+p(\rho))\rho v)_x = \tau^{-1}(V(\rho)-v)\rho\,,
\end{cases}
\end{equation*}
where $\tau>0$ is interpreted as the adjustment time, $p$ as a `pseudo' pressure, $V$ the desired velocity.
See, on this model, e.g.~\cite{FanSei,green,rasclNH}.
\end{example}

\begin{example}
A model for granular flow where the result of $\SBV$-like regularity applies is~\cite{HadelerKuttler,ShenSmooth,AmadoriShen}:
\begin{equation}
\label{eq:granular}
\begin{cases}
h_ t +(h p)_x=(p-1) h\\
\noalign{\smallskip}
p_t + ((p-1)h)_x = 0\,,
\end{cases}
\end{equation}
in the domain $\Omega=\{(h,p)\ : \ h\geq 0\,, \ p>0\}$, where $h$ is the height of the moving layer of sand and $p$ is the slope of the standing layer, sometimes eroded sometimes constructed.
\end{example}

The structure of the paper is the following:
\begin{itemize}
\item[\S~\ref{S:SBVGNargument}] We describe the `upper level argument' of the proof of $\SBV$-regularity for genuinely nonlinear systems.
\item[\S~\ref{S:SBVlike}]We describe the `upper level argument' of the proof of $\SBV$-like regularity for linearly degenerate systems.
\item[\S~\ref{S:SBVND}]We prove the $\SBV$-regularity for nondegenerate systems.
\end{itemize}
By `upper level argument' we mean that we postpone the more technical part of the construction of auxiliary measures, based on careful estimates, that provide the key tool for the proof just because they exist as Radon measures providing key balances. The more technical construction of such measures is done in the Appendix:
\begin{itemize}
\item[\S~\ref{sec:PCA}] We revise the necessary terminology and well-established constructions in the field of systems of balance laws.
\item[\S~\ref{S:qualitative}]We introduce new measures that are key tools for balances providing the above-described regularity results.
\end{itemize}

The reader who finds this complex could benefit of the introduction to such strategy in the case of a single conservation and balance law, see~\cite{notaBianchini,notaCons,ACMnota}.


\section{\texorpdfstring{$SBV$}{SBV}-regularity for genuinely nonlinear fields}
\label{S:SBVGNargument}

In this section, we outline the statements and the ``upper-level'' arguments of the proof of $\SBV$-regularity for genuinely nonlinear fields. Since we want to avoid here most technicalities, we postpone to later sections the construction of relevant measures that enter in the key estimates, and as well as the proof of auxiliary estimates. Although the general strategy was also used in the previous paper~\cite{BCSBV}, the construction and meaning of the relevant measures are different, as will be explained in the next section, where the influence of the source term will be clear.

We decompose the spatial derivative of $u$ along the right eigenvectors.
Since $u$ is generally discontinuous, and thus its derivative has a jump part, in order to do this decomposition, we need to define a point-wise representative of $\ell_{i}(u(t,x))$, $r_{i}(u(t,x))$, $\lambda_{i}(u(t,x))$ at discontinuity points.
This will not be visible in most statements, since we are interested in the $\SBV$-regularity of the solution $u$, and the Cantor part of the derivative of $u$ is concentrated on a set of points where $u$ is continuous. In addition, since we do not need it, we do not explain how the point-wise definition of such compositions is related to the solution to the Riemann problem via center manifold in~\cite{srp}.

Let $u_{L}, u_{R}\in\Omega$. Consider the average matrix
\[
\widetilde A(u_{L},u_{R})\doteq \int_{0}^{1}A(\vartheta u_{L}+(1-\vartheta)u_{R})\,d\vartheta\,,
\]
and for $i=1,\dots,N$, set $\lambda_{i}(u_{L}, u_{R})$ its $i$-th eigenvalues, while $ {\ell}_{i}(u_{L}, u_{R})$ and $ {r}_{i}(u_{L}, u_{R})$ denote its left and right eigenvectors, respectively, normalized so that 
\begin{equation}\label{E:normalizationChoice}
\vert r_k(u_{L}, u_{R})\vert\equiv 1\,,
\qquad
\scalar{r_k(u)}{\ell_h(u_{L}, u_{R})} = \delta_{kh}\,.
\end{equation}
When $u\in \BV_{\loc}(\R;\Omega)$, we are now able to define a pointwise representative of the composition
\begin{align}\label{E:vettoriTilde}
&{ \widetilde \ell}_{i} (t,x)\doteq {   \ell}_{i} \left(u_{L}(t,x),u_{R}(t,x) \right)\,,\notag\\
&{ \widetilde r}_{i} (t,x)\doteq {   r}_{i} \left(u_{L}(t,x),u_{R} (t,x)\right)\,,\\
&{ \widetilde \lambda}_{i} (t,x)\doteq {   \lambda}_{i} \left(u_{L}(t,x),u_{R}(t,x) \right)\,,\notag
\end{align}
where $u_{L}(t,x)\doteq \lim_{h\to 0^{+} }u(t,x-h)$ and $u_{R}(t,x)\doteq\lim_{h\to 0^{+} }u(t,x+h)$.

\begin{definition}[Wave measures]
\label{D:upsiloni}
Let $i\in\{1,\dots,N\}$ and $u\in \BV_{\loc}(\R;\Omega)$.
We introduce the $i$-discontinuity measures
\[
\wm_i=D_x u \cdot{ \widetilde \ell}_{i} 
\]
so that by the normalization choice~\eqref{E:normalizationChoice} we have the decomposition $D_x u=\sum_{i=1}^N\wm_i \widetilde r_i$.
\end{definition}

We refer to Definition~\ref{D:cantorPart1d} for the $\SBV$ functions of a single variable.
In the case of genuinely nonlinear fields, even when a source is present, we can now state that the $i$-component of $D_x u$, in a suitable basis, does not possess a Cantor part at every fixed time except at most countably many: it is simply the sum of a purely atomic measure and of an absolutely continuous measure.
For the decomposition of the derivative of $\BV$ functions, we refer to~\cite[Corollary~3.33]{AFPBook} in one variable and to~\cite[\S~3.9]{AFPBook} in several variables.

\begin{theorem}
\label{T:SBVGNbalance}
Let $u\in\BV_{\loc}([0,T]\times\R;\Omega)$ be an entropy solution of the strictly hyperbolic system of balance laws~\eqref{eq:sysnc} with a locally small $\BV$ norm. 
Suppose $g$ satisfies Assumption~\textbf{(G)} at Page~\pageref{Ass:G}. Suppose that the $i$-th field is genuinely nonlinear, but not necessarily the other ones.
Then there exists an at most countable set $S\subset[0,T]$ of times such that the Cantor part of $D_{x}u(t,\cdot)\cdot\widetilde \ell_i(u(t,\cdot)$ disappears for $t\in[0,T]\setminus S$.
\end{theorem}

\begin{proof}[Proof of Theorem~\ref{Th:SBV}]
Since $D_{x}u=\sum_{i=1}^{N}D_x u \cdot{ \widetilde \ell}_{i} $ then Theorem~\ref{T:SBVGNbalance} immediately implies Theorem~\ref{Th:SBV}.
\end{proof}

The Cantor part of the derivative of a function $u^{*}$ of $2$ variables is usually defined, see~\cite[\S~3.9]{AFPBook}, by means of the Decomposition~\eqref{E:decomposition}, proving that it holds with
\begin{itemize}
\item  a measure $D^{\ac}_{x}u^{*}=\nabla u\Ll^{2}$ which is absolutely continuous in the two-dimensional Le\-bes\-gue measure, where $\nabla u$ is the (well defined!) approximate differential of $u$,
\item  a singular measure $D^{\jump}_{x}u^{*}=(u_{L}-u_{R})\otimes \widehat n\Ha^{1}\restriction_{J}$, where $J$ is the union of at most countably many Lipschitz $1$-dimensional manifolds with normal $\widehat n$ and $u_{L}$, $u_{R}$ are the (well defined!) left and right strong traces of $u$ on it,
\end{itemize}
and with a remaining measure $D^{\Cantor}_{x}u^{*}$ which is not absolutely continuous with respect to $\Ll^{1}$ but which vanishes on sets which are $\sigma$-finite with respect to $\Ha^{1}$.
It can be as well understood by the disintegration~\eqref{item:slicingDx} below.
\nomenclature{$\Ll^{M}$}{The Lebesgue $M$-dimensional measure, where $M\in\N\cup\{0\}$}
\nomenclature{$\Ha^{M}$}{The Hausdorff $M$-dimensional external measure, where $M\in\N\cup\{0\}$}
\nomenclature{$\SBV_{\loc}((t_{1},t_{2})\times\R;\Omega)$}{Subspace of $\BV_{\loc}((t_{1},t_{2})\times\R;\Omega)$, where $\Omega\subseteq \R^{M}$ is open, consisting of those functions whose derivative does not have the Cantor part, see~\cite[\S~3.9]{AFPBook}}

\begin{corollary}
\label{C:SBVGNbalance}
Under the assumptions of Theorem~\ref{Th:SBV}, then $u\in\SBV_{\loc}([0,T)\times\R;\Omega)$, which means that the measure $D^{\Cantor}_{x}u$ vanishes.
\end{corollary}

\begin{proof}[Proof of Corollary~\ref{C:SBVGNbalance}]
By the slicing theory of $\BV$ functions (\cite[Theorems 3.107-108]{AFPBook}) we know that the Cantor part $D^{\Cantor}_{x}u$ of $D_{x}u$ is the measure given by the disintegration
\begin{equation}\label{item:slicingDx}
D^{\Cantor}_{x}u=\int_{0}^{T}D^{\Cantor}_{x}u(t)\,dt\,,
\end{equation}
where $D^{\Cantor}_{x}u(t)$ is the Cantor part of $D_{x}u(t)$: $D^{\Cantor}_{x}u$ thus vanishes because $D_{x}u(t)$ vanishes at all times except at most countably many by Theorem~\ref{T:SBVGNbalance}.

The vanishing of $D^{\Cantor}_{t}u$ can be deduced by the relation \[D^{\Cantor}_{t}u=- D^{\Cantor}_{x}f(u)+g(t,x,u)\,,\] since $u$ solves equation~\eqref{eq:syscf}.
Indeed, $D^{\Cantor}_{x}f(u)=A(\widetilde u)D^{\Cantor}_{x}u=0$ by Volpert Chain rule~\cite[Theorem~3.96]{AFPBook}, where $\widetilde u$ is a suitable pointwise representative defined at jump points by a specific average, see~\cite{AFPBook}.
\end{proof}

\subsection[Proof of the SBV Theorem: `upper level' argument]{Proof of Theorem~\ref{T:SBVGNbalance}: `upper level' argument}
We outline the general strategy, similar to~\cite{BCSBV}: this is based on Oleinik-type estimates on the $i$-discontinuity measures $\wm_i$ introduced in Definition~\ref{D:upsiloni}, extending~\cite[Theorem 10.3]{BressanBook}.

We can already state such key estimates in Lemma~\ref{L:frrefr}.

\vskip.3\baselineskip

Before entering the statements, we stress that the estimate on positive waves~\eqref{E:iOleinink+} was derived in a stronger form in~\cite[\S~3-4]{Ch3} under dissipativity assumptions on the source term.
Without dissipativity assumptions, there is no reason for having an exponential decay of positive waves: the linear decay known for systems of conservation laws is combined with the action of the source.
The decay of positive waves is a deep problem widely studied in the literature: we refer to references in~\cite{Ch3},~\cite[Pag.~237]{BressanBook} and in~\cite[\S~14.13]{Daf_book}.

The estimate on negative waves~\eqref{E:iOleinink-}, which is no-more a decay due to the formation of jumps, was introduced in~\cite{BCSBV} and it is here generalized.

\vskip.3\baselineskip

We denote by $[\wm_i(t)]^+$ the positive part of $\wm_i(t)$ and by $[\wm_i^{\cont}(t)]^-$ the negative part of the continuous part of the measure $\wm_i(t)$.

\begin{lemma}\label{L:frrefr}
Let $T>0$.
Suppose $\{\wm_i(t)\}_{0<t<T} $ are Radon measures on $\R$ and there exist
\begin{itemize}
\item a constant $C=C(T)$, depending on the particular given system of balance laws and on the $\BV$-bound of the total variation of the initial datum, and 
\item nonnegative Radon measures $\mu^{ICS}\leq \mu^{ICJS}$ which are finite on $[0,T]\times \R$ 
\end{itemize}
such that for all Borel sets $B\subseteq\R$
\begin{subequations}
\label{E:iOleinink}
\ba
\label{E:iOleinink+}
& [\wm_i(t)]^+ (B)\leq C\left(\frac{\Ll^1(B)}{t-s}+\mu^{ICS}([s,t]\times \R) \right)
&&
\text{if $0\leq s< t\leq T$,}
\\
& [\wm_i^{\cont}(t)]^- (B) \leq C\left(\frac{\Ll^1(B)}{s-t}+\mu^{ICJS}([t,s]\times \R) \right)
&&
\text{if $0\leq t< s\leq T$.}
\label{E:iOleinink-}
\ea
\end{subequations}
Recall that $ [\wm_i^{}]^- $, $ [\wm_i^{ }]^+ $ are nonegative.
Then $\wm_i(t)$ is the sum of a purely atomic measure and of an absolutely continuous measure, for each $0<t<T$.
\end{lemma}

The precise definition of the measures $\mu^{ICS}$, $\mu^{ICJS}$ is below in Definitions~\ref{D:measuresnu}-\ref{D:measures}: their construction is indeed one of the most important points of the paper. Such measures control not only interactions, cancellations, formation and evolution of jumps, as in~\cite{BCSBV}, but now also the action of the source.
For the upper-level argument we are explaining in this section, however, what matters is only that they are nonnegative measures, finite on time strips, providing the balances~\eqref{E:iOleinink}.
We now explain why balances~\eqref{E:iOleinink} are all we need to conclude that $\wm_i$ has no Cantor part out of at most countably many times.

\begin{proof}[Proof of Lemma~\ref{L:frrefr}]
Consider for $0\leq r\leq T$ the monotone functions \[p(r)= \mu^{ICJS}([0,r]\times\R)\,,
\qquad
q(r)= \mu^{ICS}([0,r]\times\R)\,.\]

The negative part $(\wm_i)^-$ might have a Cantor part at some time $t$.
In this case, by Definition~\ref{D:cantorPart1d} of Cantor part there exists a Borel set $B$ with $\Ll^1(B)=0$ such that $(\wm_i)^-(B)>0$.
Consider now~\eqref{E:iOleinink-} and take the limit as $s\downarrow t$: we obtain that $0<(\wm_i)^-(B)\leq \mu^{ICS}(\{t\}\times \R) $ thus finding that the time-marginal of $ \mu^{ICS}$ has an atom at time $t$ and the monotone functions $q$ jump at time $t$. This might happen at most countably many times.

The positive part $(\wm_i)^+$ might have a Cantor part at some time $t$.
In this case, according to Definition~\ref{D:cantorPart1d} of the Cantor part there exists a Borel set $B$ with $\Ll^1(B)=0$ such that $(\wm_i)^+(B)>0$.
Consider now~\eqref{E:iOleinink+} and take the limit as $s\uparrow t$: we obtain that $0<(\wm_i)^+(B)\leq \mu^{ICJS}(\{t\}\times \R) $ thus finding that the time-marginal of $ \mu^{ICJS}$ has an atom at time $t$ and the monotone $p$ jumps at time $t$. This might happen at most countably many times.
\end{proof}

The thesis thus amounts to proving Oleinik-type estimates~\eqref{E:iOleinink}. We stress that it is no longer a decay estimate, neither for the positive nor for the negative part of $\wm_{i}^{\nu,\cont}$.The presence of the source term changes the meaning of the estimate, which \emph{formally} as mentioned is the same, changing names to the relevant dominating measures.
%

\nomenclature{$\wm_i$}{The $i$-th wave measure in Definition~\ref{D:upsiloni}}

\begin{remark}
The actual construction of the measures will later show that $\mu^{ICS}$ of $[s,t]\times \R$ is controlled, thanks to  \eqref{E:estimatemuIC} and Lemma~\ref{L:estSource}, by the negative total variation of the Glimm functional $\Qg$, first introduced as a breakthrough for proving existence of solutions to the Cauchy problem in the case of systems, and the length $t-s$ of the time interval.
\end{remark}
\subsection{Approximate Oleinik-type estimates}
While the framework for the $\SBV$-regularity is consolidated, Oleinik-type estimates~\eqref{E:iOleinink} are not yet available for balance laws having the $i$-th field which is genuinely nonlinear.

The Oleinik-type estimates~\eqref{E:iOleinink} for the positive and negative part of $\wm_i(t)$ can be proved by approximation, taking advantage of the piecewise constant approximation $u^{\nu}$ constructed in~\S~\ref{Ss:convergenceAndExistence}: Lemma~\ref{L:convergenceJumps} will ensure that there is a part $ \wm_{i}^{\nu,\cont}$ of the derivative $D_{x}u^{\nu}\cdot\widetilde \ell_i$ weakly$^{*}$-converging to  $ \wm_{i}^{\cont}$, not only $ \wm_{i}^{\nu}$ to $ \wm_{i}^{}$.
Based on that, \S~\ref{S:definitionMeasures} defines Radon measures $\mu^{ICS}_{\nu} $, $\mu^{ICJS}_{\nu} $ that capture the dynamic of the system and weakly$^{*}$-converge to  $\mu^{ICS}_{} $, $\mu^{ICJS}_{} $.
They are finite measures on time-strips.

An intermediate step, preformed in \S\S~\ref{Ss:decay-}-\ref{Ss:decay+}, consists in proving that Oleinik-type estimates approximately hold on a subsequence of $u^{\nu}$: namely
\begin{subequations}
\label{E:iOleininknu}
\ba
\label{E:E:iOleininknu-}
&
- {\left[ \wm_{i}^{\nu,\cont}(s)\right](J)} \leq C\left[\frac{\Ll^{1}(J)}{t-s}+\left(\mu_{\nu}^{ICJS} \right)\left([s,t]\times \R\right)+k \beta_{\nu}\right]
\ea
for $0\leq s<t$, and still for $0\leq s<t$ one has
 \ba
 \label{E:iOleininknu+}
& {\left[ \wm_{i}^{\nu}(t)\right](J)} \leq C\left[\frac{\Ll^{1}(J)}{t-s}+\left(\mu^{ICS}_{\nu} \right)\left([s,t]\times \R\right)+k\beta_{\nu} \right]\, ,
\ea
\end{subequations}
where $J$ is any union of $k$ closed intervals.
The estimate of the positive part of $\wm_{i}^{\nu,\cont} $ is more classical, as in~\cite[Theorem 10.3]{BressanBook}, while the negative part required in~\cite{BCSBV} the introduction of the jump measure that we will introduce below.
We now prove that a limiting procedure then gives thesis as $\beta_\nu$, $\varepsilon_\nu$ are vanishingly small approximation errors.

\begin{lemma}
Let $T>0$.
Suppose $\{\wm_i^\nu(t)\}_{0<t<T} $ are Radon measures in $\R$ converging weakly$^*$ to $\{\wm_i(t)\}_{0<t<T} $, with each $\wm_{i}^{\nu,\cont}(t)$ converging weakly$^*$ to $\wm_{i}^{\cont}(t)$. 
Suppose there exist
\begin{itemize}
\item a constant $C=C(T)$, depending on the particular given system of balance laws and on the $\BV$-bound of the total variation of the initial datum, and 
\item nonnegative Radon measures $\mu^{ICS}_{\nu}\leq \mu^{ICJS}_{\nu}$ which are finite on $[0,T]\times \R$, converging weakly$^*$ to  $\mu^{ICS}_{ }, \mu^{ICJS}_{ }$,
\end{itemize}
such that estimates~\eqref{E:iOleininknu} hold for all  $J$ union of finitely many closed intervals. Then~\eqref{E:iOleinink} hold for all Borel sets $B\subseteq\R$.
\end{lemma}

\begin{proof} 
For every Borel set $B$ let $B^\circ$ be its interior.
By outer regularity of the measures $ [\wm_i(t)]^+$, $[\wm^{\cont}_i(t)]^-$, $[\wm^{\jump}_i(t)]^-$, and since they are nonnegative by definition, for any Borel set $B$ and every $\varepsilon>0$ there is a finite union $J_{\varepsilon}$ of $k_{\varepsilon}\in\N$ closed intervals which satisfies
\begin{subequations}\label{E:gagragaggargar}
\begin{gather}
[\wm_i(t)]^+(B )\leq  [\wm_i(t)]^+(J_{\varepsilon}^\circ)+\varepsilon\,,\\
[\wm^{\cont}_i(t)]^-(B )\leq  [\wm^{\cont}_i(t)]^-(J_{\varepsilon}^\circ)
+\varepsilon\\
 [\wm^{\jump}_i(t)]^-(B )\leq  [\wm^{\jump}_i(t)]^-(J_{\varepsilon}^\circ) +\varepsilon\,,\\
 \Ll^{1}(J_{\varepsilon})\leq \Ll^{1}(B ) +\varepsilon\ .
\end{gather}
\end{subequations}
As $\wm_{i}^{\nu}(t)$ converges weakly$^*$ to $\wm_{i}^{ }(t)$: there holds
\begin{gather*}
[\wm_i(t)]^+(J_{\varepsilon}^\circ)\leq \liminf_{\nu\to\infty}[\wm_i^\nu(t)]^+(J_{\varepsilon}^\circ)\;,
 \\
 [\wm_i(t)]^-(J_{\varepsilon}^\circ)\leq \liminf_{\nu\to\infty}[\wm_i^\nu(t)]^-(J_{\varepsilon}^\circ)
\end{gather*}
by lower semicontinuity on open sets of (nonnegative) measures for weak*-convergence.
Of course, since $J_{\varepsilon}^\circ\subseteq J_{\varepsilon}$ and we defined nonnegative measures,
\bas
&[\wm_i^\nu(t)]^+(J_{\varepsilon}^\circ)\leq [\wm_i^\nu(t)]^+(J_{\varepsilon})
&&\text{and}
&&[\wm_i^\nu(t)]^-(J_{\varepsilon}^\circ)\leq [\wm_i^\nu(t)]^-(J_{\varepsilon})\;.
\eas
Moreover by definition of $\mu^{ICJS}$ and of $\mu^{ICS}$, limit of $\mu^{ICJS}_{\nu}$ and $\mu^{ICS}_{\nu}$, there holds also
\begin{gather*}
 \limsup_{\nu\to\infty}\mu^{ICS}_{\nu}\left([s,t]\times\R\right)\leq \mu^{ICS}\left([s,t]\times\R\right) \;,\\
 \limsup_{\nu\to\infty} \mu^{ICJS}_{\nu}\left([s,t]\times\R\right)\leq \mu^{ICJS}\left([s,t]\times\R\right)
\end{gather*}
by the upper semicontinuity on closed sets of (nonnegative) measures for weak*-convergence.

For the particular sequence that satisfies~\eqref{E:iOleininknu} we thus get
\bas
  [\wm_i(t)] (J_{\varepsilon}^{\circ}) \leq C\left(\frac{\Ll^1(J_{\varepsilon})}{t-s}+\mu^{ICS}([s,t]\times \R) \right) 
&&
\text{if $0\leq s< t$,}
\\
  -[\wm_i^{\cont}(s)] (J_{\varepsilon}^{\circ})  \leq C\left(\frac{\Ll^1(J_{\varepsilon})}{t-s}+\mu^{ICJS}([s,t]\times \R)\right) 
&&
\text{if $0\leq s< t$,}
\eas
so that by~\eqref{E:gagragaggargar} we get
\bas
[\wm_i(t)] (B)&\leq  C\left(\frac{\Ll^1(B)}{t-s}+\mu^{ICS}([s,t]\times \R) \right)+\left(1+\frac{1}{t-s}\right)C\varepsilon\,,
\\
-[\wm_i^{\cont}(s)] (B)& \leq   C\left(\frac{\Ll^1(B)}{t-s}+\mu^{ICJS}([s,t]\times \R)\right)+\left(1+\frac{1}{t-s}\right)C\varepsilon\,.
\eas
By the arbitrariness of $\varepsilon>0$, the last relations must hold also without the last addend. As a consequence, by the arbitrariness of $B$, which we can choose concentrated, respectively, on the positive or on the negative part of the measures, thesis~\eqref{E:iOleinink} holds, just exchanges the names of $s$ and $t$ in.
\end{proof}

\subsection{Approximate balance estimates on characteristic regions}
We now informally outline the origin of the approximate Oleinink-type estimates~\eqref{E:iOleininknu}, that we rigorously prove later in \S~\ref{S:balances-}, \ref{Ss:decay-}, \ref{Ss:decay+}.
The first step consists of studying balances for $ \wm_{i}^{\nu,\cont}$ on characteristic regions, that we now introduce, explaining also what we mean by approximate balance estimates on characteristic regions, namely~\eqref{E:balancesnu} below.

Denote by $x(t;t_{0},x_{0})$ the leftmost, maximal with respect to inclusion, $i$-th characteristics of $u^{\nu}$ starting at $(t_0,x_{0})$:
\begin{subequations}\label{E:ab}
\ba
&x(t;t_{0},x_{0})=\min\Bigg\{\curva(t)\ \bigg|\ \begin{split}\curva:[0,T]\to\R\text{ satisfies }\curva(t_{0})=x_{0}\text{ and }\\\dot \curva(t)\in\left[\lambda_{i}(u^{\nu}(t,\curva(t)+)),\lambda_{i}(u^{\nu}(t,\curva(t)-))\right]\end{split}\Bigg\}\ .
\ea
Given $I=[a,b]$, let $a(t)$ and $b(t)$ be such $i$-th characteristics starting respectively at $(t_0,a)$ and $(t_0,b)$: namely 
\ba
&a(t)=x(t;t_{0},a)\ ,\qquad
b(t)=x(t;t_{0},b)\ .
\ea
\end{subequations}
We define the `characteristic region' $A^{\nu,t_0,t_0+\tau}_{[a,b]}$ within the $t$-strip $\{t_0<t\leq t_0+\tau\}$ delimited by $a(t)$ and $b(t)$: namely
\begin{subequations}
\label{E:iCharacteristicRegion}
\ba
A^{\nu,t_0,t_0+\tau}_{[a,b]}\doteq{}\left\{(t,x)\ :\ t_0<t\leq t_{0}+\tau\,,\ a(t)\leq x\leq b(t)\right\}\,.
\ea

Denote by $I(t)$ the fixed-time $x$-section of $A^{\nu,t_0,t_0+\tau}_{[a,b]}$, which is $I(t)\doteq{}[a(t),b(t)]$.
If now $J\doteq{}I_1\cup\dots\cup I_K$ is the union of the disjoint closed intervals $\{I_k\}_{k=1}^K$, we set the following notation for the union of the evolved regions:
\begin{equation}\label{E:ggrgrgggqrwgwrqgrwqgrw}
\begin{split}
&J(t)\doteq{}I_1(t)\cup\dots\cup I_K(t)\,,
\\
&A^{\nu,t_0,t_0+\tau}_{J}\doteq{}A^{\nu,t_0,t_0+\tau}_{I_1}\cup\dots\cup A^{\nu,t_0,t_0+\tau}_{I_K}\;.
\end{split}
\end{equation}
\end{subequations}
On such characteristic regions, in Lemma~\ref{L:balance0} of \S~\ref{S:balances-}, based on tools precisely defined in \S~\ref{S:definitionMeasures}, we prove the balance~\eqref{E:balancesnu} below: there exists a constant $C$, depending on the system and on the smallness of the initial data but not on the other parameters of the approximation, such that
\ba
\label{E:balancesnu} 
 -C \left[\left(\mu ^{ICS}_{\nu} \right)\left(A_{J}^{\nu,s,t}\right)+ K\beta_\nu\right]&\leq
\left[ \wm_{i}^{\nu,\cont}(t)\right](J(t))-
\left[ \wm_{i}^{\nu,\cont}(s)\right](J)
\\
&
\leq C \left[\left(\mu_{\nu} ^{ICJS} \right)\left(A_{J}^{\nu, s,t}\right)+ \beta_\nu\right]\;.\notag
\ea
The key tools for the approximate Oleinik-type estimate~\eqref{E:iOleininknu} above are
\begin{itemize}
\item estimate~\eqref{E:balancesnu} and
\item the estimate of the size of $t$-sections of characteristic regions, while they evolve. This is done by a variation of a classical lemma for ODEs applied to the speed of the boundary of \emph{characteristic} regions.
\end{itemize}
Section~\eqref{Ss:decay-} proves~\eqref{E:E:iOleininknu-} relying on such estimate of the size of \emph{forward-in-time} evolution $J(t)$ of union of intervals, related to $\left[ \wm_{i}^{\nu,\cont}(t)\right](J(t))$, jointly with the \emph{upper} estimate in~\eqref{E:balancesnu}
Section~\eqref{Ss:decay+} proves~\eqref{E:iOleininknu+} relying on such estimate of the size of \emph{backward-in-time} evolution $J(t)$ of union of intervals, related to $\left[ \wm_{i}^{\nu,\cont}(t)\right](J(t))$, jointly with the \emph{lower} estimate in~\eqref{E:balancesnu}.

\section{\texorpdfstring{$SBV$}{SBV}-like regularity}\label{S:SBVlike}
This section deals with the case of system~\eqref{eq:sysnc} when $A\in C^2(\Omega)$ is strictly hyperbolic as in~\eqref{eq:strhyp}.
When we admit linear degeneracy, as in Definition~\ref{D:LD}, there is no hope to regularize a $\BV$ initial datum to an $\SBV$ function. One could instead hope that $\lambda_i$ is regularized, but it turns out that this is not the case, see Example~\ref{Ex:counterBVlike} reported from~\cite{BYTrieste}.
Where is the problem? Exploiting the decomposition of $D_x u$ into wave-measures in Definition~\ref{D:upsiloni}, we compute
\begin{equation}\label{e:volpert}\begin{split}
D_{x}\lambda_{i}(u(t,x))&=\nabla\lambda_i(\widetilde u(t,x)) D_{x}u(t,x)
\\
&=\sum_{k=1}^N\left(\nabla\lambda_i(\widetilde u(t,x))\cdot r_k(\widetilde u(t,x))\wm_k \right)\ .
\end{split}
\end{equation}
We are able to say that the Cantor part of the $i$-th component \[\nabla\lambda_i(\widetilde u(t,x))\cdot r_i(\widetilde u(t,x))\wm_i \] vanishes, but such regularity is in general false for the other components, see the explicit Example~\ref{Ex:counterBVlike} below.

\begin{theorem}
\label{T:sbvhyp}
Let $u\in\BV_{\loc}([0,T]\times\R)$ be an entropy solution of the strictly hyperbolic system of balance laws~\eqref{eq:sysnc} with a locally small $\BV$ norm.
Then there exists a $\sigma$-compact set $ K$ such that $u\in\SBV([0,T]\times\R\setminus K)$ and there is linear degeneracy at each point $(t,x)\in K$: namely, $\nabla\lambda_i(u(t,x ))\cdot r_i(u(t,x))=0$ whenever $(t,x)$ is a continuity point of $u(t,\cdot)$ lying in the support of $D^{\Cantor}\wm_i$, for $i=1,\dots,N$.
More precisely, there exists an at most countable set $S$ of times such that for $t\in[0,T]\setminus S$, $i=1,\dots,N$, the scalar measure $\left[D^{\Cantor}_{x}\lambda_{i}(u)\right]_{i}$, namely the Cantor part of the $i$-component of $D_{x}\lambda_{i}(u(t,\cdot))$ in Definition~\ref{D:icomponent}, vanishes.
\end{theorem}

\begin{example}
Considering the degenerate equation $u_t=0$ one immediately realizes that the set $K$ in Theorem~\ref{T:sbvhyp} cannot be taken with $0$-Lebesgue measure in general, if closed. Indeed, if at time $t=0$ the function $u$ has a Cantor part concentrated on a $\sigma$-compact, negligible set $K_0$ dense in $[0,1]$ of course the set $\clos(K_{0})$ has full measure.
\end{example}

\nomenclature{$[D_{x}\lambda_{i}(u)]_{i}$}{The $i$-component of $D_{x}\lambda_{i}(u)$ in Definition~\ref{D:icomponent}}
\begin{definition}[$i$-component of $D_{x}\lambda_{i}(u)$]
\label{D:icomponent}
Let $u\in\BV([0,T]\times\R)$ be an entropy solution of the balance laws~\eqref{eq:sysnc}.
We define the \emph{$i$-component of $D_{x}\lambda_{i}(u)$} as
\[
\left[D_{x}\lambda_{i}(u)\right]_{i}\doteq
\left(\nabla\lambda_{i}(u)\cdot  r_{i}(u)\right) \wm_{i}^{\cont}
+\left[\lambda_{i}(u^{+})-\lambda_{i}(u^{-})\right]
	 \frac{|\wm_{i}^{\jump}|}{\sum_{k=1}^{N}|\wm_{k}^{\jump}|}\,,
\] 
while the \emph{continuous} and the \emph{Cantor part of the $i$-component of $D_{x}\lambda_{i}(u)$} as
\bas
\left[D^{\cont}_{x}\lambda_{i}(u)\right]_{i}&\doteq\left(\nabla\lambda_{i}(u)\cdot  r_{i}(u)\right) \wm_{i}^{\cont}
\\
&\equiv\left(\nabla\lambda_{i}(u)\cdot r_{i}(u)\right) \left((D_x^{\cont}u ) \cdot{ \ell}_{i} (u)\right) \,,
\\
\left[D^{\Cantor}_{x}\lambda_{i}(u)\right]_{i}&\doteq\left(\nabla\lambda_{i}(u)\cdot  r_{i}(u)\right) \wm_{i}^{\Cantor}
\\
&\equiv\left(\nabla\lambda_{i}(u)\cdot r_{i}(u)\right) \left((D_x^{\Cantor}u ) \cdot{ \ell}_{i} (u)\right) 
 \ .
\eas
\end{definition}
Notice that $\left[D^{\Cantor}_{x}\lambda_{i}(u)\right]_{i}=\left(\left[D_{x}\lambda_{i}(u)\right]_{i}\right)^{\Cantor}$, and the same holds for the continuous part of the measure.

\begin{proof}[Proof of Theorem~\ref{T:sbvhyp}]
As in~\cite{Robyr,BYTrieste}, the argument is a reduction argument to the case of genuine nonlinearity of the $i$-th field.
This is based on the finite speed of propagation and on the structure of $\BV_\loc$ functions.

By the structure of $\BV$ functions (\cite[Theorems 3.107-108]{AFPBook}), the following disintegration holds:
\[
\left[D_{x}^{\Cantor}u\right] =\int_{0}^{T}\left[D_{x}^{\Cantor}u(s)\right]ds\ .
\]
The first part of the statement thus follows once we prove that $D_{x}^{\Cantor}u(t)$ vanishes on the $t$-section of $K$ for $t\in[0,T]\setminus S$.

\firststep
\step{Decomposition of $[0,T]\times\R$}
Let us introduce the jump set $J_{\tau}$, the set where $D_{r_i}\lambda_i$ vanishes, at time $\tau$ and in $\R^2$, as their union: precisely, the sets
\begin{align*}
J_{\tau}&\doteq{}\{x\ :\ u(\tau,x^{-})\neq u(\tau,x^{+})\}\,,
\\
F_{\tau}&\doteq{}\{x\notin J_{\tau}\ :\ \nabla\lambda_i(u(\tau,x))\cdot r_i(u(\tau,x))=0\}\,,
\\
C_{\tau}&\doteq{}J_{\tau}\bigcup F_{\tau}\,,
\end{align*}
and $ J\doteq{}\{(\tau,x)\subset[0,T]\times\R\ :\ x\in J_{\tau}\}$, $F\doteq{}\{(\tau,x)\subset[0,T]\times\R\ :\ x\in F_{\tau}\}$, $C\doteq{}J \bigcup F$.
Being $u$ of bounded variation, the set $J_{\tau}$ is at most countable.
In particular $\mu(J_{\tau})=0$ for every measure $\mu$ without atoms, so that $ \abs{\wm_{i}^{\cont}}(J_{\tau})=0$.
Moreover $\nabla\lambda_i(u(\tau,x))\cdot r_i(u(\tau,x))$ vanishes on $F_{\tau}$ by definition, obtaining
\begin{equation}
\label{E:snfwknjknr}
\abs{\left[D_{x}^\cont\lambda_{i}(u(\tau))\right]_{i}}(C_{\tau})
\leq\int_{J_{\tau}\bigcup F_{\tau}}\!\!\!\!\!\!\!\!\!\!\!\!\!\!\!\left\{\abs{\nabla\lambda_i\cdot r_i} \,\abs{\ell_i}\right\}\Big|_{u(\tau,x)}\,\abs{\wm_{i}^{\cont}(\tau,dx)}
=0\,.
\end{equation}
We thus need to study $\left[D_{x}^\Cantor\lambda_{i}(u)\right]_{i}$ only on the complementary of $C$, as on $C$ it identically vanishes.

\step{Triangles of genuine nonlinearity}
For any $(t_0,x_0)\notin C$ the sign $\mathfrak s_{0}=\pm1$ of $\nabla\lambda_i(u(t_0,x_{0} ))\cdot r_i(u(t_0,x_{0}))$ is well defined: being $u(\tau,x)$ continuous in the $x$-variable at any point $x_0\notin J_{\tau}$, there exist $b_0=b_0(t_0,x_0)>0$ and $c_0=c_0(t_0,x_0)>0$ such that
\[
\nabla\lambda_i(u(t_0,x ))\cdot r_i(u(t_0,x))
\cdot \mathfrak s_{0}\geq c_0>0
\qquad \text{for }\abs{x-x_0}\leq b_0\ ,\ x\notin J_{\tau}\ .
\]
Let $\bar\eta>0$. By the Tame Oscillation condition, see~\cite[Lemma 2.3]{Ch2} jointly with~\cite[Theorem 4.4]{ACM1}, for any $(t_0,x_0)\notin C$ there is consequently a triangle
\begin{equation}
\label{E:triangle}
T(t_0,x_0)\doteq{}\left\{(\tau,x)\subset[t_{0},t_0+\widetilde b_0/\bar\eta ]\times\R\ :\ \abs{x-x_0}\leq \widetilde b_0-(\tau-t_0)\bar\eta\right\}
\end{equation}
whose basis $\abs{x-x_0}\leq \widetilde b_0$ is included in $\abs{x-x_0}\leq b_0$ and such that
\begin{equation}
\label{E:sjjjejwjwww}
{\nabla\lambda_i(z)\cdot r_i(z)} \cdot \mathfrak s_{0}\geq \frac{c_0}{2}>0
\qquad \text{for }z\in u\left(T_{0}\right)\ .
\end{equation}
We will prove that the Cantor part of $u$ vanishes on any one of such triangles, thus on countable unions of them.
We consider $T(t_0,x_0)\subset[0,T]\times\R$.

\begin{figure}\centering
\includegraphics[width=.6\linewidth]{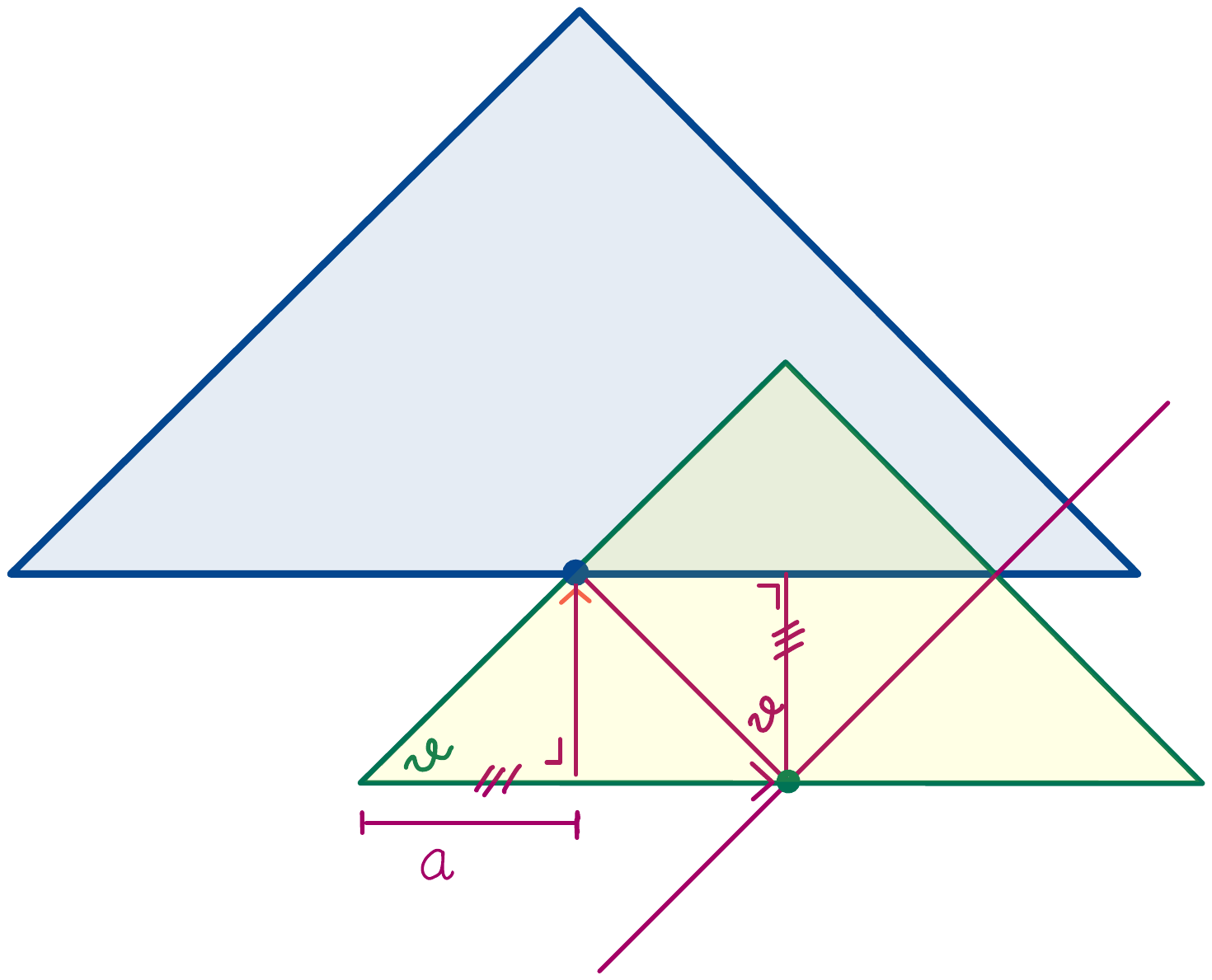}
\label{fig:triangoli}
\caption{Triangles in $F_{a,A}$ in the proof of Theorem~\ref{T:sbvhyp}.}
\end{figure}

\step{Countable covering}
We now show that we can cover the complementary of $C$ by a countable collection of the interior of such triangles, up to a remaining set $N\times\R$ having at most countable projections on the $t$-axis.
Denoting by $\ri{}$ the relative interior of a set, define the open set
\[
B=\left\{(t,x)\in[0,T]\times\R\  | \  \exists (t_0,x_0)\in[0,T]\times\R \ :\ (t,x)\in\ri\big(T(t_0,x_0)\big)\right\}\,.
\]
Of course, since $B$ is open, it can be covered by a countable collection of open triangles $\{\text{int}(T_j)\}_{j\in\mathbb{N}}$, the interiors of triangles $T_j=T(x_j,y_j)$, for $j\in\mathbb{N}$. Points outside of $B$, naturally, do not belong to the relative interior of any triangle \eqref{E:triangle}, whatever point $(t_{0},x_{0})\notin C$ we choose; if it is outside of $C$, it could still serve as the 'center' of such triangles. We now assert that $[0,T]\times\mathbb{R}\setminus(B\cup C)$ has at most a countable projection onto the $t$-axis, and we show it.

Let indeed $a>0$ and consider the set of points
\[
F_{a}=\left\{(\tau,\xi)\notin C \  | \  b_0(\tau,\xi)\geq 2a\ \wedge\  (\tau,\xi)\notin \ri\big(T(t,x)\big)\ \text{if $ b_0(t,x)\geq a $} \right\}\ .
\]
By properties of triangles isosceles, with equal angles $\arctan \bar \eta$, two points $(\tau_1,\xi_1)$ and $(\tau_2,\xi_2)$ belonging to $F_{a}$ with $\tau_1\neq \tau_2$ must have distance at least $a\sqrt{1+\bar\eta^{2}}$ just because both $(\tau_1,\xi_1)\notin T(\tau_2,\xi_2) $ and $(\tau_2,\xi_2)\notin T(\tau_1,\xi_1) $.
See Fig.~\ref{fig:triangoli}.
In particular, each set $F_{a}$ has finite projection on the $t$-axis. Of course this proves the claim being
\[
[0,T]\times\R\setminus\left(B\cup C\right) \subseteq \bigcup_{n\in\N} F_{2^{-n}}\,.
\]
We thus get $N\subset[0,T]$ at most countable and $B=\bigcup_{j\in\N}T_j$ that satisfy
\[
[0,T]\times\R=B\bigcup C\bigcup N\times\R
\,.
\]

\step{Thesis on each triangle} By the finite speed of propagation, in any fixed triangle $T_j$ like in~\eqref{E:triangle} the function $u$ is equal to the solution to
\[
\begin{cases}
 w_t+f(w)_x =g(t,x,w)\,,\\
 w(0,x) =
 \begin{cases}
 u(t_0,x) &\abs{x-x_0}\leq \widetilde b_0\,, \\
 \frac{1}{2\widetilde b_0}\int_{x_0-\widetilde b_0}^{x_0+\widetilde b_0} u(t_0,y)dy &\abs{x-x_0}> \widetilde b_0\,.
 \end{cases}
\end{cases}
\]
By the Tame Oscillation condition, see~\cite[Lemma 2.3]{Ch2} jointly with~\cite[Theorem 4.4]{ACM1}, the inequality~\eqref{E:sjjjejwjwww} holds also for the whole range of $w$: in particular, $D_{x}u(t,\cdot)\cdot\widetilde r_i(u(t,\cdot)$ is $\SBV_\loc$ outside an at most countable set $N_j$ of times by Theorem~\ref{T:SBVGNbalance}.
As a consequence, $\wm_i^\Cantor(\tau)$ vanishes on $T_j$ if $\tau\notin N_j$, so that
\begin{equation}
\label{E:snfwknjknr2}
\abs{\left[D_{x}^\Cantor\lambda_{i}(u(\tau))\right]_{i}}(T)
\leq\int_{T_{\tau}}\abs{\nabla\lambda_i\cdot r_i}\Big|_{u(\tau,x^{-})} \,\abs{\wm_{i}^{\Cantor}(\tau,dx)}=0
\end{equation}
whenever $\tau\notin N_j$.

\step{Conclusion} 
Whenever $\tau\notin N$ the complementary of $C_\tau$ is covered by $\{T_j\}_{j\in\N}$, so that $\abs{\left[D_{x}^\Cantor\lambda_{i}(u(\tau))\right]_{i}}(\R)$ is controlled by
\bas
&\abs{\left[D_{x}^\cont\lambda_{i}(u)\right]_{i}}(C_{\tau})
+\abs{\left[D_{x}^\Cantor\lambda_{i}(u)\right]_{i}}\left(\bigcup_{j\in\N} T_j\right)\\
&\stackrel{\eqref{E:snfwknjknr},\eqref{E:snfwknjknr2}}{=}0
\qquad\text{if }\tau\notin N\cup\bigcup_{j\in N} N_j\ .
\eas
We also conclude that the closed set $K^{*}=[0,T]\times\R\setminus\cup_{j\in\N}\{T_j\}\subset C$ almost satisfies the first part of the statement: by inner and outer regularity of the measures $D^{\Cantor}\wm_i$, $i=1,\dots,N$, we can pick-up a $\sigma$-compact subset $K$ of $K^{*}$ such that both $D_x^\jump(K)=0$ and
\[
\sum_{i=1}^{n}D^{\Cantor}\wm_i(\R^{2}\setminus K)=\sum_{i=1}^{n}D^{\Cantor}\wm_i(\cup_{j\in\N}\{T_j\})=0\,.
\]
\end{proof}


For a counterexample concerning that $\SBV$-regularity in this context might fail for $\lambda_{i}(u)$, as well as for $u$, we refer to Remark 7.2 in~\cite{BYTrieste}, which we repeat here for completeness. Of course, it stays true as well if we add a suitable source term.

\begin{example}[\cite{BYTrieste}]
\label{Ex:counterBVlike}
We present a $2\times2$ system of strictly hyperbolic conservation laws with the first field linearly degenerate, and the second field genuinely nonlinear. We show that some initial data $\overline U$ with arbitrarily small bounded variation are not regularized to become special functions of bounded variation.

Furthermore, explicitly computing the space derivative of the second eigenvalue $\lambda_2(U)$, we notice that it does indeed possess a Cantor part.
More precisely, in accordance with Theorem~\ref{T:sbvhyp} and the decomposition in~\eqref{e:volpert}, the Cantor part of $D_x U\cdot \ell_2(U)$ is present only where $\nabla\lambda_2(U)\cdot r_2(U)$ vanishes, while there is a Cantor part in $\left(\nabla\lambda_{2}(U)\cdot r_{1}(U)\right) \left((D_x^{ }U ) \cdot{ \ell}_{1} (U)\right)$.

This example also shows that a Cantor part in one of the component might create instantaneously a Cantor part in the other component.

Set $U=\binom{u}{v}$ and consider
\[
\begin{cases}
u_t=0\\
\wm_t+\left((1+v+u)v\right)_x=0
\end{cases}\,,
\quad 
A(U)=Jf(U)=\begin{pmatrix}0 & 0\\ v &1+2v+u\end{pmatrix}\,.
\]
We include details of its elementary analysis for clarity:
\begin{itemize}
\item Eigenvalues of $A$:
\begin{align*}
&\lambda_1(u,v)=0\,,
&&
\lambda_2(u,v)=1+2v+u\,,
\\&
\nabla\lambda_1(u,v)=\begin{pmatrix} 0\\ 0\end{pmatrix}\,,
&&
\nabla\lambda_2(u,v)=\begin{pmatrix} 1 \\ 2\end{pmatrix}\,.
\end{align*}
\item Eigenvectors of $A$:
\begin{align*}
&\ell_1(u,v)=\binom{1}{0}\,,
&&\ell_2(u,v)=\frac{1}{(1+2v+u)^2+v^2}\binom{v}{1+2v+u}\,,
\\
& r_2(u,v)=\begin{pmatrix} 0 \\ 1\end{pmatrix}\,,
&&r_1(u,v)=\frac{1}{(1+2v+u)^2+v^2}\begin{pmatrix} 1+2v+u \\ -v\end{pmatrix}\,.
\end{align*}
\item Derivatives composed with a function $U=(u,v)$ with bounded variation:
\begin{align*}
&D_x^\Cantor\lambda_1(U)=0=\left[D^{\Cantor}_{x}\lambda_{1}(U)\right]_{1}=\left[D^{\Cantor}_{x}\lambda_{1}(U)\right]_{2}
\\
&
D_x^\Cantor\lambda_2(U)=1+2D_x^\Cantor v+D_x^\Cantor u
\end{align*}
which has components according to Definition~\ref{D:icomponent} given by
\begin{align*}
\left(\nabla\lambda_{2}(U) \cdot r_{1}(U) \right)&\left(\ell_ 1(U) \cdot D_x^\Cantor U\right) 
\equiv \frac{1+u}{(1+2v+u)^2+v^2}D^{\Cantor}_{x}u
\\
\left[D^{\Cantor}_{x}\lambda_{2}(U)\right]_{2}\doteq{}&\left(\nabla\lambda_{2}(U)\cdot r_{2}(U)\right)\left(\ell_ 2(U)\cdot D_x^\Cantor(U)\right) \\
\equiv& \frac{2}{(1+2v+u)^2+v^2}\left(vD_x^\Cantor u+(1+2v+u)D_x^\Cantor v\right)
\end{align*}
\end{itemize}

Being a triangular system, with trivial first equation, the first component $u(t)\equiv\overline u$, of the solution $U=(u,v)$ to the Cauchy problem, obviously has a Cantor part for every $t>0$ whenever a Cantor part is present in the initial datum $\overline U=(\overline u,\overline v)$. 
In particular, $U$ and $\lambda_2(U)$ are not special functions of bounded variation.
This is consistent with the statement: the $1$-component $\frac{1+u}{(1+2v+u)^2+v^2}D^{\Cantor}_{x}u$ of $D_x\lambda_2(U)$ must not necessarily vanish, nor $D_x^\Cantor\lambda_2(u,v)$, which indeed does not vanish if $u(0)$ has a Cantor part.

 Of course if a Cantor part is present at time $t=0$ in $D_x\overline u$ it will be present also at all future times\dots  Nevertheless, the $\SBV$-like regularity theorem guarantees that $vD_x^\Cantor u+(1+2v+u)D_x^\Cantor v$ vanishes. 
 
 This is true but a bit counterintuitive to us: let $\overline u $ be any continuous function, in $\BV(\R)$, compactly supported in $[0,2]$, with $D^\Cantor \overline u\neq0$, and let $v$ be $\varepsilon$ times the indicator function of $[-10,10]$. As $u(t,x)=\overline u(t)$, then $v$ immediately develops a Cantor part, at any positive time, to balance $vD_x^\Cantor u$.
\end{example}

\section{\texorpdfstring{$SBV$}{SBV}-regularity for non-degenerate fluxes}
\label{S:SBVND}
The case of a genuinely nonlinear field of Theorem~\ref{T:SBVGNbalance} might seem restrictive. Nevertheless, it is not merely the key step in order to get the $\SBV$ regularity of the entropy solutions when all fields are genuinely nonlinear as in~\cite{Lax}, but also when they are piecewise genuinely nonlinear as in~\cite{IguchiLeFloch}, and more generally when they satisfy the non-degeneracy condition introduced in~\cite{LeFlochGlass}, slightly different from the one we consider here below.

\emph{It is relevant to extend the $\SBV$-regularity including the case of this last non-degeneracy condition because generic fluxes do satisfy such condition, so that non-degenerate fluxes approximate well any flux, see Remark~\ref{R:genericity} below.}

\begin{definition}\label{def:ND}
Let $\lambda_{1}\,,\dots\,,\lambda_{N}$ be the eigenvalues of a matrix $A\in C^{M}(\Omega)$, where $M\in\N\cup\{+\infty\}$ and $M\geq 2$.
We consider $A$ \emph{non-degenerate} if for all $i=1,\dots,N$, defining $ \pi_{i}^{(1)}=\nabla \lambda_{i}\cdot r_{i}$ and $ \pi_{i}^{(k+1)}= \nabla  \pi_{i}^{(k)}\cdot r_{i}$, there holds
\ba\label{eq:Anondeg}
&\left\{\pi_{i}^{(k)} \right\}_{k=1,\dots,M}\neq 0^{M}
\,.\ea
\end{definition}

In case of a scalar equation, when $N=1$, then $\pi_{1}^{(1)}=f''$ and $\pi_{1}^{(k)}=f^{(k+1)}$: a flux is called non-degenerate for us if there is a derivative higher than the first which is non-vanishing.
In~\cite{LeFlochGlass} $f$ was considered piecewise genuinely nonlinear if $f''$ and $f'''$ do not vanish simultaneously, and in~\cite{IguchiLeFloch} the same condition denoted nongenerate fluxes.

\begin{remark}\label{R:genericity}
When all fields are genuinely nonlinear, then $ \pi_{i}^{(1)}\neq 0$ for $i=1,\dots, N$ and this nondegeneracy condition holds. 
Piecewise genuinely nonlinear fluxes of~\cite{IguchiLeFloch} satisfy $ \pi_{i}^{(1)}$, $ \pi_{i}^{(2)}$ do not vanish simultaneously, for $i=1,\dots,N$. In~\cite{LeFlochGlass} fluxes were called nondegenerate when $ \pi_{i}^{(1)}$, \dots, $ \pi_{i}^{(N+1)}$ do not vanish simultaneously, for $i=1,\dots,N$.
With the Whitney topology and when $M\geq N+1$, the set of functions $A\in C^{\infty}(\Omega)$ satisfying our non-degeneracy condition~\eqref{eq:Anondeg} contains the intersection of a countable number of open dense subsets by~\cite[\S~2.6]{LeFlochGlass}: in particular, see~\cite[Theorem 2.13]{LeFlochGlass}, the relative fluxes approximate any smooth flux in the Whitney topology.
\end{remark}

Before relating the non-degeneracy of the flux to the regularity of solutions, we better understand the condition.

\begin{lemma}\label{L:transversal}
Suppose $A$ is non-degenerate as in Definition~\ref{def:ND}. For each $i=1,\dots, N$ and any $R>0$, the set 
\[
Z_{i}\doteq{}\{u\in\Omega\quad:\quad \nabla\lambda_{i}(u)\cdot r_{i}(u)=0\}\cap[-R,R]^{N}
\]
is the union of finitely many disjoint smooth sub-manifolds $\{Z_{i}^{k}\}_{k=1}^{J_{i}}$ of dimension at most $N-1$, all transversal to $r_{i}$.
\end{lemma}

\begin{proof}
Since the non-degeneracy assumption grants that $\pi^{(1)}_{i}$, \dots, $\pi^{(M)}_{i}$ do not vanish simultaneously, the decomposition can be constructed applying the implicit function theorem to the sets
\[
\bigcap_{k=1}^{\ell}\{\pi^{(k)}_{i}=0\}\cap \{\pi^{(\ell+1)}_{i}\neq0\}\qquad \text{for }\ell=M-1,\dots ,1\,.
\]

Consider now any $u\in Z_{i}$ and consider any smooth curve $\underline \gamma:(-1,1)\to Z_{i}$ with $\underline \gamma(0)=u$.
In particular, for $|t|\leq1$ one has $ \pi^{(1)}_{i}(\underline \gamma(t))=0$ identically so that
\[
\ddt  \pi^{(k)}_{i}\circ\underline \gamma|_{t=0}=\nabla \pi^{(k)}_{i}(u)\cdot \dot{\underline \gamma}(0)=0
\qquad \forall k\in\N\,.
\] 
By the non-degeneracy assumption, necessarily $ \dot{\underline \gamma}(0)\neq r_{i}(u)$: being different from any direction tangent to $Z_{i}$, thus $r_{i}(u)$ must be transversal to $Z_{i}$.
\end{proof}

We now state the regularity result.

\begin{theorem}[$\SBV$ regularity for non-degenerate systems]
\label{C:SBVGNbalance1den}
Consider an entropy solution $u\in\BV_{}([0,T]\times\R;\Omega)$ of the strictly hyperbolic system of balance laws~\eqref{eq:sysnc} with small $\BV$ norm.
Suppose $g$ satisfies Assumption~\textbf{(G)}, Page~\pageref{Ass:G}, and $A$ is nondegenerate in $\Omega$ as in~\eqref{eq:Anondeg}. Then there exists an at most countable set $S\subset[0,T]$ of times such that  $x\mapsto u(t,x)$ is a special function of bounded variation for every $t\in[0,T]\setminus S$.
In particular, $u\in\SBV([0,T]\times\R)$.
\end{theorem}

Before entering the proof, we provide a general auxiliary lemma of real analysis.

\begin{lemma}\label{L:realAn}
Consider a function $w:\R\to \R^{N}$ with bounded variation and let $B\subset \R$ be a Borel set of continuity points of $w$ with $\left|D_{x}^{\Cantor}w\right|(B)>0$, $\Ll^1(B)=0$ and $w(B)\subset \{x_{1}=0\}$.
Then $\mathrm{e}_{1}\cdot D_{x}^{\Cantor}w\restriction_{B}=0$.
\end{lemma}

\begin{proof}
By the inner regularity of the Borel measure, it suffices to prove the lemma for all compact sets $K\subseteq B$. 
We can also assume that $\Ll^{1}(K)=0$, since the Cantor part of $D_{x}w$ is singular with respect to the Lebesgue measure.

Notice that only the Cantor part of $Du$ is concentrated on $K$, while the jump part vanishes because $K$ is made of continuity points and the absolutely continuous part vanishes because it is Lebesgue negligible.

Let $\varepsilon>0$.
Cover the compact set $K$ with open intervals $(a_{1},b_{1})$, \dots, $(a_{J},b_{J})$ with $b_{i}<a_{k}$ if $i<k$ and such that 
\begin{equation}\label{E:approxCantor}
\sum_{i=1}^{J}(b_{i}-a_{i})\leq \varepsilon\,,
\qquad
\left|D_{x}^{}w\right|\left(\bigcup_{i=1}^{J}(a_{i},b_{i})\setminus K\right)<\varepsilon\,.
\end{equation}
If $\overline a_{i}$ is the closest point to $a_{i}$ in $K\cap(a_{i},b_{i})$, $\overline b_{i}$ the closest to $b_{i}$, then $K\subset \cup_{i=1}^{J}[\overline a_{i},\overline b_{i}]$ and $\sum_{i=1}^{J}(\overline b_{i}-\overline a_{i})\leq \varepsilon$.
Having a compact set, the extremes $a_1$ and $b_J$ are fixed and do not depend on $\varepsilon$.
Since by hypothesis $w(K)\subset \{x_{1}=0\}$, in particular $w(\overline a_{i})\in \{x_{1}=0\}$ and $w(\overline b_{i})\in \{x_{1}=0\}$ for $i=1,\dots,J$.

Consider in $[\overline a_{1},\overline b_{J}]$ the auxiliary function
\[
w_{\varepsilon}
= w- \int_{\overline a_{1}}^{x} D_{x}w\restriction_{\cup_{i=1}^{J-1}[\overline b_{i},\overline a_{i+1}]  }
=w(\overline a_{1})+  \int_{\overline a_{1}}^{x} D_{x}w\restriction_{\cup_{i=1}^{J}[\overline a_{i},\overline b_{i}] } \,.
\]
In particular
\begin{equation}\label{E:ggrgrgr}
D_{x}w_{\varepsilon}\restriction_{\cup_{i=1}^{J}[\overline a_{i},\overline b_{i}]}=
D_{x}w_{}\restriction_{\cup_{i=1}^{J}[\overline a_{i},\overline b_{i}]}
\quad\text{and}\quad
D_{x}w_{\varepsilon}\restriction_{\cup_{i=1}^{J-1}[\overline b_{i},\overline a_{i+1}]}=0\,.
\end{equation}

Since $w_{ }(\overline a_{i}),w_{ }(\overline b_{i})\in\{x_{1}=0\}$, then for $i=1,\dots,J$ one has
\begin{equation*}
0=(w_{}(\overline b_{i})-w_{ }(\overline a_{i}))\cdot \mathrm e_{1}=\left(\int_{\overline a_{i}}^{{\overline b_{i}}}Dw_{}(x)\,dx\right) \cdot \mathrm e_{1}\,.
\end{equation*}
In particular,
\begin{equation}\label{E:proiezx1}
(w_{\varepsilon}(\overline b_{i})-w_{ \varepsilon}(\overline a_{i}))\cdot \mathrm e_{1}=0\,.
\end{equation}
Taking into account that $w_{\varepsilon}(\overline a_{1})=w_{}(\overline a_{1})\in\{x_{1}=0\}$, we have $w_{\varepsilon}(\overline b_{1})\in\{x_{1}=0\}$.

If $ \overline b_{i}\leq x\leq \overline a_{i+1}$, then $w_{\varepsilon}(x)=w_{\varepsilon}(\overline b_{i})$ so that $w_{\varepsilon}(x)\in\{x_{1}=0\}$ in $[\overline b_{1},\overline a_{2}]$.
Continuing iteratively with~\eqref{E:proiezx1}, this yields $w_{\varepsilon}(x)\in\{x_{1}=0\}$ in each $[ \overline b_{i}, \overline a_{i+1}]$, $i=1,\dots J-1$.

If $x\in( \overline a_{i}, \overline b_{i})\setminus K$ then $w_{\varepsilon}(x)\in\{|x_{1}|\leq \varepsilon\}$ because
\[
\abs{w_{\varepsilon}(x)-w_{\varepsilon}(\overline a_{i})}\leq  \abs{D_{x}^{}w_\varepsilon}\left( (a_{i},b_{i})\setminus K\right)\stackrel{\eqref{E:ggrgrgr}}{=} \abs{D_{x}^{}w}\left( (a_{i},b_{i})\setminus K\right)\stackrel{{\eqref{E:approxCantor}}}{\leq} \varepsilon\,.
\]

As $\varepsilon\downarrow 0$ we can require $I_{\varepsilon}=\cup_{i=1}^{J-1}[\overline b_{i},\overline a_{i+1}]$ increasing and $w_{\varepsilon}$ converges to $\widehat w(x)=w(\overline a_{1})+\int_{a_{1}}^{x}D^{\Cantor}_{x}w\restriction_{K}$, thus valued in $\{x_{1}=0\}$, the convergence is also uniform.
Since $\widehat w\cdot \mathrm e_{1}=0$ identically, we thus arrive at the thesis
\[
0=D_{x}(\widehat w\cdot \mathrm e_{1})=(D_{x}^{}\widehat w)\cdot \mathrm e_{1}=(D_{x}^{\Cantor}  w)\cdot \mathrm e_{1}\,.
\qedhere
\]
\end{proof}

\begin{corollary}\label{C:cantorPart}
Let $Z\subset \R^{N}$ be a smooth manifold with codimension at least $1$ and with any normal direction $\hat n$.
Consider a function $w:\R\to \R^{N}$ with bounded variation and let $K\subset \R$ be a $\sigma$-compact set of continuity points of $w$ with  $\left|D_{x}^{\Cantor}w\right|(K)>0$ and $w(K)\subset Z$.
Then $\hat n\cdot D_{x}^{\Cantor}w\restriction_{K}=0$, which implies, applying this repeatedly with different $\hat n$, that $D_{x}^{\Cantor}w\restriction_{K}$ is tangent to $Z$.
\end{corollary}

\begin{proof}
If the dimension of $Z$ is less than $N-1$, just extend $Z$ to a larger one which still has $\hat n$ as a normal direction: we thus consider only the case of dimension $N-1$.

Focus on a small neighborhood of a point, for example, of the origin.
Once parametrized there the surface smoothly as $\{\sigma(x_{2},\dots ,x_{N})\ :\ (x_{2},\dots ,x_{N})\in B_{\delta}(0)\}$ with $\sigma$ one-to-one, consider the function
\[
\widehat w=\sigma^{-1}\circ w: \R\to\{x_{1}=0\}\subset\R^{N}\,.
\]
By Volpert chain rule, $\widehat w\in \BV(\R;\R^{N})$ and $D_{x}^{\Cantor}\widehat w=J\sigma^{-1}(w)\cdot D_{x}^{\Cantor}w$, equivalent to 
\begin{equation}\label{e:chgvar}
D_{x}^{\Cantor}w=J\sigma(\widehat w)\cdot D_{x}^{\Cantor}\widehat w\,.
\end{equation}

Since the tangent plane to $Z$ is generated by $\partial_{x_{2}}\sigma$, \dots,  $\partial_{x_{N}}\sigma$ and $D_{x}^{\Cantor}\widehat w$ is a linear combination of $ \mathrm e_{2}$, \dots, $ \mathrm e_{N}$ by Lemma~\ref{L:realAn}, we get from~\eqref{e:chgvar} that $D_{x}^{\Cantor}w$ is a linear combination of $\partial_{x_{2}}\sigma$, \dots,  $\partial_{x_{N}}\sigma$ thus tangent to $Z$.
\end{proof}

We now give an auxiliary lemma related to balance laws.

\begin{lemma}\label{L:bbebbeb}
According to the hypothesis of Theorem~\ref{C:SBVGNbalance1den}, set $w(x)=u(t,x)$ for some $t$.
Consider a compact set $K\subseteq \R$ of continuity points of $w$ with $\Ll^{1}(K)=0$.
Let $\{Z_{i}^{k}\}_{k=1}^{J_{i}}$  be as in Lemma~\ref{L:transversal}, $i=1,\dots,N$.
If there are $X\subseteq \{1,\dots,N\}$ and  $k_{i}\in\{1,\dots,J_{j}\}$, for $i\in X$, with $w(K)\subset  Z_{i}^{k_{i}}$ precisely if $i\in X$ and $k=k_{i}$, then $D_{x}^{\Cantor}w{\restriction_{K}}$ vanishes.
\end{lemma}

\begin{proof}[Proof of Lemma~\ref{L:bbebbeb}]
We can decompose $D_{x}^{\Cantor}u{\restriction_{K}}$ along the right eigenvectors $r_{i}$ as in Definition~\ref{D:upsiloni}, where we introduced the wave measures $\wm_i$.
Since $ \nabla\lambda_{i}(v(x))\cdot r_{i}(v(x))\neq 0$ for $i\notin X$ and $x\in K$, Theorem~\ref{T:sbvhyp} ensures that $D^{\Cantor}\wm_i\restriction_{K}$ vanishes when $i\notin X$. We can therefore write:
\begin{equation}\label{E:inri}
 D_{x}^{\Cantor}w{\restriction_{K}}=\sum_{i\in X} { \wm_i}r_{i}(w)
 \qquad
 \text{where ${ \wm_i}\doteq{} D_{x}^{\Cantor}w{\restriction_{K}}\cdot \ell_{i}(w)$.}
\end{equation}

If $K$ satisfies the hypothesis of Corollary~\ref{C:cantorPart}, which means if \[\left|D_{x}^{\Cantor}w\right|(K)>0\,,\] then $D_{x}^{\Cantor}w$ must be tangent to each $Z_{i}^{j_{i}}$, $i\in X$: by Lemma~\ref{L:transversal} it must be thus transversal to each $r_{i}$ for $i\in X$.

 By the very definition of $Z_{i}^{k_{i}}\subseteq \{\nabla\lambda_{i}\cdot r_{i}=0\}$, and by~\eqref{eq:norm}, the tangent plane to each $Z_{i}^{k_{i}}$ is contained in the linear span $\langle\{\ell_{k}\}_{k\neq i}\rangle$ of the vectors $\{\ell_{k}\}_{k\neq i}$. As a consequence, the tangent plane to $\cap_{i\in X}Z_{i}^{k_{i}}$ is contained in $\cap_{k\in X}\langle \{\ell_{i}\}_{i\neq k}\rangle=\langle \{\ell_{i}\}_{i\in{1,\dots,N}\setminus X}\rangle$.
As $D_{x}^{\Cantor}w$ is tangent to $\cap_{i\in X}Z_{i}^{k_{i}}$ by Corollary~\ref{C:cantorPart}, one can thus also write
\begin{equation}\label{E:intersectionZi}
 D_{x}^{\Cantor}w{\restriction_{K}}=\sum_{i\notin X}\widehat{ \wm_i}\ell_{i}(w)\,.
\end{equation}
Jointly with~\eqref{E:inri}, we conclude that $D_{x}^{\Cantor}w{\restriction_{K}}$ vanishes because $\{ r_{i}\}_{i\in X}$ and $\{ \ell_{i}\}_{i\notin X}$ are linearly independent: from
\begin{align*} 
& \ell_{k}(w)\cdot  D_{x}^{\Cantor}w{\restriction_{K}}
=\ell_{k}(w)\cdot\sum_{i\in X} { \wm_i}r_{i}(w)
=\sum_{i\in X}{ \wm_i} \,\ell_{k}(w)\cdot r_{i}(w) =0
&&\text{if }   k\notin X\,,
\\
& \ell_{k}(w)\cdot  D_{x}^{\Cantor}w{\restriction_{K}}
=\ell_{k}(w)\cdot\sum_{i\notin X} \widehat{ \wm_i}\ell_{i}(w)
=\sum_{i\notin X}\widehat{ \wm_i} \, \ell_{k}(w)\cdot \ell_{i}(w) 
&&\text{if }  k\notin X\,,
\end{align*}
 by the invertibility of the square matrix with elements $[  \ell_{k}(w)\cdot \ell_{i}(w)]_{i,k\notin X}$, indeed $ \widehat{ \wm_i}=0$ for $i\notin X$.
\end{proof}

%

\begin{proof}[Proof of Theorem~\ref{C:SBVGNbalance1den}]
In the proof of Theorem~\ref{T:sbvhyp}, the strip $[0,T]\times\R$ was covered with
\begin{itemize}
\item at most countably many exceptional time-lines $\{t_{k}\}\times\R$, $k\in \N$, $0\leq t_{k}\leq T$,
\item a jump set $J$, whose section $J_{\tau}=J\cap\{\tau\}\times\R$ is at most countable for all $0\leq \tau\leq T$, 
\item a countable union $B$ if open triangles where $u$ is a special function of bounded variation,
\item the inverse images $u^{-1} (Z_{i}^{j})$, $j=1,\dots,J_{i}$, of the manifolds of Lemma~\ref{L:transversal} transversal to $r_{i}$ where $\nabla\lambda_{i}\cdot r_{i}=0$.
\end{itemize}

Consider any time $\overline t\notin \{t_{k}\}_{k\in\N}$.
Consider any compact set $K\subseteq \R$ of continuity points of $w(x)\doteq u(\overline t,x)$ that satisfies $\Ll^{1}(K)=0$.
For each $X\subset \{1,\dots,N\}$ and any selection $\{j_{i}\}$ one has that $D_{x}^{\Cantor}w{\restriction_{K}}\cap u^{-1}(\{Z_{i}^{j_{i}}\}_{i\in X})$ vanishes by Lemma~\ref{L:bbebbeb} jointly with $\sigma$-additivity.
By $\sigma$-additivity and by inner regularity of the measures, the Cantor part of $D_{x}u(\overline t)$ must thus vanish.
Because of the slicing theory of $\BV$-functions~\eqref{item:slicingDx} then the Cantor part of $D_{x}u$ must thus vanish, as the Cantor part of the derivative of the time restrictions vanishes at almost every time.
\end{proof}

\appendix

\section{An account on the Cauchy problem}
\label{sec:PCA}

In this section we describe the main ingredients in order to construct a piecewise
constant approximation of a solution $u$ to Eqs.~\eqref{eq:sysnc}-\eqref{eq:inda}, first in the homogeneous case $g=0$ then in the non homogeneous case.

In \S~\ref{subsec:RP} we summarise the construction of the self-similar solution to conservation laws to the Riemann problem, which is the Cauchy problem~\eqref{eq:RP} below where the datum has a single jump in the origin. This is an essential building block for the general Cauchy problem~\eqref{eq:sysnc}.

 This allows then to describe in \S~\ref{Ss:frontTr} an algorithm for defining piecewise-constant approximations first to the Cauchy problem for systems conservation laws, in particular the front-tracing solution, then, pairing it with an operator splitting method, to the Cauchy problem for systems of balance laws.
 
 In turn, in \S~\ref{Ss:PreviousEstimates} we review the functionals that provide estimates needed for compactness, precisely the interaction estimates that also play an important role in the estimates to obtain $\SBV$ regularity.
 
  Finally, in \S~\ref{Ss:convergenceAndExistence} we recall the convergence result of the constructed approximation.
\subsection{The nonconservative Riemann problem}
\label{subsec:RP}
As in the Introduction, we let $A$ be a smooth matrix-valued function defined on an open domain
$\Omega\subseteq\real^N$ and satisfying the strict hyperbolicity Assumption~\eqref{eq:strhyp}.

Since we deal with a system that, in general, it is not in conservation form,
we briefly recall the construction of the solution to a Riemann problem in
the homogeneous case, i.e.
\begin{subequations}
\label{eq:RP}
\begin{align}
\label{eq:syshom}
&u_t+A(u)u_x=0\,,&\\
\noalign{\smallskip}
\label{eq:indaRP}
&u(0,x) =\begin{cases}
u^L\quad &\text{if\quad $x<0$\,,}\\
u^R\quad &\text{if\quad $x>0$\,.}
\end{cases}&
\end{align}
\end{subequations}
We refer to~\cite{BB, srp} for details. Here we just summarise the most relevant step, which is the construction of the admissible elementary curve of the $k$-th family for any given left state $u_L$, by quoting the following theorem from~\cite{BYu}.

\begin{theorem}
For every $u\in \Omega$ there exist $N$ Lipschitz continuous elementary curve $\tau\mapsto T_k[u^L] (\tau)$ and $N$ Lipschitz continuous velocities $(r,\, s)\mapsto\sigma_k[u^L] (r,\,s)$, $k=1,\dots,N$ satisfying $\sigma_k[u] (0,\,0)=\lambda_k(u)$ and
\[
\dds  T_k[u] (s) {\big|}_{s=0}=r_k(u) \,,\quad s\mapsto\sigma_k[u] (r,\,s)\text{ increasing,}
\]
and with the following property:  when $u^L\in\Omega$, $u^R= T_k[u^L] (s)$ for some $s$ sufficiently small, the unique vanishing viscosity solution $u$ of the Riemann problem~\eqref{eq:RP} is defined a.e.~by
\begin{equation}
 \label{eq:solRPi}
 u (t,x) =
 \begin{cases}
 u^L\quad &\text{if\quad $x/t <\sigma_k[u^L](s,\,0)\,,$}\\
 \noalign{\smallskip}
 T_k[u^L] (\tau)\quad &\text{if\quad $x/t = \sigma_k[u^L](s,\,\tau)$
 \quad for some $\tau\in\mathcal{I}$\,,}\\
 \noalign{\smallskip}
 u^R\quad &\text{if\quad $x/t >\sigma_k[u^L] (s,\,s)\,,$}
 \end{cases}
\end{equation}
where $\mathcal{I}=\{\tau\in [0,s]\ : \ \sigma_k[u^L](s,\tau)\neq \sigma_k[u^L](s,\tau')\ \forall \tau'\neq\tau \}$.

\end{theorem}

\begin{remark}
\label{rem2:consform}
{\rm
If the system (\ref{eq:sysnc}) is in conservation form, i.e., in the case
where $A(u)=DF(u)$ for some smooth flux function $F$, the general
solution of the Riemann problem provided by (\ref{eq:solRPi}) is a
composed wave of the $k$-th family containing a countable number of
rarefaction waves and contact discontinuities or compressive shocks
which satisfy the Liu admissibility condition~\cite{tplrp2p2, tplrpnpn}.
When the $k$-th family is genuinely nonlinear, if one denotes by $R_{k}[u]$ the $k$-th rarefaction curve and by $S_{k}[u]$ the $k$-th shock curve, see~\cite[\S\S~5.1-2]{BressanBook}, then the elementary curve can be written as
\[
 T_k[u] (s)=\begin{cases} R_{k}[u](s) &s\geq0 \,,\\ S_k[u](s) &s<0\,, \end{cases}
 \qquad
 \text{
where $s=\ell_i(u)\cdot \left(T_k[u] (s)-u\right)$.}
\]
}
\end{remark}
In view of the considerations of Remark~\ref{rem2:consform},
we will extend the standard terminology adopted for the
elementary waves that are present in the solution
of a hyperbolic system of conservation laws
to the general case of non-conservative systems.
Thus, 
we will say that any (vanishing viscosity) solution of the Riemann
problem for (\ref{eq:sysnc}) of the form (\ref{eq:solRPi}) is a {\it
centered rarefaction wave} of the $k$-th family whenever $u^R\in
R_k[u^L](s)$ for some wave size~$s$ such that $\tau\mapsto
\sigma_k[u^L](s,\tau)$ be strictly increasing on $[0,s]$, $s>0$ (or
strictly decreasing on $[s,0]$ if $s<0$), while we will say that any
(vanishing viscosity) solution of a Riemann problem for (\ref{eq:sysnc}) of
the form
$$
u(t,x)=\begin{cases}
u^L\quad &\text{if\quad $x<\lambda t$\,,}\\
u^R\quad &\text{if\quad $x>\lambda t$\,,}
\end{cases}
$$
is an {\it admissible shock wave} of the $k$-th
family when $u^R=T_k[u^L](s)$ and
$\sigma_k[u^L](s,0)=\sigma_k[u^L](s,s)=\lambda$.

Once we have constructed the elementary curves $T_k$ for each $k$-th
characteristic fa\-mi\-ly, the {\it vanishing viscosity solution} of a
general Riemann problem for (\ref{eq:sysnc}) is then obtained by a standard
procedure observing that the composite mapping
\begin{equation}
\Phi(s_1,\dots,s_N) [u^L]\doteq T_N\Big[ T_{N-1}\Big[\cdots
 \big[T_1[u^L](s_1)\big]\cdots\Big](s_{N-1})
\Big](s_N)\doteq
u^R\,,
\label{eq2:RP1}
\end{equation}
is one-to-one from a neighborhood of origin to a neighborhood
of $u^L$. This is a consequence of the fact that the curves $T_k[u]$
are tangent to $r_k(u)$ at $s=0$~\cite{BB, srp}.
Therefore, 
we can uniquely determine intermediate states $u^L\doteq\omega_0,$
$\omega_1,$ $\dots,$ $\omega_N\doteq u^R$, and wave sizes $s_1, \dots,
s_N,$ such that there holds
\begin{equation}
\label{eq2:RP2}
\omega_k =T_k[\omega_{k-1}](s_k)\quad\qquad k=1,\dots,N\,,
\end{equation}
provided that the left and right states $u^L, u^R$ are sufficiently
close to each other. Each Riemann problem
with initial data 
\begin{equation}
\label{elemriem}
\overline u_k(x) = 
\begin{cases}
\omega_{k-1} &\text{if \ \ $x<0$,}\\
\omega_k &\text{if \ \ $x>0$,}
\end{cases}
\end{equation}
admits a vanishing viscosity solution of {\it total size} $s_k$,
containing a sequence of rarefactions and Liu admissible
discontinuities of the $k$-th family. Then, due to the uniform
strict hyperbolicity Assumption~(\ref{eq:strhyp}), the general solution of
the Riemann problem with initial data $\big(u^L,\,u^R\big)$ is
obtained by combining the vanishing viscosity solutions of the
elementary Riemann problems (\ref{eq:sysnc}) (\ref{elemriem}). Throughout
the paper, with a slight abuse of notation, we shall often call $s$ a
wave of (total) size~$s$, and, if $u^R=T_k[u^L](s)$, we will say that
$(u^L,\,u^R)$ is a wave of size $s$ of the $k$-th characteristic
family.

\subsection{The algorithm}
\label{Ss:frontTr}

Now we briefly describe the algorithm we use to construct
a piecewise constant approximate solution to Eqs.~\eqref{eq:sysnc}-\eqref{eq:inda}.
We will summarize the convergence theorems~\cite[Theorems~4.1, 4.4]{ACM1} in Theorem~\ref{T:localConv} after describing the construction.

First of all, let us recall what a front tracking solution is for a homogeneous
hyperbolic system (see~\cite{AMfr} for details).
\begin{definition}
\label{def:ft}
Let $\eps>0$ and an interval $I \subset \R$ be fixed and let
$A=A(u)$, $u\in\R^N$, be a smooth hyperbolic matrix $N\times N$.
We say that a continuous map $u:
I\mapsto \elleuno_{loc}(\real;\,\real^N)$, is an
$\eps$-approximate front tracking solution to Equation~
\eqref{eq:syshom}
if the following conditions hold:
\begin{enumerate}
\item
 As a function of two variables, $u=u(t,x)$ is piecewise constant
 with discontinuities occurring along finitely many straight lines in
 the $t$-$x$ plane. Jumps can be of two types: elementary wavefronts and nonphysical wavefronts, denoted, respectively, as $\E$ and $\NP$. Only finitely many wave-front interactions occur, each involving exactly two incoming fronts.
\item
 Along each elementary front $x=x_\alpha(t)$, the values $u^L\doteq u(t,\,x_\alpha-)$ and $u^R\doteq u(t,\,x_\alpha+)$
 satisfy the following properties for each $\alpha\in \E$. There exists some wave size
 $s_\alpha$ and some index $k_\alpha\in\{1,\dots,N\}$ such that
 \begin{equation}
 \label{eq3:appphy}
 u^R = T_{k_\alpha} [u^L](s_\alpha)\,.
 \end{equation}
 Moreover, the speed $\dot x_\alpha$ of the wave-front satisfies
 \begin{equation}
 \label{shockspeederr}
 \Big| \dot x_\alpha - \sigma_{k_\alpha}[u^L](s_\alpha,\tau)\Big|
 \leq \rarbound\,,
 \qquad\forall~\tau\in [0,\,s_\alpha]\,.
 \end{equation}
\item
 All nonphysical fronts $x=x_\alpha(t)$, $\alpha\in {\mathcal N}$
 have the same speed
 \begin{equation}
 \label{npspeed1}
 \dot x_\alpha\equiv\widehat\lambda\,,
 \end{equation}
 where $\widehat\lambda$ is a fixed constant strictly greater than
 all characteristic speeds, i.e. 
 \begin{equation}
 \label{nps}
 \widehat\lambda>\lambda_k(u)\qquad
 \quad\forall~u\in\Omega,\quad k=1,\dots,N\,.
 \end{equation}
 Moreover, the total strength of all nonphysical
 fronts in $u(t,\cdot)$ remains uniformly small, namely one has
 \begin{equation}
 \label{npbound}
 \sum_{\alpha\in\NP}
 \big|u(t,x_\alpha+)-u(t,x_\alpha-)\big|\leq 
 \eps
 \qquad\quad\forall~t\geq 0\,.
 \end{equation}
\end{enumerate}
We consider non-physical fronts to belong to a fictitious $N+1$-th family.
\end{definition}

In order to construct piecewise constant approximations
to~\eqref{eq:sysnc}-\eqref{eq:inda}, we follow the approach of~\cite{CP} and
construct a local solution to~\eqref{eq:sysnc}-\eqref{eq:inda} by means
of a fractional step algorithm combined with a front tracking method.
In order to do this, we assume that Assumption~\textbf{(G)} at~\pageref{Ass:G}
holds.
Hence, once two sequences
\[
\{ \tau_{\nu} \}_{\nu\in\nat},\ 
\{ \eps_{\nu} \}_{\nu\in\nat} \ ,
\qquad
0<\tau_{\nu}\leq \eps_\nu \downarrow 0 \ ,
\]
are given, we fix $\nu\in\nat$ and we proceed in this way
in order to construct and $\eps_\nu$-approximate fractional-step
approximation $u_\nu=u_{\eps_{\nu}}$ of the solution.
Fist of all, we approximate the initial datum $\overline{u}$ by means of a piecewise
constant function $\overline{u}_{\nu}$ such that
$$
\TV \overline{u}_{\nu} \leq \TV \overline{u}\,, \qquad
\Vert \overline{u}_{\nu}-\overline{u}\Vert_{L^1} \to 0
\quad\text{as $\nu\uparrow\infty$}\,.
$$
Then, we take a suitable approximation
$g_{\nu}$ of $g$ piecewise constant w.r.t. $x$, i.e., following~\cite[\S~3]{CP},
we let
\begin{subequations}
\label{E:discretizationSource}
\begin{equation}
\label{eq:gnu1}
g_\nu(t,x,v)\doteq \sum_{j\in\Z} \chi_{[j\eps_\nu,(j+1)\eps_\nu[}(x)
g_j(t,v)\,,
\end{equation}
where $\chi_I$ is the characteristic function of the set $I$, and
\begin{equation}
\label{eq:gnu2}
g_j(t,v) = \frac{1}{\eps_\nu} \int_{j\eps_\nu}^{(j+1)\eps_\nu}
g(t,x,v)\, dx\,.
\end{equation}
\end{subequations}
Then, the algorithm that
leads to the construction of the approximation $u_{\nu}=u_{\nu}(t,x)$
essentially consists of the following steps. 
\begin{enumerate}
\item
We apply a front tracking algorithm as described in~\cite{AMfr}, which we refer
to, to construct
an $\eps_\nu$-approximate front tracking solution in the sense of
Definition~\ref{def:ft} in the time interval $]0,\tau_{\nu}[$.
\item\label{item:sourcecorrection}
At $t=\tau_{\nu}$ we correct the term $u_{\nu} (\tau_{\nu}-,\cdot)$
by setting
\ba
\label{E:correctionsource}
u_{\nu} (\tau_{\nu}+,\cdot)= u_{\nu} (\tau_{\nu}-,\cdot)+
\tau_{\nu} g_{\nu} \big( \tau_{\nu} , \cdot, u_{\nu} (\tau_{\nu}-,\cdot) \big))\,,
\ea
which turns out to be piecewise constant by construction. Slightly varying the speeds, we can also assume for simplicity that no jump of $u_{\nu}$ is present at $(\tau_{\nu},j\eps_\nu)$ and that no interaction point lies on the line $t=\tau_{\nu}$.
\item
In general, once $u_{\nu} (n\tau_{\nu}+,\cdot)$, $n\geq 1$, is given,
we again use the algorithm in~\cite{AMfr} to construct
an $\eps_\nu$-approximate front tracking solution in the time interval
$]n\tau_{\nu}, (n+1)\tau_{\nu}[$.
\item
Similarly to what done above at $t=\tau_{\nu}$,
at $t=(n+1)\tau_{\nu}$ we correct the term $u_{\nu} ((n+1)\tau_{\nu}-,\cdot)$
by setting
$$
u_{\nu} \big( (n+1)\tau_{\nu}+,\cdot \big)= u_{\nu} \big(
(n+1) \tau_{\nu}-,\cdot \big)+
\tau_{\nu} g_{\nu} \big( (n+1)\tau_{\nu} , \cdot, u_{\nu}
\big( (n+1)\tau_{\nu}-,\cdot \big) \big))\,.
$$
\end{enumerate}
 We stress that, in the construction described above, nonphysical waves are implicitly restarted at each time step: the corresponding jumps are solved using physical waves. 
As usual with such algorithms, the main difficulties we have to face are the following.
\begin{itemize}
\item
To bound uniformly the total variation of $u_{\nu} (t,\cdot)$ in order to get
compactness of the approximating sequence;
\item
To let the number of the fronts to remain bounded in any time interval $[0,t]$.
\end{itemize}
We will briefly discuss how to overcome the first difficulty in
\S~\ref{Ss:PreviousEstimates}, taking advantage of the results contained
in~\cite{AMfr,sie,CP}. Regarding the second difficulty,
using the arguments contained in~\cite[\S~6.2]{AMfr},
it can be easily seen that
the number of wave fronts stays bounded in each time interval $[k\tau_\nu,
(k+1)\tau_\nu[$, and their number depends on the parameter $\eps_\nu$
and on the total variation of $u_\nu(t,\cdot)$ which remains uniformly bounded.

\subsection{Evolution / interaction estimates}
\label{Ss:PreviousEstimates}
\indent
In correspondence of a sequence $\{\eps_\nu\}_{\nu\geq 1}$,
$\eps_\nu\downarrow 0$,
and following~\cite{Glimm}, in this subsection we will define the interaction potential
and give the interaction estimates that will allow us to perform uniform
bounds on the total variation of an $\eps_\nu$ front-tracking approximate
solution. To this purpose,
following~\cite[Definition~3.5]{sie}, we first summarise in the following lemma the relevant points of the definition
of the quantity of interaction between wave-fronts of an approximate
solution.

\begin{lemma}
\label{def:qi}
{\rm
Consider two interacting wave fronts of sizes $s', s''$ ($s'$ located on
the left of $s''$), belonging to the $k', k''\in\{1,\dots, N+1\}$-th
characteristic family, respectively, and let 
 $u^L,\,u^M,\,u^R$,
 denote the left, middle, and right
states before the interaction.
There is a function called {\it
amount of interaction} ${\I}(s^\prime,\,s^{\second})$
between $s^\prime$ and $s^{\second}$ that in the following particular cases is the quantity specified as
follows and for which Proposition~\ref{pro:ie} holds.
\begin{enumerate}
 \item
 If $s^\prime$ and $s^{\second}$ belong to different characteristic
 families, i.e. if $k''<k'\leq N+1$, then set
 \begin{equation}
 \label{eq2:intdf}
	{\I}(s^\prime,\,s^{\second}) \doteq \vert s^\prime
	s^{\second} \vert\,.
 \end{equation}
 \item
 If $s^\prime$ and $s^{\second}$ belong to the same $k\, (\leq\!\!
 N)$-th characteristic family $(k\doteq k'=k'')$, in the
conservative case, and if they are both shocks of the $k$-th family
that have the same sign, then the amount of interaction takes the form of the product of the strength of the waves times
the difference of their Rankine Hugoniot speeds:
$$
\I(s', s'')=\big|s' s''\big| \Big|\sigma_k[u^L](s')]-\sigma_k[u^M](s'')\Big|\,,
$$
where $u^M=S_{k}[u^{L}](s')$ and $u^R=S_{k}[u^{M}](s'')
$.
\end{enumerate}
}
\end{lemma}

Now, whenever an $\eps$-approximate front tracking solution
$u_{\nu}=u_{\eps_{\nu}}(t,x)$ to Equation~
\eqref{eq:syshom} is given,
we define the interaction potential (see~\cite[(4.2)]{sie})
\begin{align}
\label{eq:ip}
\Qg (u_{\nu}(t,\cdot)) = \sum_{\substack{i<j\\ x'>x''}}&
\big\vert s'_{x',i}s''_{x'',j} \big\vert +\\
&+
\frac{1}{4} \sum_{x',x'',i} \int_0^{\vert s_{x',i}\vert} \int_0^{s_{x'',i}}
\big\vert \sigma_{x',i}(\tau')-\sigma_{x'',i}(\tau'') \big\vert\, d\tau'd\tau''\,,\notag
\end{align}
where $s_{x,k}$ is the size of the wave of the $k$-th characteristic family
at $x$, and $\sigma_{x,k}(\tau)$ is its speed, suitably defined. 
Moreover, we let
\begin{equation}
\label{eq:V}
\mathcal{V} (u_{\nu}(t,\cdot)) = \sum_{x,i} \vert s_{x,i}\vert\,.
\end{equation}
With these definitions, the following result holds (see~\cite[Proposition~4.1]{sie}):
\begin{proposition}
\label{pro:ie}
There exists $\dH>0$ such that, if
$u_{\nu}=u_{\eps_{\nu}}(t,x)$ is an $\eps$-approximate front tracking solution
to Eqs.~\eqref{eq:syshom}-\eqref{eq:inda} with $\TV \overline{u}<\dH$,
the following holds. There exists constants $c,\kappa>0$ such that,
whenever two wave fronts $s',s''$ interact, then
$$
\Delta \mathcal{Q} \leq -c\I (s',s'')\,,
$$
and, moreover, the functional
\begin{equation}
\label{eq:Ups}
t\mapsto \Upsilon (u_{\nu}(t,\cdot)) = \mathcal{V} (u_{\nu}(t,\cdot))
+\kappa \Qg (u_{\nu}(t,\cdot))
\end{equation}
is decreasing.
\end{proposition}
%

\subsection{Existence and convergence of approximations}
\label{Ss:convergenceAndExistence}

In this section local in time existence of time-step approximations constructed in \S~\ref{sec:PCA} and that they converge to the entropy solution of the Cauchy problem in Eqs.~\eqref{eq:sysnc}-\eqref{eq:inda}, as stated in Theorem~\ref{T:localConv} below. Uniqueness was proved roughly following the lines of~\cite{AG2}.
%
%
%
%

We first fix a classical notation.
Let $S_{i}[u^{L}](s)$ denote the \textbf{$i$-th Hu\-go\-niot curve} issuing from $u^{L}$, $i\in\{1,\dots,N\}$. We denote by $\sigma_{i}[u^{L}](s)$ the corresponding Rankine-Hugoniot \textbf{speed} of the $i$-th discontinuity $[u^{L},S_{i}[u^{L}](s)]$: $\sigma_{i}$ and $S_{i}$ are defined by the Implicit Function Theorem by the relation
\begin{equation}\label{E:RH}
\sigma_{i}[u^{L}](s)\, \left(S_{i}[u^{L}](s)-u^{L}\right) = f\left(S_{i}[u^{L}](s)\right)-f(u^{L})
\end{equation}
together with
\[
S_{i}[u^{L}](0)=u^{L}\ ,
\qquad
\sigma_{i}[u^{L}](0)=\lambda_{i}(u^{L})\ ,
\qquad
\dds S_{i}[u^{L}](0)=r_{i}(0)
\ .
\]
One can suppose that $ S_{i}[u^{L}]$ is parameterized by the $i$-th component relative to the basis $r_1(u),\ldots, r_N(u)$.
If $u^{R}=S_{i}[u^{L}](s)$, we denote also by $\sigma_{i}(u^{L},u^{R})=\sigma_{i}[u^{L}](s)$ the \textbf{speed} of the $i$-th discontinuity $[u^{L},u^{R}]$. This $i$-th discontinuity is \textbf{admissible} when~\cite{tplrpnpn}
\[
\forall 0\leq |\tau|\leq |s| 
\qquad
\sigma_{i}[u^{L}](\tau)\geq\sigma_{i}(u^{L},u^{R})\ .
\]
%

Let $\Upsilon$ be the functional introduced in Equation~
\eqref{eq:Ups} for studying the well-posedness of the system~\eqref{eq:sysnc}.

\begin{theorem}
\label{T:localConv}
There exist $\dssb, T>0$ such that for initial data $\overline{\uu}$ in the closed domain
\[
\mathfrak D_{p}(\dssb)\doteq{}\left\{\uu\in L ^{1}(\R;\R^{N})\cap \BV(\R;\R^{N})\text{ piecewise constant s.t.~} \Upsilon(\uu)\leq\dssb\right\}
\]
the algorithm described in \S~\ref{Ss:frontTr} defines for $t\in[0,T]$ and for every $\nu$ an approximating function
\bel{E:domt}
\ww_{\nu}(t,\cdot)\in\mathfrak D_{p}\left( \dssb+Gt\right) \quad \text{where $\dssb,G$ only depend on $A$ and $g$.}
\end{equation}

This approximating function $w_{\nu}$ satisfies the following comparison estimate with the viscous semigroup~\cite{Ch1} $\vSC{t}{h}[\cdot]$ of the Cauchy problem in~\eqref{eq:sysnc}-\eqref{eq:inda} starting at time $h$: there is a function $o(s)$ depending only on $A$, $g$, $\dssb$, $T$ such that $o(s)/s\to0$ if $s\to0$ and such that for $n\in\nat$
\bel{E:uniqEst}
\big\lVert \ww_{\nu}(n \tau_{\nu}+,\cdot) - \vSC{n \tau_{\nu}}{(n-1)\tau_{\nu}} \left[\ww_{\nu}((n-1)\tau_{\nu}+,\cdot)\right]\big\rVert_{L^{1}} \leq \OL(1)\left(o(\tau_{\nu}) +\eps_{\nu}\tau_{\nu}
\right)
\ .
\end{equation}

For every $\overline{\uu}\in\mathfrak D(\dssb)$ one can choose a suitable piecewise-constant approximation $\overline{\ww}_{\nu}\in \mathfrak D(\dssb)$ of $\overline{\uu}$ such that, denoting by $\ww_{\nu}$ the $\nu$-approximation as in \S~\ref{Ss:frontTr} with initial datum $\overline{\ww}_{\nu}$, for a.e.~$t\in[0,T]$ the sequence $\ww_{\nu}(t,\cdot)$ converges in $L^{1} (\R;\R^{N})$ to $\vSC{t}{0}\overline{\ww}$.

\end{theorem} 

\evidenzia{Introduction to the proof}
Before the proof, we briefly recall our notation and previous results that we need. We denote by $\ftS{t}{h}$ the wavefront tracking approximation of the semigroup $\vSB{t}{h}[\cdot]$ relative to the homogeneous system, constructed by vanishing viscosity~\cite{BB}, where the `initial datum' is fixed at time $h\leq t$ rather than at $h=0$.

We exploit the definition in \S~\ref{sec:PCA} of the approximation: set
\begin{subequations}
\label{E:approx}
\bel{E:approx1}
\ww_{\nu}(t,\cdot) =\ftS{ t}{(n-1)\tau_{\nu}}\left[\ww_{\nu}((n-1) \tau_{\nu}+,\cdot) \right]
\end{equation}
for $(n-1)\tau_{\nu}<t<n\tau_{\nu}$, $n\in\N$, and
\bel{E:approx2}
\begin{split}
\ww_{\nu}( n\tau_{\nu}+,\cdot) 
&\equiv 
 \ftS{ n\tau_{\nu}}{(n-1)\tau_{\nu}}\left[\ww_{\nu}((n-1) \tau_{\nu}+,\cdot) \right]
 \\
 &\quad+ \tau_{\nu}g\left( n\tau_{\nu},\cdot, \ftS{ n\tau_{\nu}}{(n-1)\tau_{\nu}}\left[\ww_{\nu}( n-1) \tau_{\nu}+,\cdot) \right] \right)
\\
&\equiv \ww_{\nu}( n\tau_{\nu}-,\cdot) + \tau_{\nu}g\left( n \tau_{\nu},\cdot, \ww_{\nu}( n\tau_{\nu}-,\cdot) \right)
\end{split}
\end{equation}
\end{subequations}
relative to the balance law with the initial condition $\ww_{\nu}(0+,\cdot)\equiv\bar u(\cdot)$.
We recall that 
\[
\text{if $\overline \ww\in \mathfrak D_{p}(\dssb)$, for $\dssb\leq \dH$ small enough as in~\cite{AMfr}, }
\]
then, considering also Proposition~\ref{pro:ie} above,
\ba
\label{E:rgabab}
&\Upsilon\left(\ftS{t}{h}\overline\ww\right)\leq\Upsilon(\overline\ww)
\qquad \forall0\leq h\leq t\leq \tau_{\nu}<T&
&\text{see~\cite[(6.4)]{AMfr}}&
\\
\label{E:ftest1}
&{}\lVert \ftS{n \tau_{\nu}}{(n-1)
\tau_{\nu}} \overline{\ww} -\vSB{n \tau_{\nu}}{(n-1) \tau_{\nu}} \overline{\ww}
\rVert_{L^{1}} \lesssim (1+\dssb)\varepsilon_{\nu}\tau_{\nu}&
&\text{see~\cite[(3.5)]{AMfr}}&
\\
\label{E:ftest2}
&{}\lVert \ftS{t+s}{t}\overline{\ww} -\overline{\ww}\lVert \leq L s &
&\text{see~\cite[(1.23)]{AMfr}}&
\ea
We also borrow the following lemma from~\cite[Lemmas~2.1-2]{AG2}, given in a similar setting. Of course we could state it similarly also by localizing in space the estimates.
We recall that $\Upsilon, \Qg$ are the functionals introduced in Eqs.~\eqref{eq:ip}-\eqref{eq:Ups} while $\ell_{g}$ and $\alpha$ are as in the Assumption~\textbf{(G)} on the source term at Page~\pageref{Ass:G}.

\begin{lemma}
\label{L:timeEst}
Let $t>0$.
If $0<\dssb<\dH$ and $\overline{\ww},\overline{\uu}$ are piecewise constant with $\Upsilon(\overline{\uu})+\Upsilon(\overline{\ww})\leq\dssb$ then \[\vv(x)\doteq{}\overline{\uu}(x)+\tau g_{\nu}(t,x,\overline{\ww}(x))\] satisfies for $G= \max\{ \ell_{g}\TV(\overline\ww)+\TV(\overline\uu)+\norm{\alpha}_{L^{1}}; 1\} $ the inequalities
\bel{E:rwrg}
\left| \TV(\vv)- \TV(\overline\uu)\right|\lesssim G \tau
\,,\ 
\left|\Qg(\vv) - \Qg(\overline\uu)\right|
\lesssim G^{2} \tau
\,,\ 
\left|\Upsilon(\vv) - \Upsilon(\overline\uu)\right|
\lesssim
 G^{2}\tau \ .
\end{equation}
\end{lemma}

\begin{remark}
\label{R:timeupdate}
When $\overline{\ww}=\overline{\uu}$, then the proof of
Lemma~\ref{L:timeEst} states that where $\overline{\uu}$ has a jump
of strength $\sigma$ then the strength $\sigma'$ of the corresponding jump in
$\vv$ 
satisfies
\[
0< (1-\OL(1)\tau)\,\sigma\leq \sigma' \leq (1+\OL(1)\tau)\,\sigma
 \]
 or\[
 (1+\OL(1)\tau)\,\sigma\leq \sigma' \leq (1-\OL(1)\tau)\,\sigma <0\ .
\]
Moreover, if $\overline{u}$ does not any jump at $x=j\eps_\nu$, 
then the new jump $\sigma_{i,k}$ introduced because of the discontinuity of $g_\nu$
at $j\eps_\nu$ satisfies
$$
\sum_{k=1}^N|\sigma_{j,k}''|\leq \tau \int_{(j-1)\eps_\nu}^{(j+1)\eps_\nu}|\alpha|
\qquad
\Rightarrow 
\qquad
\sum_{j\in\N}\sum_{k=1}^N|\sigma_{j,k}''|\leq 2\tau\norm{\alpha}_{L^{1}(\R)}
\,,
$$
where $\sigma''_{j,k}$ is the strength of the new front of the $k$-th family
emerging from $(\tau, j\eps_\nu)$.
\end{remark}

\section{Estimates on Genuinely nonlinear families}
\label{S:qualitative}

\newcommand{\oc}{\overline{c}}
\newcommand{\op}{\overline{p}}
\newcommand{\oq}{\overline{q}}
\newcommand{\og}{\overline{\curva}}
\newcommand{\ox}{\overline{x}}
\newcommand{\ot}{\overline{t}}
\newcommand{\oP}{\overline{P}}

In~\ref{Ss:convergenceAndExistence} we provided an approximation $w^{\nu }$ of the viscosity solution to the Cauchy problem in Equations~\eqref{eq:sysnc}-\eqref{eq:inda}.
This section is devoted to the proof of estimate~\eqref{E:balancesnu} based on balances for $\wm_{i}^{\nu }$ and $\wm_{i}^{\nu,\jump}$ when the $i$-th characteristic field is genuinely nonlinear in the sense of Definition~\ref{D:GN}.

We work under the standard Lipschitz regularity Assumption~\textbf{(G)} at Page~\pageref{Ass:G} on the source term.

The section is organized as follows:
\begin{itemize}
\item[\S~\ref{s:reviewDisc}] We review how to detect discontinuities in the approximation $u_{\nu}$ that are in the limit converging to a discontinuity of the entropy solution $u$, so that the jump part of $D_{x}u$ is approximated by a part of $D_{x}u_{\nu}$.
\item[\S~\ref{S:definitionMeasures}] We introduce a new measure to control the action of the source, that we call source measures, and we revise other classical tools, like the interaction-cancellation measures.
\item[\S~\ref{S:balances-}]  We establish the lower part of estimate~\eqref{E:balancesnu}.
\item[\S~\ref{Ss:decay-}] We establish the upper part of estimate~\eqref{E:balancesnu}.
\end{itemize}

\subsection{Review of the fine convergence of \texorpdfstring{$i$}{i}-discontinuities}
\label{s:reviewDisc}
We identify which discontinuities in the approximation $u_{\nu}$ are in the limit converging to a discontinuity of the entropy solution $u$: we call these `surviving' discontinuities ``$(\beta,i)$-approximate discontinuities''. We fix for this purpose thresholds $\beta$ and $\beta/2$, later related to $\nu$ as in Remark~\ref{R:relation}.

\begin{remark}
\label{R:norelabelingsubs}
Before entering the topic, we stress that the limiting theorems in the present section apply up to a subsequence of $u^{\nu}$, that we do not relabel. Of course, the limit of $\{u^{\nu}\}_{\nu\in\N}$ is known to be unique~\cite{AG2}, nevertheless the interaction and the interaction--cancellation measures, as well as the other measures related to $u$ introduced in this and previous papers, are defined for the entropy solutions only as limits on suitable subsequences.
\end{remark}

\subsubsection{Definition of approximate discontinuities}

\begin{definition}
\label{D:subdisc}
Let $\beta>0$.
A maximal, leftmost $(\beta,i)$-approximate discontinuity curve is any maximal (concerning set inclusion) closed polygonal line---parametrized with time in the $(t,x)$-plane---with nodes $(t_{0},x_{0})$, $(t_{1},x_{1})$, $\dots$, $(t_{n},x_{n})$, where $t_{0}\leq\dots\leq t_{n}$, such that
\begin{enumerate}
\item each node $(t_{k},x_{k})$, $k=1,\dots,n$ is an interaction point or an update time;
\item
\label{definitionSF2}
the segment $[(t_{k-1},x_{k-1}),(t_{k},x_{k})]$ is the support of an $i$-discontinuity front with strength $|s_{i}|\geq\beta/4$ and there is at least one time $t\in(t_{0},t_{n})$ such that $|s_{i}|\geq\beta$;
\item 
\label{definitionSF3}
it stays on the left of any other polygonal line it intersects and having the above properties.
\end{enumerate}
We denote by $\J^{}_{\beta,i}(\nu)$ the family of maximal, leftmost $(\beta,i)$-approximate discontinuities of $u_{\nu}$.
\end{definition}

Notice that the family of curves $\J^{}_{\beta,i}(\nu)$ is enriched as $\beta\downarrow0$.
We recall an estimate in~\cite[Lemma 5.7]{ACM1}, which is proved assuming the $i$-th characteristic field being piecewise genuinely nonlinear, for an $i$: the cardinality
\[\sharp \J^{}_{\beta,i}(\nu)=:M_{\beta,i}(\nu)\] of maximal, leftmost
$(\beta,i)$-approximate discontinuities---up to any fixed positive
time---when the threshold $\beta$ is fixed is uniformly bounded in $\nu$, and thus also in $i=1,\dots,N$:
it is of order
\bas
M_{\beta,i}(\nu)\leq \bar M_{\beta,i}\lesssim \beta^{-2}\;.
\eas
We repeat once more that such bound, uniform in $\nu$, holds up to a suitable subsequence.

\emph{Since~\cite[Lemma 5.7]{ACM1} was proved without the dependence of $g$ on $x$, we briefly repeat such proof later in Lemma~\ref{L:boundsOnJumps} within \S~\ref{S:definitionMeasures}}.
 The reader will also appreciate in this proof that the new measures we introduce are useful.
  
\begin{definition}
\label{D:approximateDiscC}
For $i=1,\dots, N$,  by~\cite[Lemma 5.7]{ACM1}, proved as Lemma~\ref{L:boundsOnJumps} of the present paper, we enumerate the maximal, leftmost $(\beta,i)$-approximate discontinuities of $u^{\nu}$ possibly with repetitions, as
\bas
& \J^{}_{\beta,i}(\nu)=\left\{\curva_{\beta,m}^{\nu,i}\right\}_{m=1}^{\bar M_{\beta,i}}
&&\curva_{\beta,m}^{\nu,i}:\left[t_{\beta,m}^{\nu,i,-},t_{\beta,m}^{\nu,i,+}\right]\to\R
&&m=1,\dots,\bar M_{\beta,i}\;.
\eas
We define the $(\beta,i)$-jump set and the $(\beta,i)$-approximate discontinuity measure of $u^{\nu}$ as
\bas
&J^{\nu}_{}\doteq{}\bigcup_{m=1}^{\bar M_{\beta,i}}\Graph(\curva_{\beta,m}^{\nu,i})\,,
&&\wm_{\beta,i}^{\nu,\jump}\doteq{}\widetilde \ell_{i}^{\nu}\cdot D_{x}u^{\nu}\restriction_{J^{\nu}}\,,
\eas
where we introduced
$\widetilde \ell_{i}^{\nu}(t,x)\doteq{}\widetilde \ell_{i}(u^{\nu}(t,x-),u^{\nu}(t,x+))$ in~\eqref{E:vettoriTilde}.

For $i=1,\dots,N$ we also define the approximate $(\beta,i)$-wave measure and the approximate $(\beta,i)$-continuity measure of $u^{\nu}$ respectively as
\begin{align}\label{E:imeasCont}
&\wm_i^{\nu}=D_x u^{\nu} \cdot{ \widetilde \ell}_{i}^{\nu} (t,x)\,,
&&
\wm_{\beta,i}^{\nu,\cont}=\wm_{i}^{\nu}-\wm_{\beta,i}^{\nu,\jump}.
\end{align}
\end{definition}
\nomenclature{$\wm_i^{\nu}$}{The $i$-wave measure of $u^{\nu}$, Definition~\ref{D:approximateDiscC}}
\nomenclature{$\wm_{\beta,i}^{\nu,\cont}$}{The approximate $(\beta,i)$-continuity measure of $u^{\nu}$, Definition~\ref{D:approximateDiscC}}
\nomenclature{$\wm_{\beta,i}^{\nu,\jump}$}{The approximate $(\beta,i)$-jump measure of $u^{\nu}$, Definition~\ref{D:approximateDiscC}}

\subsubsection{Interaction and cancellation measures}
\label{Sss:ICm}
We remind the definition of the well established interaction and interaction-cancellation measures of $u^{\nu}$, generalizing~\cite[\S~7.6]{BressanBook} by the Definition~\ref{def:qi} of amount of interaction $\I$.
For simplifying later estimates, we also include in such measures interactions among physical and nonphysical fronts.

\begin{definition}
\label{D:ICmeasuresnu}
The interaction $\mu^{I}_{\nu}$ and the interaction-cancellation $\mu^{IC}_{\nu}$ measures of $u^{\nu}$ are purely atomic, positive measures concentrated on the set of points $P$ where two wave-fronts of $u^{\nu}$ interact. If the incoming fronts belong to the families $i,j\in\{1,\dots,N,N+1\}$ and they have sizes $s'$, $s''$ we define
\ba\label{E:interactionMeas}
&\mu^{I}_{\nu}(P)=\I (s',s'')\;,
\\
&
\mu^{IC}_{\nu}(P)=\I (s',s'')+\begin{cases}
\abs{s'}+\abs{s''}-\abs{s'+s''}&\text{if $i=j$,}
\\
0&\text{otherwise.}\end{cases}
\ea
\end{definition}

\nomenclature{$\mu^{I}_{\nu}$, $\mu^{IC}_{\nu}$}{The interaction and interaction-cancellation measures of $u^{\nu}$ in Definition~\ref{D:ICmeasuresnu}}

Of course $0\leq \mu_{\nu}^{I}\leq \mu_{\nu}^{IC}$.
We remind that the interaction-cancellation measure can be controlled by the interaction potential $\Qg$ defined at Equation~
\eqref{eq:ip}, even when a source term is present~\cite[Lemma~5.2]{ACM1}:
\ba\label{E:estimatemuIC}
 \mu_{\nu}^{IC}((t_{1},t_{2}]\times\R)
\lesssim
 \TV^{-}\left( \Qg(u_{\nu});(t_{1},t_{2}]\right)
 \qquad
\forall \ 0\leq t_{1}< t_{2}\ ,
\ea
where $ \TV^{-}$ is the negative total variation.
In particular, $\mu_{\nu}^{IC}$ is a locally bounded Radon measure.
We stress that the proof of~\cite[Lemma~5.2]{ACM1} is already made for a source term satisfying Assumption~\textbf{(G)}.

Following Remark~\ref{R:norelabelingsubs}, we assume that the measures $ \mu_{\nu}^{I}$ and $ \mu_{\nu}^{IC}$ associated to the sequence $\{u^{\nu}\}_{\nu}$ we consider are weak*-convergent. Indeed, by compactness, this is true up to subsequence, even if the limit measures do depend on the particular subsequence we are considering.
Reading the brief proof of~\cite[Lemma~5.2]{ACM1} one notices that was already proved for source terms satisfying Assumption~\textbf{(G)}, not only when $g=g(u)$.

\begin{definition}
\label{D:ICmeasures}
We define the interaction $\mu^{I}_{}$ and the interaction-cancellation $\mu^{IC}_{}$ measures of $u$ as
\ba\label{E:muICl}
&\mu^{I}_{}=\wlim{\nu}\mu^{I}_{\nu} \;,
&&
\mu^{IC}_{}=\wlim{\nu}\mu^{IC}_{\nu} \;.
\ea
\end{definition}

\nomenclature{$\mu^{I}$, $\mu^{IC}$}{An interaction and an interaction-cancellation measures of $u$ in Definition~\ref{D:ICmeasures}}

\subsubsection{Convergence of jumps}
We state the following convergence as a direct consequence of the points in~\cite[Teorema 5.1]{ACM1} related to piecewise genuinely nonlinear families, we clarify the proof just below. The $i$-discontinuity measure $\wm_i$ of $u$ was introduced in Definition~\ref{D:upsiloni} in order to decompose $D_x u $ along the right eigenvectors $\widetilde r_i$ of the matrix $\widetilde A$, see~\eqref{E:vettoriTilde}.

\begin{lemma}
\label{L:convergenceJumps}
Let $u\in\BV_{\loc}([0,T]\times\R)$ be an entropy solution of the strictly hyperbolic system of balance laws~\eqref{eq:sysnc} with a sufficiently small $\BV$ norm. Suppose that the $i$-th field is genuinely nonlinear.
Then, as $\nu\uparrow\infty$, each curve $\curva_{\beta,m}^{\nu,i}$ converges locally uniformly to a Lipschitz continuous curve 
\bas
&\curva_{\beta,m}^{i}:\left[t_{\beta,m}^{i,-},t_{\beta,m}^{i,+}\right]\to\R
&&\text{with $t_{\beta,m}^{\nu,i,-}\to t_{\beta,m}^{i,-}$ and $t_{\beta,m}^{\nu,i,+}\to t_{\beta,m}^{i,+}$,}
\eas 
the jump measure $\wm_{i}^{\jump}$ of $u$ is concentrated on 
\bas
&J\doteq{}\bigcup_{\beta_{\nu}\downarrow0}\bigcup_{m=1}^{\bar M_{\beta,i}}\Graph(\curva_{\beta_{\nu},m}^{i})
&&\text{for a sequence $\beta_{\nu}\downarrow0$}
\eas 
and the approximate $(\beta_\nu,i)$-discontinuity measure $\wm_{\beta_{\nu},i}^{\nu,\jump}$ of $u^{\nu}$ converges weakly$^{*}$ to the $i$-discontinuity measure $\wm_{i}^{\jump}$.
\end{lemma}

By the linear relation~\eqref{E:imeasCont}, since the approximate $(\beta_\nu,i)$-wave measure $\wm_i^{\nu}$ converges weakly$^{*}$ to the $i$-wave measure $\wm_i$, Lemma~\ref{L:convergenceJumps} implies that the approximate $(\beta_\nu,i)$-continuity measure $\wm_{\beta_{\nu},i}^{\nu,\cont}$ converges weakly$^{*}$ to the $i$-continuity measure $\wm_{i}^{\cont}$.

\begin{remark}
\label{R:relation}
The sequence $\beta_{\nu}\downarrow0$ can be chosen so that 
\bas
\beta_{\nu}>4(1+\norm{\alpha}_{L^{1}(\R)})\pr{\varepsilon_{\nu}+\tau_{\nu}\overline M}\,,
\eas
where $\overline M$ is a bound on the total variation.
\end{remark}
\nomenclature{$\beta_{\nu}$}{See Remark~\ref{R:relation}}

\begin{proof}[Proof of Lemma~\ref{L:convergenceJumps}]
We stress that in~\cite{ACM1} there were assumptions on all the fields, of either linear degeneracy or genuine nonlinearity, in order to have a stronger statement concerning the structure of solutions. 
Moreover, for that only theorem there was the additional assumption $g=g(u).$
In the proof of the points we quote, nevertheless, only the assumption on the $i$-th field was used, and not on the other fields, as we explain.
We moreover extend the proof in~\cite{ACM1} where the assumption $g=g(u)$ was used, as we specify at the relative occurrences.

First of all, having uniformly Lipschitz continuous curves, the uniform convergence of the curves $\J^{}_{\beta,i}(\nu)$, up to a suitable subsequence, is just a consequence of Ascoli-Arzel\`a theorem and a diagonal argument: let
\bas
& \J^{}_{\beta,i}=\left\{\curva_{\beta,m}^{i}\right\}_{m=1}^{\bar M_{\beta,i}}\,,
&&\curva_{\beta,m}^{i}:\left[t_{\beta,m}^{i,-},t_{\beta,m}^{i,+}\right]\to\R
&&m=1,\dots,\bar M_{\beta,i}\;.
\eas

By the fine property of $\BV$ function ~\cite[\S~3.9]{AFPBook} moreover the jump measure of $u$ that was introduced in Definition~\ref{D:upsiloni} is concentrated on the graphs $\{G_{m}\}_{m\in\N}$ of at most countably many Lipschitz continuous curves and at $\Ha^{1}$-a.e.~points of such curves there is a tangent line and there are approximate left and right limits.
Such values must satisfy Rankine-Hugoniot conditions\eqref{E:RH}: we associate them to the $j$-th characteristic family when the speed of the curve lies within the range of the $j$-th eigenvalue of the system, for $j=1,\dots,N$.

Consider the at most countable set of atoms of the interaction-can\-cel\-la\-tion measure that we will recall in Definition~\ref{D:ICmeasures} below:
\ba\label{E:Theta}
\Theta\doteq{}\left\{(t,x)\ |\ \mu^{IC}(\{(t,x)\})>0\right\}\;.
\ea
Recall that the source measure $\mu^{\sour}_{\nu}$ in Definition~\ref{D:measuresnu} has an absolutely continuous limit as $\nu\to+\infty$.

Point (5) of~\cite[Theorem~5.1]{ACM1} states that when $(\bar t,\bar x=\curva_{\beta,m}^{i}(t))\notin\Theta$ then
\ba
\label{E:limitAtJump}
&\lim_{r\downarrow0}\limsup_{\nu\uparrow\infty}\left(\sup_{\substack{\abs{t-\bar t}+\abs{x-\bar x}\\ x<\curva_{\beta,m}^{i}(t)}} \abs{u_{\nu}(t,x)-u^{L}}\right)=0
&&\text{with }u^{L}\doteq{}u(\bar t, \bar x-)\;,
\ea
and similarly on the right of $\curva_{\beta,m}^{i}(\bar t)$ with $u^{R}\doteq{}u(\bar t, \bar x+)$.
The proof given in~\cite[\S~5.3.2 Page 367,368]{ACM1} does not use the assumption that the other families $j\neq i$ are either linearly degenerate or piecewise genuinely nonlinear.
Moreover, the source term only enters in the use of~\cite[Lemma~5.10-11]{ACM1}, we thus treat in Remark~\ref{R:lemma51011} below why they hold also when the source satisfies Assumption~\textbf{(G)}, when the $i$-th field is genuinely nonlinear.
 
Given any sequence $\beta_{\nu}\downarrow0$, thus, the convergence of the graphs of the curves in $ \J^{}_{\beta,i}$ implies
\bas
&Du\restriction_{J}=\wlim{\nu} Du^{\nu}\restriction_{J_{\nu}}\,,
&&J_{\nu}\doteq{}\bigcup_{m=1}^{\bar M_{\beta,i}}\Graph(\curva_{\beta_{\nu},m}^{i})\;.
\eas
As well, at those points on such graphs a convergence analogous to~\eqref{E:limitAtJump} holds also for $\widetilde \ell_{i}$, $\widetilde r_{i}$ and $\widetilde \lambda_{i}$ since the flux is assumed to be $C^{2}$.
As a consequence
\bas
&\wm_{i}^{\jump}\restriction_{J}=\widetilde \ell_{i}\cdot Du\restriction_{J}=\wlim{\nu} \widetilde \ell_{i}^{\nu}\cdot Du^{\nu}\restriction_{J_{\nu}}=\wlim{\nu} \wm_{\beta_{\nu},i}^{\nu,\jump}
\;.
\eas

It remains to show that $\wm_{i}^{\jump}\restriction_{J}=\wm_{i}^{\jump}$.

Suppose now that $\overline P=(\bar t,\bar x)\notin\Theta$ and $\bar x\neq \curva_{\beta_{\nu},m}^{i}(\bar t)$, for all elements of the sequence $\beta_{\nu}\downarrow 0$ we are considering.
If $\wm_{i}^{\jump}(\overline P)\neq0$, which means that $P$ is a jump point of $u$, there exists $\varepsilon>0$ and a space-like segment $P_{\nu}Q_{\nu}$ degenerating to the single point $\overline P$ for which
\bas
&u_{\nu}(P_{\nu})\to u(\overline P)\;,
&&\abs{\widetilde \ell_{i}^{\nu}\cdot \left(u_{\nu}(Q_{\nu})- u(\overline P)\right)}\geq\varepsilon\quad\forall \nu\;.
\eas
Two cases are possible.
Case 1. There are two distinct indices $j\neq k$ such that each segment $P_{\nu}Q_{\nu}$ is crossed by a fixed amount of both $j$-waves and $k$-waves in $u_{\nu}$:~\cite[Lemma 5.10]{ACM1} contradicts the assumption $\mu^{IC}(\overline P)=0$ because of non-vanishing interactions among the different families $k$ and $j$.
Case 2. For a single index $j\in\{1,\dots,n\}$ each segment $P_{\nu}Q_{\nu}$ is crossed by an amount $\geq\varepsilon>0$ of vanishingly small $j$-waves, but for $k\neq j$ the total amount of $k$-waves crossing the segment $P_{\nu}Q_{\nu}$ vanishes.
Since $\wm_{i}^{\jump}(\overline P)\neq0$ that index $j$ must actually be $i$ itself.
By genuine nonlinearity of the $i$-th characteristic family,~\cite[Lemma 5.11]{ACM1} applies and it contradicts the assumption $\mu^{IC}(\overline P)=0$ because of non-vanishing interactions among $i$-waves.

 For completeness, we prove in Remark~\ref{R:lemma51011} below why \cite[Lemma 5.10-11]{ACM1} hold also when the source satisfies Assumption~\textbf{(G)}, when the $i$-th field is genuinely nonlinear, and not only when $g=g(u)$ Lipschitz.
\end{proof}

\subsection{Definition of Source Measures and review of relevant tools}
\label{S:definitionMeasures}
We explained in \S~\ref{S:SBVGNargument} that the upper-level argument to achieve $\SBV$-regularity for genuinely nonlinear characteristic fields is based on the construction of relevant measures to prove balances for the flux of positive and negative waves on characteristic regions. The task of the present section is to precisely construct such measures.

Let $\{u^{\nu}\}_{\nu}$ be a sequence of fractional-step approximations running with front-tracking as defined in \S~\ref{Ss:convergenceAndExistence}.
We already introduced in \S~\ref{Sss:ICm} the interaction and cancellation measures $\mu^{IC}_{\nu}$, $\mu^{IC}_{\nu}$, $\mu^{IC}_{ }$, $\mu^{I}_{ }$.


\vskip.3\baselineskip

We now define measures to keep track of how the wave measures $\wm_{i}$ in Definition~\ref{D:upsiloni}, their jump part, and their continuous part vary. Since we are not presently able to compute estimates directly on $u$, we actually look at how the \emph{approximate} wave measures $\wm_{ i}^{\nu}$, their jump part $\wm_{\beta_{\nu},i}^{\nu,\jump}$ and their continuous part $\wm_{\beta_{\nu},i}^{\nu,\cont}$ of Definition~\ref{D:approximateDiscC}--Lemma~\ref{L:convergenceJumps} vary.

\begin{definition}\label{def:velCarAppr}
We define the speed of the $i$-th wave of the $\nu$-approximate solution $u^{\nu}$ as $\widetilde{\lambda_{i}^{\nu}}(t,x)$ equal to
\begin{itemize}
\item $\lambda_{i}(u^{\nu}(t,x))$ where $u^{\nu}(t,x)$ is continuous, namely constant, and
\item
 the actual speed of the discontinuity curves at discontinuities.
\end{itemize}
\end{definition}

To simplify the analysis, we assume that the fronts satisfy the Rankine-Hugoniot conditions exactly, although  it is only close to RH speed~\eqref{E:RH} by construction~\eqref{shockspeederr} of the approximation.

\nomenclature{$ \widetilde{\lambda_{i}^{\nu}}$}{Speed of the $i$-th wave of the $\nu$-approximate solution $u^{\nu}$ in Definition~\ref{def:velCarAppr}}

Since the total size of nonphysical wavefronts is of the same order of $\eps_\nu$, for simplicity in the following decomposition we only consider the physical fronts.
This is in line with the fact that at update times no nonphysical front arises by construction: this motivates the fact of defining the $\nu$-approximate source measure with a summation over the physical families only.

We recall that $[\![\cdot]\!]$ denotes the integer part of a number.

\begin{definition}
\label{D:measuresnu}
On $[0,T]\times \R$, we define $\nu$-\emph{approximate source measure} the purely atomic measure
\[
\mu^{\sour}_{\nu} (dt,dx)
\doteq{} \sum_{n=1}^{[\![T/\tau_{\nu}]\!]}\left(  \sum_{i=1}^{N}\abs{\wm_{i}^{\nu}(t-,x)}\delta_{n\tau_{\nu}}(t)+
\sum_{ j\in\Z}q_{j}\delta_{(n\tau_{\nu},j\varepsilon_{\nu})}(t,x) \right)\tau_{\nu}\ ,
\]\[ q_{j}\doteq{} \int_{j\eps_\nu}^{(j+1)\eps_\nu}\!\!\! \!\!\! \!\!\! \alpha(s)ds\ .
\]
\nomenclature{$\mu_{\nu}^{\sour}$}{The $\nu$-approximate source measure in Definition~\ref{D:measuresnu}}%
\nomenclature{$\mu_{i}^{\nu} $}{The $\nu$-approximate wave-balance measure in Definition~\ref{D:measuresnu}}%
\nomenclature{$\mu_{i}^{\nu,\jump} $}{The $\nu$-jump-wave-balance measure in Definition~\ref{D:measuresnu}}%
For $i=1,\dots,N$, we define $\nu$-\emph{$i$-wave-balance measure} the measure
\[
\mu_{i}^{\nu} \doteq{} \partial_{t} \left(\wm_{i}^{\nu}\right) + \partial_{x}\left( \widetilde{\lambda_{i}^{\nu}}\wm_{i}^{\nu} \right)\;.
\]
Define $\nu$-\emph{jump-wave-balance measure} and $\nu$-\emph{continuous-wave-bal\-ance mea\-sure} the mea\-sures
\bas
\mu_{i}^{\nu,\jump} &\doteq{} \partial_{t} \left(\wm_{i}^{\nu,\jump}\right) + \partial_{x}\left( \widetilde{\lambda_{i}^{\nu}}\wm_{i}^{\nu,\jump} \right)\;,\\
\mu_{i}^{\nu,\cont} &\doteq{} \partial_{t} \left(\wm_{i}^{\nu,\cont}\right) + \partial_{x}\left( \widetilde{\lambda_{i}^{\nu}}\wm_{i}^{\nu,\cont} \right)=\mu_{i}^{\nu}-\mu_{i}^{\nu,\jump} \;.
\eas
Summing up, we define the measures
\bas
&\mu_{\nu}^{ICJS}\doteq{}\mu_{\nu}^{IC}+\sum_{i=1}^{N}\abs{\mu_{i}^{\nu,\jump}}+  \mu ^{\sour}_{\nu}\;,
&&
\mu_{\nu}^{ICS}\doteq{}\mu_{\nu}^{IC}+  \mu ^{\sour}_{\nu} \;.
\eas
\end{definition}
\nomenclature{$\mu_{\nu}^{ICJS}$}{The interaction-cancellation-source measure in Definition~\ref{D:measuresnu}}%
\nomenclature{$\mu_{\nu}^{ICJS}$}{The interaction-cancellation-jump-source measure in Definition~\ref{D:measuresnu}}%

To get familiar with such measures, one can look at~\cite[Example 3.1]{ACMnota}, where technicalities of nonphysical waves is totally missing.
Due to the fact that $\widetilde{\lambda_{i}^{\nu}}(t,x)$ is the speed of fronts, a direct computation shows that $\mu_{i}^{\nu}$ is the purely atomic measure given by
\ba\label{E:muinuexp}
\mu_{i}^{\nu}=\sum_{k}p_{k}\delta_{(t_{k},z_{k})}+\rho^{\nu}_{i}\,,
\ea
where:
\begin{itemize}
\item The sum runs on nodes of discontinuity lines of $u^{\nu}$ corresponding to the initial time, update times and interaction times of physical waves: at those times $p_{k}=\sigma_{i}-\sigma'_{i}-\sigma''_{i}$ is the difference among the strengths of the outgoing $i$-wave $\sigma_{i}$ and the incoming one(s) $\sigma'_{i}$, $\sigma''_{i}$.
We set $0$ the strength corresponding to waves that are not present. 

\item We listed separately the part $\rho^{\nu}_{i}$ of the measure due to nodes on nonphysical fronts, where \[\mu_{i}^{\nu}(\{P\})=\rho^{\nu}_{i}(\{P\})=\widetilde{\ell_{i}}\cdot(u^{R}-u^{L})-\widetilde{\ell_{i}^{'}}\cdot({u'}^{R}-{u'}^{L})\] is the difference before (with prime) and after (without prime) the nonphysical interaction of the projection onto $\widetilde{\ell_{i}}$ of the nonphysical wave $u^{R}-u^{L}$.
The estimate of $\rho$ is substantially the one present in the proof of~\cite[Lemma 5.3]{ACM1} between each update time, since by construction no nonphysical wave is generated at update times.
We stress that, differently from~\cite{ACM1}, in this paper the interaction measure also takes into account the points of interaction among physical and nonphysical waves.
We specify with more detail in Remark~\ref{R:lemma51011} below why~\cite[Lemma 5.3]{ACM1} can be used also with sources satisfying Assumption~\textbf{(G)}.
\end{itemize}

As well, reminding once more that for simplicity we neglect possible errors in the RH speeds, $\mu_{i}^{\nu,\jump}$ is the purely atomic measure given by
\[
\mu_{i}^{\nu,\jump}=\sum_{k}q_{k}\delta_{(t_{k},z_{k})}\,,
\]
where the sum runs on nodes of discontinuity lines of $u^{\nu}$ corresponding to the initial time, update times and interaction times, with both physical and nonphysical waves, and the values $q_{k}$ are studied in~\eqref{E:strenghtq} below.

We now prove uniform bounds that ensure that the measures in the above definition converge indeed to measures that are locally finite on the half-plane. 

\begin{lemma}
\label{L:estSource}
Consider a strictly hyperbolic system~\eqref{eq:syscfSh} where the source term satisfies Assumption~\textbf{(G)} at Page~\pageref{Ass:G} and $F\in C^{2}(\Omega;\R^{n})$.
Then the measures in Definition~\ref{D:measuresnu} satisfy the following bounds: for $0\leq t_{1}\leq t_{2}\leq T$
\bas
&{\mu^{\sour}_{\nu} }((t_{1},t_{2}])
\lesssim
 \pr{{ {\dssb+G T}}+ \norm{\alpha}_{L^{1}}} (t_{2}-t_{1}+\tau_{\nu})\,,
\\
&\abs{\mu_{i}^{\nu}}\pr{(t_{1},t_{2}]\times\R}\lesssim  \mu_{\nu}^{ICS}\pr{(t_{1},t_{2}]\times\R}\,,
&&i=1,\dots,N\,,
\eas 
where $\dssb+GT $ is a bound on the total variation of $u^{\nu}$ given in Theorem~\ref{T:localConv}.
\end{lemma}

Such new estimate can be combined with~\eqref{E:estimatemuIC} bounding $\mu^{IC}$.

\begin{proof}
\firststep
\step{Control of the source measures}
By definition, the source measures are concentrated at times $n\tau_{\nu}$, $n\in\N$: in particular \[{\mu^{\sour}_{\nu} }\left([t_{1},t_{2}]\times\R\right)= {\mu^{\sour}_{\nu} }((n\tau_{\nu}, (n+m)\tau_{\nu}])\,,\] where $n=\inte{t_{1}/\tau_{\nu}}$ and $n+m=\inte{t_{2}/\tau_{\nu}}$.
By estimates~\eqref{E:domt} on the total variation and Assumption \textbf{(G)} at Page~\pageref{Ass:G} on the source, then
\ba
 {\mu^{\sour}_{\nu} }((n\tau_{\nu}, (n+m)\tau_{\nu}])
&\leq 
\sqrt N \sum_{k=1}^{m}\left( \mathcal{V} (u^{\nu}((n+k)\tau_{\nu})) +
\sum_{j\in\N}^{} \abs{q_{j}}\right)\tau_{\nu}\notag
\\
&\leq \sqrt N  \pr{{\pr{\dssb+G(n+m)\tau_{\nu}}}+ \norm{\alpha}_{L^{1}}}m\tau_{\nu}\notag
\\
&\label{E:cicciat}\leq  \sqrt N \pr{{ {\dssb+G T}}+ \norm{\alpha}_{L^{1}}} (t_{2}-t_{1}+\tau_{\nu}) \,.
\ea
\step{Control of the wave-balance measure $\mu_{i}^{\nu} $}
We note that~\cite[Lemma 5.3]{ACM1} was restricted to the case when $g=g(u)$. We can extend it relying on~\cite[Lemma 4.2]{ACM1}: that lemma controls jumps introduced at update times in the fractional step approximation, see~\eqref{E:correctionsource}, because of the discontinuities of an approximated source $g_{\nu}$ which depends also on $x$, as it was defined in~\eqref{E:discretizationSource}. 
We specify it with more detail in Remark~\ref{R:lemma51011} below.
We now outline the relevant steps in such extension.
\begin{itemize}
\item At interaction times $t$ among physical fronts by interaction estimates~\cite[Lemma 1]{AMfr}, generalizing~\cite[(7.98)]{BressanBook}, and by~\eqref{E:muinuexp} one has
\[\abs{\wm_{i}^{\nu}\pr{\{t+\}\times\R}-\wm_{i}^{\nu}\pr{\{t-\}\times\R}}=\abs{\mu_{i}^{\nu}}\pr{\{t\}\times\R}\lesssim\mu^{IC}_{\nu}\pr{\{t\}\times\R} \,.\]
Indeed, at each interaction point $(t,x)$ we have \[\abs{\mu_{i}^{\nu}}\pr{\{(t,x)\}}\lesssim\mu^{IC}_{\nu}\pr{\{(t,x)\}} \,.\]
\item Over interaction times $\{t_{k}\}_{k\in \mathcal N\mathcal P}$ with nonphysical fronts, repeating the argument in~\cite[Page~19]{BCSBV} one gets
\[
\sum_{k\in \mathcal N\mathcal P}\abs{\mu_{i}^{\nu}}\pr{\{t_{k}\}\times\R}\lesssim \varepsilon_{\nu}\,.
\]
At each point, it can as well be controlled by the interaction measure by its definition in~\eqref{E:interactionMeas}.
\item At update times $t$ one can repeat the estimate~\cite[(2.8)]{AG2} thanks to~\cite[Remark 4.3]{ACM1} to get
\bas
 \abs{\wm_{i}^{\nu}(t+)-\wm_{i}^{\nu}(t-)}\pr{ \R}=\abs{\mu_{i}^{\nu}}\pr{\{t\}\times\R}\lesssim & \tau_{\nu}\pr{\abs{\wm_{i}^{\nu} }\pr{\{t-\}\times\R}+ \norm{\alpha}_{L^{1}}}\,.
\eas
More precisely, $ {\mu_{\nu}^{\sour} }$ is defined in such a way that $\abs{\mu_{i}^{\nu}}\pr{\{(t,j\eps_\nu)\}}\lesssim  {\mu_{\nu}^{\sour} }\pr{\{(t,j\eps_\nu)\}\cup\{ (t,(j-1)\eps_\nu)\}}$, $j\in\N$, and if $\wm_{i}^{\nu}(t)$ has an atom at $x\neq j\eps_\nu $ then $\abs{\mu_{i}^{\nu}}\pr{\{(t,x)\}}\lesssim  {\mu_{\nu}^{\sour} }\pr{\{(t,x)\}}$.
\end{itemize}
Since $\mu_{i}^{\nu}$ is concentrated only on such times, summing up in the strip $(t, (k+1)\tau_{\nu}]\times\R$, for any $k\tau_{\nu}\leq t<(k+1)\tau_{\nu}$, by additivity
\bas
\abs{\mu_{i}^{\nu}}&\pr{(t, (k+1)\tau_{\nu}]\times\R}=
\abs{\mu_{i}^{\nu}}\pr{(t, (k+1)\tau_{\nu})\times\R}+
\abs{\mu_{i}^{\nu}}\pr{\{(k+1)\tau_{\nu}\}\times\R}
\\
&\lesssim  \mu^{IC}_{\nu}\pr{(t,(k+1)\tau_{\nu}]\times\R}+
   \abs{\mu_{\nu}^{\sour} }\pr{\{(k+1)\tau_{\nu}\}\times\R}
\\
&\lesssim \mu^{ICS}_{\nu}\pr{(t,(k+1)\tau_{\nu}]\times\R}\,.
\eas
Summing up, if $n\tau_{\nu}\leq t<(n+1)\tau_{\nu}$, by the estimate~\eqref{E:cicciat} obtained in the previous step on the source measure, for $ n,m=0,1,\dots$ we get
\bas
&\abs{\mu_{i}^{\nu}}\pr{(t, (n+m)\tau_{\nu}]\times\R}\lesssim  \mu_{\nu}^{ICS}\pr{(t, (n+m)\tau_{\nu}]\times\R}
\\
&\lesssim {m \tau_{\nu}\pr{\dssb+(n+m)\tau_{\nu}G+\norm{\alpha}_{L^{1}}}}+   \mu_{\nu}^{IC}\pr{(n\tau_{\nu}, (n+m)\tau_{\nu}]\times\R}\;.\qedhere
\eas
\end{proof}

\begin{lemma}
\label{L:fagrgrgggraarg}
Suppose the $i$-th characteristic field in the strictly hyperbolic system~\eqref{eq:syscfSh} is genuinely nonlinear as in Definition~\ref{D:GN} and the source term satisfies Assumption~\textbf{(G)} on Page~\pageref{Ass:G}.
Then $\mu_{i}^{\nu,\jump}$ in Definition~\ref{D:measuresnu} satisfies the following bound: for $0\leq t_{1}\leq t_{2}\leq T$
\bas
-\pr{\dssb+Gt}- \mu_{\nu}^{\sour }\left([0,T]\times\R\right)\lesssim&\mu_{i}^{\nu,\jump}\left([0,T]\times\R\right)\lesssim \mu_{\nu}^{ICS }\left([0,T]\times\R\right)\,,
\eas 
where $\dssb+GT $ is a bound on the total variation of $u_{\nu}$ given in Theorem~\ref{T:localConv}.
\end{lemma}

\begin{corollary}\label{cor:unifBounds}
Suppose the $i$-th characteristic field in the strictly hyperbolic system~\eqref{eq:syscfSh} is genuinely nonlinear as in Definition~\ref{D:GN} and the source term satisfies Assumption~\textbf{(G)} at Page~\pageref{Ass:G}.
For $i=1,\dots,n$
\[
0\leq\mu_{\nu}^{ICJS}([t_{1},t_{2}]\times \R)\lesssim e^{T(\dssb+G T+\norm{\alpha}_{L^{1}})}\left(\mu^{IC}_{\nu}([0,T]\times \R)+\Upsilon^{\nu}(0)\right)\,.
\]
\end{corollary}

In particular, by Definition~\ref{D:measuresnu} and due to uniform bounds in $\nu$ established in Corollary~\ref{cor:unifBounds},  weak$^*$-compactness allows us to assume that $ \mu_{\nu}^{\sour }$ and $\abs{\mu_{i}^{\nu,\jump}}$ are weakly convergent, not only the well-known measures $\mu^{I}_{\nu}$ and $\mu^{IC}_{\nu}$ converging to $\mu^{I}$ and $\mu^{IC}$. As declared in Remark~\ref{R:norelabelingsubs}, also in this case, we do not relabel the subsequence we fix to achieve convergence.
This makes possible giving the following definition.

\begin{definition}
\label{D:measures}
We define the weak$^{*}$-limits of the measures
$\mu_{\nu}^{\sour}$, $\mu_{i}^{\nu}$, $\mu_{i}^{\nu,\jump} $ respectively as the \emph{approximate source measure} $\mu_{\nu}^{\sour} $, the \emph{wave-balance measure} $\mu_{i}$, the \emph{jump-wave-balance measure} $\mu_{i}^{\jump}$.
Summing up, we define the measures
\bas
&\mu^{ICJS}\doteq{}\wlim{\nu}\mu_{\nu}^{ICJS}\;,
&&
\mu^{ICS}\doteq{}\wlim{\nu}\mu_{\nu}^{ICS}\;.
\eas
\end{definition}

\begin{proof}[Proof of Lemma~\ref{L:fagrgrgggraarg}]

$\lozenge$ 
Control of the jump-wave-balance measure. 
A direct computation shows that
\begin{subequations}
\label{E:mujump}
\ba
\mu_{i}^{\nu,\jump} =\sum_{k} q_{k}\delta_{(t_{k},x_{k})} \,,
\ea
where $\{(t_{k},x_{k})\}_{{k}}$ are the nodes in the maximal, leftmost $(\beta_\nu,i)$-approximate discontinuity curves $\J_{\beta,i}(\nu)$ of Definition~\ref{D:approximateDiscC} and the quantities $q_{k}$ are computed as follows. Concerning nodes $(t_{k},x_{k})$ that are interaction points, or where strength and slope either arise or change due to time-updates, denote by $s$ the strength of the outgoing $i$-shock and $s'$, $s''$ the size of the $i$-th wave incoming, if present and belonging in $\J_{\beta,i}(\nu)$: then $q_{k}$ is equal to
\ba\label{E:strenghtq}
\begin{cases}
-s' &\text{terminal nodes of a front not merging into another one,}\\
s &\text{initial nodes of a maximal front, necessarily $\leq0$,}\\
s-s'-s'' &\text{a triple point of $\J_{\beta,i}(\nu)$, corresponding to interaction}\\
&\text{points of waves in $\J_{\beta,i}(\nu)$}\\
s-s' & \text{interaction with a wave not in $\J_{\beta,i}(\nu)$,} 
\\
s-s' & \text{at update times, if not already in the cases above.} 
\end{cases}
\ea
\end{subequations}
By triple point in $\J_{\beta,i}(\nu)$ we mean a point of interaction among two fronts belonging to $\J_{\beta,i}(\nu)$ so that $q_k=s-s'-s''$, while when one of the two interacting waves does not belong in $\J_{\beta,i}(\nu)$ we have $q_k=s-s' $. By genuine nonlinearity when two $i$-shock interact there is an outgoing $i$-shock. 

\textbf{Claim:} The following estimates hold: \[\Upsilon^{\nu}(0)- \mu_{\nu}^{\sour }\left([0,T]\times\R\right)\lesssim\mu_{i}^{\nu,\jump}\left([0,T]\times\R\right)\lesssim \mu_{\nu}^{ICS }\left([0,T]\times\R\right)\,.\]

We prove it considering the various cases listed in~\eqref{E:strenghtq}.

$\lozenge$ 
Neither rarefactions, nor jumps arising because of discontinuities of $g_{\nu}$, are in $\J_{\beta,i}(\nu)$ if we require the bound of Remark~\ref{R:relation}.
At initial points thus $q_{k}\leq0$ since strengths $s$ of shocks are negative, and $q_{k}\leq0$ as well if there is an interaction with an $i$-shock not in $\J_{\beta,i}(\nu)$, just due to interaction estimates.

$\lozenge$ In case of interaction of an $i$-rarefaction front with a physical wave of a different family, including the case of nonphysical waves, or in case of triple points in $\J_{\beta,i}(\nu)$, by interaction estimates we have
\[
q_{k} \lesssim \mu_{\nu}^{IC}(t_{k},x_{k})\;.
\] 
In case of interaction with an $i$-shock of strength $s''<0$ not in $\J_{\beta,i}(\nu)$ we observed $q_{k}\leq0$; we would have as well
\[
q_{k}=s-s'\leq s-s'-s''\lesssim \mu_{\nu}^{IC}(t_{k},x_{k})\;.
\]

$\lozenge$ Terminal points are instead a bit more difficult to handle since the estimate of $q_k$ is related to what happened before time $t_k$, not necessarily to what is happening at $(t_k,x_k)$.
By definition of maximal, leftmost $(\beta_\nu,i)$-approximate discontinuity curve $y$ in Definition~\ref{D:approximateDiscC}, whenever $(t_k,x_k)$ is the terminal point of $ \curva\in\J^{}_{\beta,i}(\nu)$ then $\curva$ at time $t_{k}^{-}$ has strength $|s_{0}|\geq\beta_\nu/4$, but after $t_{k}^{-}$ the strength is less than $\beta_\nu/4$, and there is at least one time $t<t_{k}$ with strength $|s_{1}|\geq\beta_\nu$: thus considering the variation of the strengths along the curve $y$ before the final point we see that
\[
\beta_\nu-\beta_\nu/4<\abs{\mu_{i}^{\nu}}(\mathfrak l)\leq \OL(1)  \mu_{\nu}^{ICS}(\mathfrak l)\,,
\qquad\text{where }\mathfrak l=\curva([t^{-},t^{+}]),
\]
so that 
\begin{equation}\label{e:rgqrrrgqgqrwgqrgqregr}
1\leq \frac{ \OL(1)}{\beta_\nu-\beta_\nu/4}  \mu_{\nu}^{ICS}(\mathfrak l)\,.
\end{equation}
If an $i$-front survives time $t_{k}$ with strength $s>-\beta_\nu/4$, possibly null, so that $\beta_\nu/4+s>0$, then considering the terminal point $(t_k,x_k)$ of the curve we deduce
\bas
0\leq q_{k}&=-s'\leq \beta_\nu/4+s-s'\leq \beta_\nu/4+ \OL(1) \mu_{\nu}^{ICS} (\{(t_{k},x_{k})\})\\
&\stackrel{\eqref{e:rgqrrrgqgqrwgqrgqregr}}{\leq} \frac{\beta_\nu}{4} \cdot\frac{ \OL(1)}{\beta_\nu-\beta_\nu/4} \mu_{\nu}^{ICS}(\mathfrak l)+\OL(1)( \mu_{\nu}^{IC}+\mu_{\nu}^{\sour} )(\{(t_{k},x_{k})\})
\eas
Since the endpoints correspond to disjoint maximal, $(\beta_\nu,i)$-approximate discontinuities we conclude that in the compact set $[0,T]\times[-R,R]$ the following upper estimate holds:
\[
0\leq \sum_{k \text{ endpoints}} q_{k}\leq\OL(1) \mu_{\nu}^{ICS} \left([0,T]\times[-R-\widehat\lambda T,R+\widehat\lambda T]\right)\,.
\]
This concludes the proof of the upper bound \[\mu_{i}^{\nu,\jump }\pr{[0,T]\times\R}\lesssim \mu_{\nu}^{ICS} \pr{[0,T]\times\R}\,.\]

We follow~\cite[Page 464]{BYTrieste} for the lower bound.
We compute the integral of the Lipschitz continuous function $\phi_{\alpha}:\R^{+}\to[0,1]$ defined as
\bas
\phi_{\alpha}(t)=\begin{cases}1 &0\leq t\leq T\\ 1-(t-T)/\alpha & T\leq t\leq T+\alpha\\0&t>T+\alpha \end{cases},
\quad\alpha>0\,,
\eas
in the Radon measure $\overline\mu\doteq{}-\mu_{i}^{\nu,\jump }+\overline C \mu_{\nu}^{ICS} $. Here $\overline C$ thanks to the upper bound on $\mu_{i}^{\nu,\jump }$ we just proved, is chosen indeed in such a way that $\overline\mu$ is nonnegative. From such integral, as $\alpha\downarrow0$ since  $0\leq\phi_{\alpha}\leq1$ we find
\bas
\overline \mu\pr{[0,T]\times\R}\leq\lim_{\alpha\downarrow0}&\left\{
\int_{\R^{+}}\prq{\int_{\R} \pr{\partial_{t}  \phi_{\alpha}+\widetilde{\lambda_{i}^{\nu}}\partial_{x}  \phi_{\alpha}}d\wm_{i}^{\nu,\jump}(t)}dt\right.
\\
&\left.+\int \phi_{\alpha} d\wm_{i}^{\nu,\jump}(0)
+\overline C  \int_{[0,T+\alpha]\times\R}\!\!\!\!\!\!\!\phi_{\alpha}d\mu_{\nu}^{ICS }\right\}
\,.
\eas
Since $\wm_{i}^{\nu,\jump}(0)\leq0$, due to genuine nonlinearity, the second addend in the RHS is nonpositive. In particular computing the limits of the other two addends in the RHS, as $ \wm_{i}^{\nu,\jump}(t)$ is right continuous, we get
\[
\pr{-\mu_{i}^{\nu,\jump }\cancel{+\overline C \mu_{\nu}^{ICS}}  }\pr{[0,T]\times\R}
\leq [ \wm_{i}^{\nu,\jump}(t)]\pr{\R} +\cancel{\overline C \mu_{\nu}^{ICS} \pr{[0,T]\times\R}}\,,
\]
so that by estimates~\eqref{E:domt} on the total variation
\[
-\pr{\dssb+Gt} \lesssim \mu_{i}^{\nu,\jump }\pr{[0,T]\times\R}\lesssim \mu_{\nu}^{ICS} \pr{[0,T]\times\R} \ .
\]

$\lozenge$ 
We finally consider the remaining cases at update times.
Shocks that arise or disappear because of discontinuities of $g_{\nu}$ are treated as initial or terminal points, that we already discussed above; actually, if $\tau_{\nu}$ is sufficiently smaller than $\beta_{\nu}$ shocks are neither created nor destroyed from the family $\J_{\beta,i}(\nu)$, but we do not care of that.
Concerning nodes $(t_{k},x_{k})$ at update times $n\tau_{\nu}$, if we exclude interaction points, initial and terminal points, we are left with points where the slopes of the discontinuity and the left and right values change because the source acts correcting the term $u_{\nu} (\tau_{\nu}-,\cdot)$, according to~\eqref{E:correctionsource}. 
By~\cite[Remark 4.3]{ACM1}, denoting by $s$ the strength of the outgoing $i$-wave and by $s'$ the size of the $i$-th wave incoming
\bas
&\abs{s-s'}\lesssim \tau_{\nu}\abs{s'}
\qquad \Longrightarrow
\qquad\abs{q_{k}}\lesssim \mu_{\nu}^{\sour}(t_{k},x_{k})\,.\qedhere
\eas
\end{proof}

 We exploit the introduction of the source measures explaining that, with this tool, \cite[Lemmas~5.3;10,11]{ACM1} and be extended under Assumption~\textbf{(G)} instead of only considering $g=g(u)$ Lipschitz continuous.

\begin{remark}\label{R:lemma51011}
We consider a plygonal region $\Gamma$ with edges transversal to the waves it encounters. For $k\in\{1,\dots,N\}$, the incoming flux $W^{\nu,k}_\mathrm{in}$ and outgoing flux $W^{\nu,k}_\mathrm{of}$ of the $k$-waves through the boundary of the region are defined as~\cite[Page~359]{ACM1}, as the corresponding amount of positive $k$-waves $W^{\nu,k,+} (t)$ or negative $k$-waves $W^{\nu,k,-} (t)$ present in a $t$-section of $\Gamma$.
Assume also, just for simplicity, that no singular point of $\mu_\nu^S$ lies on the boundary of $\Gamma$.

When the source satisfies Assumption~\textbf{(G)}, in Step 2 of the proof of \cite[Lemma~5.3, Page~360]{ACM1}, applying the same \cite[Remark~4.3]{ACM1} as there, thanks to the introduction of source measures, one immediately has the worse estimates
\begin{gather*}
\abs{W^{\nu,k,+} (t+)-W^{\nu,k,+} (t-)}\lesssim \tau_\nu W^{\nu,k,+} (t-)+\mu_\nu^\sour(\Gamma\cap\{t\}\times\R)\,,
\\
\abs{W^{\nu,k,-} (t+)-W^{\nu,k,-} (t-)}\lesssim \tau_\nu W^{\nu,k,-} (t-)+\mu_\nu^\sour(\Gamma\cap\{t\}\times\R)\,.
\end{gather*}
As a consequence, the statement of~\cite[Lemma~5.3, Page~360]{ACM1} tells
\begin{gather*}
  e^{-\tau_{(k-h)\nu C}}\left[W^{\nu,k,+}_\mathrm{in} \!\!\!- C \mu_\nu^{ICS}(\overline\Gamma)\right]\leq W^{\nu,k,+}_\mathrm{in}\!\!\!\leq e^{\tau_{(k-h)\nu C}}\left[W^{\nu,k,+}_\mathrm{in} \!\!\!+C \mu_\nu^{ICS}(\overline\Gamma)\right],
  \\
  e^{-\tau_{(k-h)\nu C}}\left[W^{\nu,k,-}_\mathrm{in}\!\!\! - C \mu_\nu^{ICS}(\overline\Gamma)\right]\leq W^{\nu,k,-}_\mathrm{in}\!\!\!\leq e^{\tau_{(k-h)\nu C}}\left[W^{\nu,k,-}_\mathrm{in}\!\!\! +C \mu_\nu^{ICS}(\overline\Gamma)\right].
  \end{gather*}
  
  Concerning~\cite[Lemmas~5.10-11]{ACM1}, what slightly changes is the estimate of the slopes of $k$-characteristics $\gamma$:
  \[
  \abs{\gamma(t+)-\gamma(s-)} \lesssim \abs{t-s} + M_k^\nu(\gamma;s,t)+ M_*^\nu(\gamma;s,t)+ \mu_\nu^{ICS}(\{(r,\gamma(r))\}_{s\leq r\leq t})\,,
    \]
    where $M_k^\nu(\gamma;s,t)$ is the total strength of all $k$-waves interacting with $\gamma$ at times in $[s,t]$, while $ M_*^\nu(\gamma;s,t)$ is the total strength of all $j$-waves interacting with $\gamma$ at times in $[s,t]$ for $j\neq k$.
   Since $ \mu_\nu^{ICS}(\{(r,\gamma(r))\}_{s\leq r\leq t})$ is the sum of $ \mu_\nu^{IC}(\{(r,\gamma(r))\}_{s\leq r\leq t})$, already present in such estimate when $g=g(u)$, and $ \mu_\nu^{\sour}$, controlled by $ \pr{{ {\dssb+G T}}+ \norm{\alpha}_{L^{1}}} (t-s+\tau_{\nu}) $ in Lemma~\ref{L:estSource}, then the arguments of such lemmas can be repeated enlarging accordingly the constants in the inequalities.
\end{remark}

 We conclude this auxiliary section repeating that the number of $(\beta,i)$-approximate discontinuities is uniformly bounded.

 \begin{lemma}\label{L:boundsOnJumps}
 When the threshold $\beta$ is fixed, then the cardinality $\sharp \J^{}_{\beta,i}(\nu)$ of $(\beta,i)$-approximate discontinuities, up to a fixed time $T$, is uniformly bounded in $\nu$ and of order $M_{\beta,i}\lesssim \beta^{-2}$.
 \end{lemma}
 
 \begin{proof}
Let $\gamma:[t^{-},t^{+})\to \R^{2}$ be a curve along which $u^{\nu}$ is discontinuous, let $s_{i}$ be its strength.
For any $[t_{1},t_{2}]\subseteq(t^{-},t^{+})$, by the estimates on interactions, on the action of the source, see Remark~\ref{R:timeupdate} and the definition of the source measure,
\[
|s(t_{2}{+})|-|s(t_{1}{-})|\lesssim\mu_{\nu}^{ICS} \left(\gamma[t_{1},t_{2}]\right)+\beta_{\nu}\,.
\]

When $\gamma(t_{2}) $ belongs in a leftmost $(\beta,i)$-approximate discontinuity originating a time $0$ then $s(0)\geq \beta/4$, thus there are at most $\TV(u(0,x))/\beta$ of them.

When $\gamma(t_{2}) $ belongs in a leftmost $(\beta,i)$-approximate discontinuity originating at positive time $t_{1}$ where $\abs{s(t_{1-})}\leq\beta/4$ and contemporarily $t_{2}$ is one of such times of the $(\beta,i)$-approximate discontinuity satisfying $s(t_{2})\geq\beta$, then for $\nu$ large
\[
\beta/2\lesssim \mu_{\nu}^{ICS} \left(\gamma[t_{1},t_{2}]\right) \,.
\]
We conclude that there are at most $C\mu^{ICS}([0,T]\times\R)/\beta$ of them since they are disjoint in the plane.
 \end{proof}

\subsection{Balances on characteristic regions}
\label{S:balances-}

\begin{subequations}
Generalized $i$-characteristic curves of $u^{\nu}$ are curves which are Lipschitz continuous and with speed which is given, at almost every time, by $\lambda_{i}(u^{\nu})$ at continuity points of $u^{\nu}$, or by the speed of the jump at $i$-discontinuities of $u$.

Consider generalized, order preserving $i$-characteristic curves of $u^{\nu}$:
\ba
&x_{m}^{\nu,\ell},x_{m}^{\nu,r}:[0,T]\to\R\,,\quad \text{ for }m=1,\dots,k \text{ with $k\in\N$ fixed,}
\ea
such that for $0\leq t\leq T$ they satisfy
\ba\label{E:orderCurves}
&x_{1}^{\nu,\ell}(t)\leq x_{1}^{\nu,r}(t)\leq x_{2}^{\nu,\ell}(t)\leq\dots\leq x_{k-1}^{\nu,r} (t)\leq x_{k}^{\nu,\ell}(t)\leq x_{k}^{\nu,r}(t)
\,.
\ea
If $J=\cup_{m=1}^{k}[x_{m}^{\ell},x_{m}^{r}]$ is the union of $k$ disjoint ordered intervals, thus $x_{m}^{r}<x_{m+1}^{\ell}$ for $m=1,\dots,k-1$, a common choice consists in the leftmost $i$-th characteristics $ x(t;t_{0},x_{m}^{\ell})$, $ x(t;t_{0},x_{m}^{r})$ as in~\eqref{E:ab}.
We stress that we do not need to assume this choice.

Following~\eqref{E:iCharacteristicRegion} we define for $0\leq s<t\leq T$ the region, see Fig.~\ref{fig:balance},
\ba\label{E:ffdgagwagfffff}
A_{\{x_{m}^{\nu,\ell/r}\}}^{s,t}\doteq{}
\bigcup_{m=1}^k\left\{(\zeta ,x)\ :\ s<\zeta\leq t \,,\ x_{m}^{\nu,\ell}(t)\leq x\leq x_{m}^{\nu,r}(t)\right\}\;.
\ea
\end{subequations}

We recall from Remark~\ref{R:relation} that the threshold $\beta_{\nu}$ can be chosen larger than $\varepsilon_{\nu}$, $\tau_{\nu}$, and we assume this choice.

 
\begin{lemma}\label{L:balance0}
Suppose the $i$-th family is genuinely nonlinear: then, with the above notation, for $0\leq s<t\leq T$ and setting $J (r)=\cup_{m=1}^{k}[x_{m}^{\nu,\ell}(r),x_{m}^{\nu,r}(r)]$,
\ba
\label{E:contest} \left[ \wm_{i}^{\nu,\cont}(t)\right](J (t))-
\left[ \wm_{i}^{\nu,\cont}(s)\right](J(s))
&\lesssim \left(\mu_{\nu}^{ICJS} \right)\left(A_{\{x_{m}^{\nu,\ell/r}\}}^{s,t}\right)+ \varepsilon_\nu\,,
\ea\ba
\label{E:contest2}
\left[ \wm_{i}^{\nu,\cont}(t)\right](J (t))-
\left[ \wm_{i}^{\nu,\cont}(s)\right](J(s)) &\gtrsim
 -\left(\mu_{\nu}^{ICS} \right)\left(A_{\{x_{m}^{\nu,\ell/r}\}}^{s,t}\right)-k \beta_\nu\,,
\\\label{E:allest}
 \left[ \wm_{i}^{\nu}(t)\right](J (t))-
\left[ \wm_{i}^{\nu}(s)\right](J(s))
&\lesssim \left(\mu_{\nu}^{ICS} \right)\left(A_{\{x_{m}^{\nu,\ell/r}\}}^{s,t}\right)+ \varepsilon_\nu\,.
\ea
\end{lemma}

 \begin{figure}\centering
\includegraphics[width=.6\linewidth]{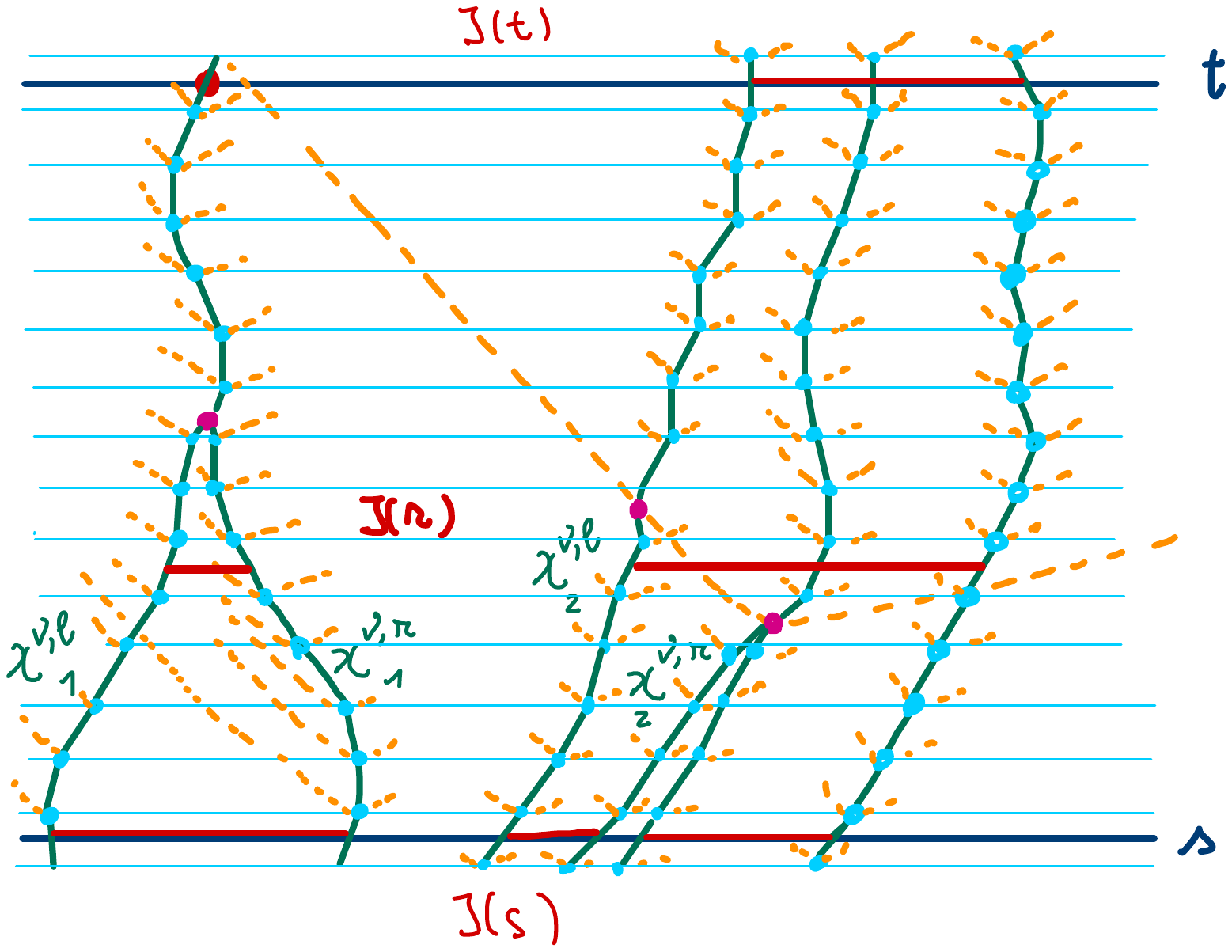}
\caption{Highlighted in bold, there are $t$-sections $J(t)$ of the regions $A_{\{x_{m}^{\nu,\ell/r}\}}^{s,t}$ of Lemma~\ref{L:balance0}. Horizontal thin full lines are update times. Dashed lines evoke discontinuities of $u^{\nu}$ different from the curves $\{x_{m}^{\nu,\ell/r}\}$, that might produce interactions and that arise also at update times and interaction times. Dots also highlight some interactions.}
\label{fig:balance}
\end{figure}

 \begin{proof}
 We recall that the strength of nonphysical fronts remains uniformly bound\-ed by $\eps_{\nu}$ as in~\eqref{npbound}, and such nonphysical fronts have constant speed $\widehat\lambda$.
Neglecting them, and due to the definition of the piecewise-constant approximation $u^{\nu}$ and by Definition~\ref{D:approximateDiscC} of the $i$-wave measure of $u^{\nu}$, the wave-measures $\wm_{i}^{\nu}$ are concentrated on (approximate) generalized characteristics; the speeds $\widetilde{\lambda_{i}^{\nu}}$ of such characteristics change only due to interactions or at update times: in particular the wave-balance measures $\mu_i^{\nu}$ of Definition~\ref{D:measuresnu} vanish on time intervals which do not contain neither interaction times nor update times. 
 
\emph{ From Step 1 up to Step 4 we assume that $ x_{m}^{\nu,r}(\zeta) < x_{m+1}^{\nu,\ell}(\zeta)$, for $m=1,\dots,k-1$ and $s<\zeta<t$.}
 
  \firststep\step{Setting up notation}
 For the rest of the proof, since $\nu$, $s$, $t$, $J$ are fixed, we omit them: 
 \bas
&A\doteq{}A_{\{x_{m}^{\nu,\ell/r}\}}^{s,t}\,,
\quad
 x_{m}^{\ell}(\zeta)=x_{m}^{\nu,\ell}(\zeta)\,,
 \quad
  x_{m}^{r}(\zeta)=x_{m}^{\nu,r}(\zeta)\,,
  \quad
  m=1,\dots,k\ .
\eas
\emph{We repeat that for the time being $ x_{m}^{r}(\zeta)< x_{m+1}^{\ell}(\zeta)$, for $m=1,\dots,k-1$ and $0\leq s<\zeta<t\leq T$}, so that the different interval stay separated in their evolution. Define the lateral boundaries by\bas
&\Lf_{\out}\doteq{}\bigcup_{m=1}^{k}\left\{(\zeta ,\ x_{m}^{\ell}(\zeta))\cup (\zeta,\ x_{m}^{r}(\zeta))\ :\ s\leq \zeta < t   \right\}\,,
\\
&\Lf_{\rin}\doteq{}\bigcup_{m=1}^{k}\left\{(\zeta ,\ x_{m}^{\ell}(\zeta))\cup (\zeta ,\ x_{m}^{r}(\zeta))\ :\ s< \zeta \leq t   \right\}
\;.
\eas

For time intervals $(s,t)$, possibly including interaction times or update times, we define the $i$-fluxes  $\Phi_{i,\rin}^{\nu}$ and $\Phi_{i,\out}^{\nu}$, entering and exiting the region across the lateral boundary, so that the following balance hold:
\bas
 \left[ \wm_{i}^{\nu}(t)\right](J(t))-
\left[ \wm_{i}^{\nu}(s)\right](J(s))
= \mu_i^{\nu} \left(A \right)+\Phi_{i,\rin}^{\nu}\left(\Lf_{\rin}\right)+\Phi_{i,\out}^{\nu}\left(\Lf_{\out}\right)\,.
\eas
By definition of $ \mu_i^{\nu} $ and $\wm_{i}^{\nu}(t)$, as $u(t)$ is right continuous, more precisely, there is a contribution equal to
\begin{itemize}
\item $-\sigma$ in $\Phi_{i,\out}^{\nu}$ whenever an $i$-wave of strength $\sigma$ is leaving the region at some time in $[s,t)$, 
\item $+\sigma$ in $\Phi_{i,\rin}^{\nu}$ whenever an $i$-wave of strength $\sigma$ is entering the region at some time in $(s,t]$. 
\end{itemize}
Of course the above situation might happen contemporarily at the same points and also with more waves. 

Similarly, we define the $i$-flux $\Phi_i^{\nu,\jump} $ due to $(\beta_\nu,i)$-approximate discontinuities of Definition~\ref{D:approximateDiscC}, not anymore due to all $i$-waves, again across the lateral boundaries of the region, so that
\bas 
\left[ \wm_{i}^{\nu,\jump}(t)\right](J(t))-
\left[ \wm_{i}^{\nu,\jump}(s)\right](J(s))
=&\mu_i^{\nu,\jump} \left(A \right)+
\\
&+\Phi_{i,\rin}^{\nu,\jump}\left(\Lf_{\rin}\right)
+\Phi_{i,\out}^{\nu,\jump}\left(\Lf_{\out}\right)\,.
\eas
We remind the expression $\eqref{E:mujump}$ of $\mu_i^{\nu,\jump} $.
More precisely, there is a contribution equal to
\begin{itemize}
\item $-\sigma$ in $\Phi_{i,\out}^{\nu,\jump}$ whenever a $(\beta_\nu,i)$-approximate discontinuity of strength $\sigma$ leaves the region in $[s,t)$, 
\item $+\sigma$ in $\Phi_{i,\rin}^{\nu,\jump}$ whenever a $(\beta_\nu,i)$-approximate discontinuity of strength $\sigma$ enters the region in $(s,t]$. 
\end{itemize}

Subtracting the balances for $\wm_{i}^{\nu}$ and for $\wm_{i}^{\nu,\jump}$ we get
\bas
\left[ \wm_{i}^{\nu,\cont}(t)\right](J(t))-
\left[ \wm_{i}^{\nu,\cont}(s)\right](J)
=&\mu_i^{\nu,\cont} \left(A \right)+
\\ &+\Phi_{i,\rin}^{\nu,\cont}\left(\Lf_{\rin}\right)+\Phi_{i,\out}^{\nu,\cont}\left(\Lf_{\out}\right)\,.
\eas
where $\Phi_{i,\rin}^{\nu,\cont} $ and $\Phi_{i,\out}^{\nu,\cont} $ finally denote the $i$-flux across the lateral boundary of the region only due to $i$-waves of $u^{\nu}$ which are not maximal leftmost $(\beta,i)$-approximate discontinuities of Definition~\ref{D:approximateDiscC}.
More precisely, there is a contribution
\begin{itemize}
\item $-\sigma$ in $\Phi_{i,\out}^{\nu,\cont}$ whenever an $i$-wave of strength $\sigma$---which is not a $(\beta_\nu,i)$-approximate discontinuity according to Definition~\ref{D:approximateDiscC}---is leaving the region at some time in $[s,t)$, 
\item $+\sigma$ in $\Phi_{i,\rin}^{\nu,\cont}$ whenever an $i$-wave of strength $\sigma$---which is not a $(\beta_\nu,i)$-approximate discontinuity according to Definition~\ref{D:approximateDiscC}---is entering the region at some time in $(s,t]$. 
\end{itemize}

\step{Upper estimate for the balance of the continuous part}
The thesis~\eqref{E:contest} follows just estimating from above in terms of $\mu^{ICS}$ the continuous flux $\Phi_{i}^{\nu,\cont}$ across the lateral boundary of $A $, introduced at the end the previous step.

There are a large number of possible configurations for waves entering and exiting, since to all the possibilities of the homogeneous case in~\cite[Page 25]{BCSBV} we need to add waves updated, created or cancelled when correcting the term $u_{\nu} (\tau_{\nu}-,\cdot)$ according to~\eqref{E:correctionsource}. In this latter cases, by~\cite[Remark 4.3]{ACM1} the source measures controls the change in speed and strength of waves, including the strength of new ones arising or the old ones cancelled, if any.

The key to control the flux $\Phi_{i}^{\nu,\cont}$ in terms of the measure $\mu^{ICS}$ is again exploiting the genuine nonlinearity: from above, we need to estimate the total strength of \emph{positive $i$-waves entering the region} and of \emph{negative waves exiting the region}. Namely: 
\begin{itemize}
\item A positive $i$-wave, namely an $i$-rarefaction wave, might enter the region at $P$ only if the boundary it crosses is a negative $i$-wave, since such boundary is an $i$-approximate characteristic: cancellation occurs and thus \[\Phi_i^{\nu,\cont}\left(P\right) \lesssim\mu^{IC}_{\nu}(P) \ .\]
\item A positive $i$-wave, namely an $i$-rarefaction wave, coming from outside might be cancelled on the boundary at $P$ because of the update. In such case $\mu_i^{\nu,\cont} (P)=\left[\wm_{i}^{\nu,\cont}(t^{+})-\wm_{i}^{\nu,\cont}(t^{-})\right](\R)$: then in particular
\[
\Phi_i^{\nu,\cont}\left(P\right) \lesssim \abs{\mu_{\nu}^{\sour}}(P)\ .
\]

 \item A negative $i$-wave cannot exit the region, because of the entropy condition, since it is an $i$-shock.

 \item We neglect nonphysical waves for a reduction argument as in~\cite[Page 24]{BCSBV}. 
 \end{itemize}
Summarizing, we have that $\Phi_i^{\nu,\cont}\pr{\mathfrak L}\lesssim\mu^{ICS}_{\nu}(\mathfrak L)$, where $\mathfrak L$ is the lateral boundary of $A_{\{x_{m}^{\nu,\ell/r}\}}^{s,t}$. By the upper estimate on $\mu_{i}^{\nu}$ in Lemma~\ref{L:fagrgrgggraarg} and the lower estimate on $\mu_{i}^{\nu,\jump}$, both with $\mu_{\nu}^{ICJS}$, we get the upper estimate on $\left[ \wm_{i}^{\nu,\cont}(t)\right](J(t))-
\left[ \wm_{i}^{\nu,\cont}(s)\right](J)$ in thesis~\eqref{E:contest}.

\step{Lower estimate \eqref{E:contest2} for the balance of the continuous part}
We need to estimate the negative $i$-waves entering the region and the positive $i$-waves exiting the region.
From below the estimate is less easy:
\begin{itemize}
\item A positive $i$-wave, namely an $i$-rarefaction wave, might exit the region at $P$ only if it is generated at $P$. This can either happen if there is an interaction, or a cancellation, or an update at $P$: summing up either by classical interaction estimates or by~\cite[Remark 4.3]{ACM1} the interaction-cancellation-source measures at the point controls the change in speed and strength of waves present, including the strength of new ones arising or the old ones cancelled: then in particular
\[
\abs{\Phi_i^{\nu,\cont}\left(P\right)} \gtrsim \mu_{\nu}^{ICS}(P)\ .
\]

\item A negative $i$-wave, namely an $i$-shock wave, can enter the region simply because any of the $2k$ lateral $i$-characteristics merges into the previously external shock. We do not see a way to estimate this situation other than observing that it might happen only once for each lateral side, since once an $i$-characteristic merges into a shock it must always follow such shock later on by genuine nonlinearity, and the shock cannot be cancelled. Counting the curves of the boundary, and waves in the continuous part have strength $\geq-\beta_{\nu}$, we get estimate~\eqref{E:contest2}, where we stress that $k$ could be replaced by the number of times any of the $2k$ lateral $i$-characteristics merges into an external shock within the interval $(s,t]$.
Indeed, in Remark~\ref{R:relation} we noticed that we can choose $\varepsilon_{\nu}\leq \OL(1)\beta_{\nu}$.

 \item A negative $i$-wave can enter the region also if there is an interaction.
 In case the interaction takes place with a $(\beta,i)$-approximate discontinuity at a boundary, at time $t^{*}$, this gets reinforced so that it is not present in the measure $\wm_{i}^{\nu}(t^{*}+)$, $\wm_{i}^{\nu}(t^{*}-)$; also the negative $i$-wave that is trying to enter the region, being out of the region before the interaction takes place, does not contribute to estimate~\eqref{E:contest2}.
In case the interaction takes place with a wave that is present in $\wm_{i}^{\nu}(t^{*}-)$, by classical interaction estimates \[\abs{\Phi_i^{\nu,\cont}\left(P\right)} \gtrsim\mu^{IC}_{\nu}(P) \ .\]

 \item  We recall that we neglect nonphysical waves by a reduction argument as in~\cite[Page 24]{BCSBV}. 
 \end{itemize}

\step{Upper estimate for the balance of full waves~\eqref{E:allest}} 
The argument for~\eqref{E:allest} is similar to the one for~\eqref{E:contest}starting from the balance for $\wm_{i}^{\nu}$, being $ \wm_{i}^{\nu}= \wm_{i}^{\nu,\cont}+ \wm_{i}^{\nu,\jump}$ and similarly for the fluxes.
We need to estimate the positive $i$-waves entering the region and the negative $i$-waves exiting the region.
We remind Lemma~\ref{L:fagrgrgggraarg} that guarantees that $\mu_{i}^{\nu}$ is controlled by $ \mu_{\nu}^{ICS }$.

\begin{itemize}
\item A positive $i$-wave of strength $\sigma$, namely an $i$-rarefaction wave, might enter the region at $P$ only if the boundary it crosses is a negative $i$-wave, since such boundary is an $i$-approximate characteristic: cancellation occurs and thus 
\[\left[ \wm_{i}^{\nu}(t^{+})\right](J(t))-
\left[ \wm_{i}^{\nu}(t^{-})\right](\ri(J(t)))
 -\mu_i^{\nu} (P)= \sigma=\Phi_i^{\nu}\left(P\right) \lesssim\mu^{IC}_{\nu}(P) \ .\]

 \item A positive $i$-wave of strength $\sigma$, namely an $i$-rarefaction wave, coming from outside might be cancelled on the boundary at $P$ because of the update. In such case $\mu_i^{\nu} (P)=\left[\wm_{i}^{\nu}(t^{+})-\wm_{i}^{\nu}(t^{-})\right](\R)$: then in particular
\[
\sigma=\Phi_i^{\nu}\left(P\right) \lesssim \abs{\mu_{\nu}^{\sour}}(P)\ .
\]

 \item A negative $i$-wave cannot exit the region, because of the entropy condition, since it is an $i$-shock.

 \item We neglect nonphysical waves for a reduction argument as in~\cite[Page 24]{BCSBV}. 
 \end{itemize}
Summarizing, we have that $\Phi_i^{\nu}\pr{\mathfrak L}\lesssim\mu^{ICS}_{\nu}(\mathfrak L)$, where $\mathfrak L$ is the lateral boundary of $A_{\{x_{m}^{\nu,\ell/r}\}}^{s,t}$. By the upper estimate on $\mu_{i}^{\nu}$ in Lemma~\ref{L:fagrgrgggraarg} we get~\eqref{E:allest}.

\step{Conclusion with more generality}
Consider now curves ordered as in~\eqref{E:orderCurves} without requiring $ x_{m}^{r}(\zeta)<x_{m+1}^{\ell}(\zeta)$, for $m=1,\dots,k-1$ and $0\leq s<\zeta<t\leq T$: indeed, $ x_{m}^{r}$ and $x_{m+1}^{\ell}$ could meet and then later on move far apart many times.
Still, working with piecewise constant approximations, it will be only finitely many times.
In particular we can define finitely many times $s  =  t^{\heart}_{0}\leq t^{\heart}_{1}\leq  \dots \leq  t^{\heart}_{q-1}\leq  t^{\heart}_{q}=t$, cardinalities $k_1,\dots,k_q$, and new monotone indices $h_1^q,\dots,h_m^q\in\{1,\dots,m\}$ so that  $ x_{h_m^q}^{r}<x_{h_{m+1}^q}^{\ell}$ in the interval $(t^{\heart}_{j},t^{\heart}_{j+1} )$ for $m=1,\dots,k_q-1$ and $j=0,\dots ,q-1$.
We thus already proved the balances of thesis referred to each interval $(t^{\heart}_{j},t^{\heart}_{j+1} )$. 
We perform a telescopic summation: we thus obtain, for example studying estimate~\eqref{E:contest2},
\bas
\left[ \wm_{i}^{\nu,\cont}(t)\right](J (t))-
\left[ \wm_{i}^{\nu,\cont}(s)\right](J(s)) =&
\sum_{j=0}^{q-1}\left(\left[ \wm_{i}^{\nu,\cont}(t^{\heart}_{j+1} )\right](J (t^{\heart}_{j+1} ))\right.+
\\
&\left.\quad-
\left[ \wm_{i}^{\nu,\cont}(t^{\heart}_{j})\right](J(t^{\heart}_{j} ))\right)\,,
 \eas
\bas
\left[ \wm_{i}^{\nu,\cont}(t^{\heart}_{j+1} )\right](J (t^{\heart}_{j+1} ))-
\left[ \wm_{i}^{\nu,\cont}(t^{\heart}_{j})\right](J(t^{\heart}_{j} ))\geq 
 & -C\left(\mu_{\nu}^{ICS} \right)\left(A_{\{x_{m}^{\nu,\ell/r}\}}^{t^{\heart}_{j},t^{\heart}_{j+1} }\right)+
  \\
  &-Ck_{j} \beta_\nu\,,
 \eas
 where $k_{j}$ denotes the number of times a lateral $i$-characteristic merges into the external shock in the interval $(t^{\heart}_{j},t^{\heart}_{j+1}]$.
Nevertheless, even with the configuration when boundary characteristics merge and later split the number of times a lateral $i$-characteristic merges into the external shock in the interval $(s,t)$ is still bounded by $k$, for the same reasons: thus $\sum_{j=0}^{q-1}k_{j}\leq k$.
Observing that the regions $A_{j}=A_{\{x_{h_m^j}^{\nu,\ell/r}\}}^{t^{\heart}_{j},t^{\heart}_{j+1} }$ are disjoint, by additivity of measures
 \[
 \sum_{j=0}^{q-1}\left(\mu_{\nu}^{ICS} \right)\left(A_{j}\right)=\left(\mu_{i}^{\nu,ICS} \right)\left(\bigcup_{j=0}^{q-1}A_{j}\right)
 =\left(\mu_{\nu}^{ICS} \right)\left(A_{\{x_{m}^{\nu,\ell/r}\}}^{s,t  }\right)
 \]
so that we get~\eqref{E:contest2} in the generality we stated it:
  \bas
  \left[ \wm_{i}^{\nu,\cont}(t)\right](J (t))-
\left[ \wm_{i}^{\nu,\cont}(s)\right](J(s)) 
&\geq 
-C\sum_{j=0}^{q-1}\left(\mu_{\nu}^{ICS} \right)\left(A_{j }\right)-C\sum_{j=0}^{q-1}k_{j}\beta_{\nu}
\\&\geq
-C\left(\mu_{\nu}^{ICS} \right)\left(A_{\{x_{m}^{\nu,\ell/r}\}}^{s,t  }\right)-C k\beta_{\nu}\,.
  \eas
The telescopic series works similarly for estimates~\eqref{E:contest},~\eqref{E:allest}.
 \end{proof}

\subsection{Decay estimate for the negative continuous part}
\label{Ss:decay-}

We keep the notation introduced in \S~\ref{S:balances-}, and in particular the relation in Remark~\ref{R:relation} for the vanishingly small parameters. 

Let $i\in\{1,\dots,N\}$.
 
\begin{lemma}\label{L:grgaaear}
Let $0\leq s\leq t\leq T$. Suppose $J=\cup_{m=1}^{k}[x_{m}^{\nu,\ell}(s),x_{m}^{\nu,r}(s)]$ is the union of $k$ disjoint intervals, where the generalized $i$-characteristics are as in~\eqref{E:orderCurves} and the region $A_{\{x_{m}^{\nu,\ell/r}\}}^{s,t}$ is defined in~\eqref{E:ffdgagwagfffff}.
Then
\ba\label{E:thesisnfakvaargrarw}
&- {\left[ \wm_{i}^{\nu,\cont}(s)\right](J)} \lesssim\frac{\Ll^{1}(J)}{t-s}+\left(\mu_{\nu}^{ICJS} \right)\left(A_{\{x_{m}^{\nu,\ell/r}\}}^{s,t}\right)+k  \beta_{\nu} 
\ea
for all $i=1,\dots,N$.
\end{lemma}
This part is an extension of \S 6.2.4, Page 466, in~\cite{BYTrieste}, relative to $1d$ conservation laws.

\begin{proof}
For simplicity of notations, by the semigroup property assume that $s=0$; denote $t^{*}$ the final time $t$. 

Consider at time $s=0$ any closed interval $I=[a,b]$ of which $J$ is made of, for example $[x_{m}^{\nu,\ell}(0),x_{m}^{\nu,r}(0)]$. Set
%
\[
 a(t)\doteq{}x_{m}^{\nu,\ell}(t)\ , \qquad b(t)\doteq{}x_{m}^{\nu,r}(t)
 \qquad \text{for $0\leq t\leq t^{*}$.}
\]
For $t\in[0,t^{*}]$, set $I(t)=[a(t),b(t)]$ and $z(t)=\Ll^{1}(I(t))=b(t)-a(t)$. Then by definition
\[
\dot z(t)=\dot b(t)-\dot a(t)=\widetilde{\lambda_{i}^{\nu}}(t,b(t))-\widetilde{\lambda_{i}^{\nu}}(t,a(t))\ .
\]
where $\widetilde{\lambda_{i}^{\nu}}$ is the speed of the $i$-th wave of $u^{\nu}$ in Definition~\ref{def:velCarAppr}.
From now on we omit $\nu$ and denote $\widetilde{\lambda_{i}^{\nu}}$ by $\widetilde{\lambda_{i}}$.
\\
In~\cite[Page 213]{BressanBook} it is introduced the following explicit \emph{piecewise} Lipschitz continuous function
\begin{subequations}\label{E:PhiAlberto}
\ba
&\Phi(t)\doteq{}\sum_{k_{\alpha}\neq i}\varphi_{k_{\alpha}}(t,x_{\alpha}(t))\cdot \abs{\sigma_{\alpha}}
\ea
denoting by $\sigma_{\alpha}$ the signed size of the $k_{\alpha}$-th wave of $u^{\nu}$ at $(t,x_{\alpha}(t))$,
where for $j\neq i$ it is defined
\ba
&\varphi_{j}(t,x)
=
\begin{cases}
\uno_{j<i} &\text{if $x<a(t)$,}\\
\rho_{j}(t) &\text{if $a(t)\leq x< b(t)$,}\\
\uno_{i>j}  &\text{if $x\geq b(t)$,}
\end{cases}
&&
\rho_{j}(t)
=
\begin{cases}
\displaystyle\frac{b(t)-x}{b(t)-a(t)} & \text{if $j<i$,}\\
\displaystyle\frac{x-a(t)}{b(t)-a(t)} & \text{if $j>i$,}
\end{cases}
\ea
\end{subequations}
where $\uno_{j<i} $ takes value $1$ if $j<i$ and $0$ otherwise and $\uno_{i<j} $ takes value $1$ if $i<j$ and $0$ otherwise.

While in~\cite{BressanBook} the function $\Phi$ can jump only at interaction times, here $\Phi$ can jump also at update times: at interaction times as before the variation of $\Phi$ is controlled by the Glimm function~\cite[(10.51)]{BressanBook}, while at update times it is controlled by the source measure by~\cite[Remark 4.3]{ACM1}.
As before, the derivative $\dot\Phi$  satisfies $\dot\Phi\geq0$ where differentiable, which is excluding interaction and update times, and $\Phi$ controls the variation of the $i$-wave due to waves of other families in this quantitative way:
\ba\label{est:Bressan}
\abs{\dot z(t)+\xi(t)-\left[\wm_{i}^{\nu}(t)\right](I(t))}\leq \dot\Phi(t)z(t)+\OL(1)\varepsilon_{\nu}\ ,
\ea
where again we are excluding interaction and update times and we have defined
\[
\xi(t)\doteq{}\pr{\widetilde{\lambda_{i}}(t,a(t ))-\widetilde{\lambda_{i}}(t,a(t-))}+\pr{\widetilde{\lambda_{i}}(t,b(t+))-\widetilde{\lambda_{i}}(t,b(t) )}\ .
\]
Since by its very definition $\Phi(t)$ is controlled by the total variation of $u^{\nu}(t)$, then $\int_{0}^{r}\dot\Phi $ is uniformly bounded at finite times and it can be estimated in terms only of the total variation of $u^{\nu}(t)$, the Glimm function and the source measure, in turn estimated in Lemma~\ref{L:estSource}.

Suppose $\left[ \wm_{i}^{\nu,\cont}(0)\right](I)<0$, otherwise thesis is trivial.
Two cases are possible:

\paragraph{Case 1} Suppose, in addition to $\left[ \wm_{i}^{\nu,\cont}(0)\right](I)<0$, that
\bas
\dot z(r)-\dot\Phi(r)z(r)<\frac{\left[ \wm_{i}^{\nu,\cont}(0)\right](I)}{4}<0
&&\forall r\in(0,t^{*})\ .
\eas
By elementary calculus, multiplying this assumption by $e^{-\int_{0}^{r}\dot\Phi }>0$ we arrive to
\bas
\ddr\pr{z(r)e^{-\int_{0}^{r}\dot\Phi }}<\frac{\left[ \wm_{i}^{\nu,\cont}(0)\right](I)}{4}\cdot e^{-\int_{0}^{r}\dot\Phi }\,.
\eas
Integrating the inequality in $[0,t^{*}]$ we get
\bas
z(t^{*})\exp\pr{- \int_{0}^{t^{*}}\!\!\dot\Phi }-z(0)< \frac{\left[ \wm_{i}^{\nu,\cont}(0)\right](I)}{4}\cdot \int_{0}^{t^{*}}\!\!\!\exp\pr{-\int_{0}^{r}\!\!\dot\Phi }dr
 .
\eas
Of course, having $z\geq0$, $\left[ \wm_{i}^{\nu,\cont}(0)\right](I)<0$ and $\dot\Phi\geq0$, this implies
\bas
 -z(0)\leq \frac{\left[ \wm_{i}^{\nu,\cont}(0)\right](I)}{4}\cdot t^{*} \exp\pr{-\int_{0}^{t^{*}}\dot\Phi(s)ds }
\ .
\eas
Such inequality, taking into account that $\int_{0}^{T}\dot\Phi$ is uniformly bounded by some constant depending only on time and on the system, can be rewritten as
\ba\label{E:estCase1-}
&0<-\left[ \wm_{i}^{\nu,\cont}(0)\right]^{}(I)\leq C(T)\cdot \frac{\Ll^{1}(I) }{t^{*}}
&&
 C(T)\geq4e^{\int_{0}^{T}\dot \Phi}\ .
\ea

\paragraph{Case 2} Suppose, in addition to $\left[ \wm_{i}^{\nu,\cont}(0)\right](I)<0$, that
\bas
\dot z(r)-\dot\Phi(r)z(r)\geq\frac{\left[ \wm_{i}^{\nu,\cont}(0)\right](I)}{4}
&&\text{at some time }r\in(0,t^{*})\ .
\eas
By genuine nonlinearity and Rankine-Hugoniot conditions, and by the parameterization choice,
\bas
\abs{ \left[\wm_{i}^{\nu,\jump}(t)\right](a(t))-\pr{ \widetilde{\lambda_{i}}(t,a(t+))- \widetilde{\lambda_{i}}(t,a(t-))}}\leq2\varepsilon_{\nu}+\beta_\nu\,.
\eas
and the same holds for $b$. Again by genuine nonlinearity, since $\wm_{i}^{\nu,\jump}\leq0$, thus
\bas
\xi(t)&\geq \left[\wm_{i}^{\nu,\jump}(t)\right](a(t))+ \left[\wm_{i}^{\nu,\jump}(t)\right](b(t))-4\varepsilon_{\nu} {-2\beta_{\nu}}
\\
&\geq \left[\wm_{i}^{\nu,\jump}(t)\right](I(t))-4\varepsilon_{\nu} {-2\beta_{\nu}}\,,
\qquad\qquad I(t)=[a(t),b(t)]\,.
\eas
We thus find that $\left[\wm_{i}^{\nu,\jump}(t)\right](I(t))-\xi(t)\leq 4\varepsilon_{\nu}{+2\beta_{\nu}}$: plugging this in turn into~\eqref{est:Bressan} we get
\bas
\dot z(r) -\dot\Phi(r)z(r)\leq \left[\wm_{i}^{\nu,\cont}(r)\right](I(r))+\OL(1)\varepsilon_{\nu} {+2\beta_{\nu}}\ .
\eas
Under the assumption of Case 2, at some $r\in(0,t^{*})$ we thus find
\ba\label{E:grgqqggqqrggqr}
\frac{\left[ \wm_{i}^{\nu,\cont}(0)\right](I)}{4}\leq\left[\wm_{i}^{\nu,\cont}(r)\right](I(r))+\OL(1)\varepsilon_{\nu} {+2\beta_{\nu}}\ .
\ea
\paragraph{Conclusion} Suppose now $J=\cup_{\ell=1}^{k}I_{\ell}$ where $\{I_{\ell}\}_{\ell=1}^{k}$ is a disjoint family of $k$ closed intervals.
 Let $S_{1}$ be the set of indexes $\ell$ where the interval $I_{\ell}$, bounded by the two curves $x_{m}^{\nu,\ell}$ and $x_{m}^{\nu,r}$ only, fall in Case 1.
 Let $S_{2}=\{1,\dots,k\}\setminus S_{1}$ be thus the set of indexes $\ell$ where the interval $I_{\ell}$, bounded by the two curves $x_{m}^{\nu,\ell}$ and $x_{m}^{\nu,r}$ only, fall in Case 2.
 
Focusing only on intervals that fall in Case 1, thanks to~\eqref{E:estCase1-} on each interval and by additivity of measures, the set $J_{1}\doteq{}\cup_{\ell\in S_{1}}I_{\ell}$ satisfies
\ba\label{E:auxjdjejjeje}
&0<-\left[ \wm_{i}^{\nu,\cont}(0)\right]^{}(J_{1})\leq C(T)\cdot \frac{\Ll^{1}(J_{1}) }{t^{*}}\,,
&&
 C(T)=4e^{\int_{0}^{T}\dot \Phi} \ .
\ea

On the other hand, by additivity of measures jointly with \eqref{E:grgqqggqqrggqr}, the set $J_{2}\doteq{}\cup_{\ell\in S_{2}}I_{\ell}$ satisfies
\ba\label{E:robuvpbwqwgqqgqgw}
&\frac{\left[ \wm_{i}^{\nu,\cont}(0)\right](J_{2})}{4}\leq\sum_{\ell\in S_{2}}\left[\wm_{i}^{\nu,\cont}(r_{\ell})\right](I_{_{\ell}}(r_{\ell}))+k\OL(1)\varepsilon_{\nu}+2k\beta_{\nu}
\ea
for some $r_{\ell}\in (0,t)$ for $\ell\in S_{2}$\,.

We now work on the sum in the right hand side to bring every addend to the initial time $s=0$.
Reorder such special times so that $r_{\ell_{m}}\leq r_{\ell_{m+1}}$ for $m=1,\dots,\#S_{2}$.
By the continuous balance estimate~\eqref{E:contest}
\bas
 \left[\wm_{i}^{\nu,\cont}(r_{\ell_{m+1}})\right](I_{\ell_{m+1}}(r_{\ell_{m+1}}))\leq & \left[\wm_{i}^{\nu,\cont}(r_{\ell_{m}})\right](I_{\ell_{m+1}}(r_{\ell_{m}}))+
 \\
 &+\OL(1) \left(\mu_{\nu}^{ICJS} \right)\left(A_{I_{\ell_{m+1}}}^{r_{\ell_{m}},r_{\ell_{m+1}}}\right)+\OL(1) \varepsilon_\nu\,,
\eas
with $A_{I_{\ell_{m+1}}}^{r_{\ell_{m}},r_{\ell_{m+1}}}$ as in~\eqref{E:ffdgagwagfffff} relative only to the generalized, approximate $i$-characteristics $x_{m+1}^{\nu,\ell}(t)$, $x_{m+1}^{\nu,r}(t)$.
Add $I_{\ell_{m}}(r_{\ell_{m}})$: if $I_{\ell_{m}}(r_{\ell_{m}})$ and $ I_{\ell_{m+1}}(r_{\ell_{m}})$ intersect we estimate such intersection with \[\wm_{i}^{\nu,\cont}(\{(r_{\ell_{m}},x_{\ell_{m}}^{\nu,r}(r_{\ell_{m}}))\})\leq \OL(1) \varepsilon_\nu\] as at all points; with the convention in Remark~\ref{R:relation}, by additivity of measures
\bas
 \sum_{q=m}^{m+1}\left[\wm_{i}^{\nu,\cont}(r_{\ell_{q}})\right](I_{\ell_{q}}(r_{\ell_{q}}))\leq 
 &\left[\wm_{i}^{\nu,\cont}(r_{\ell_{m}})\right]((I_{\ell_{m}}\cup I_{\ell_{m+1}})(r_{\ell_{m}}))+
 \\
 &+\OL(1) \left(\mu_{\nu}^{ICJS} \right)\left(A_{I_{\ell_{m+1}}}^{r_{\ell_{m}},r_{\ell_{m+1}}}\right)+\OL(1) \beta_\nu\,.
\eas
Once more, denoting $J_{m}^{m+1}\doteq{}I_{\ell_{m}}\cup I_{\ell_{m+1}}$, by the continuous balance estimate~\eqref{E:contest} jointly with Remark~\ref{R:relation}
\bas
\left[\wm_{i}^{\nu,\cont}(r_{\ell_{m}})\right](J_{m}^{m+1}(r_{\ell_{m}}))\leq
&\left[\wm_{i}^{\nu,\cont}(r_{\ell_{m}})\right](J_{m}^{m+1}(r_{\ell_{m-1}}))+
\\
&+
\OL(1) \left(\mu_{\nu}^{ICJS} \right)\left(A_{J_{m}^{m+1}}^{r_{\ell_{m-1}},r_{\ell_{m}}}\right)+
\OL(1) \beta_\nu\,,
\eas
so that
\bas
 \sum_{q=m}^{m+1}\left[\wm_{i}^{\nu,\cont}(r_{\ell_{q}})\right](I_{\ell_{q}}(r_{\ell_{q}}))\leq 
& \left[\wm_{i}^{\nu,\cont}(r_{\ell_{m}})\right](J_{m}^{m+1}(r_{\ell_{m-1}}))+
 \\&+
\OL(1) \left(\mu_{\nu}^{ICJS} \right)\left(A_{J_{m}^{m+1}}^{r_{\ell_{m-1}},r_{\ell_{m+1}}}\right)+
\OL(1) 2\beta_\nu\,,
\eas
Iterating this argument back up to time $s=0$, and setting $r_{\ell_{0}}\doteq{}r_{0}=0$, we get
\ba\label{E:sommaAux}
\sum_{\ell\in S_{2}}\left[\wm_{i}^{\nu,\cont}(r_{\ell})\right](I_{\ell}(r_{\ell}))\leq &\left[\wm_{i}^{\nu,\cont}(0)\right](J_{2})+\\
\notag
&+\OL(1) \left(\mu_{\nu}^{ICJS} \right)\left(A_{\{x_{m}^{\nu,\ell/r}\}}^{0,t}\right)+k\OL(1) \beta_\nu\,.
\ea
We used above that $A_{I_{\ell}}^{r_{\ell_{i}},r_{\ell_{i+1}}}\subseteq A_{I_{\ell}}^{0,r_{\ell}}\subseteq A_{\{x_{m}^{\nu,\ell/r}\}}^{0,t}$ for all $\ell\in S_{2}$ and $\mu_{\nu}^{ICJS} $ nonnegative.
By~\eqref{E:robuvpbwqwgqqgqgw}-\eqref{E:sommaAux} thus
\ba\label{E:auxjdjejjeje333he}
-\frac34 {\left[ \wm_{i}^{\nu,\cont}(0)\right](J_{2})}     & \leq\OL(1)  \left(\mu_{\nu}^{ICJS} \right)\left(A_{\{x_{m}^{\nu,\ell/r}\}}^{0,t}\right)+k\OL(1)\beta_{\nu}  \ .
\ea
Joining the two cases, thus adding~\eqref{E:auxjdjejjeje} and~\eqref{E:auxjdjejjeje333he} multiplied by $4/3$, we get the claim~\eqref{E:thesisnfakvaargrarw}.
\end{proof}

\subsection{Decay estimate for the positive continuous part}
\label{Ss:decay+}

We keep the notation introduced in \S~\ref{S:balances-} and, once more, the relation in Remark~\ref{R:relation} among the vanishingly small parameters.

Let $i\in\{1,\dots,N\}$.
 
\begin{lemma}
Let $0\leq s\leq t\leq T$. 
Suppose $J=\cup_{m=1}^{k}[x_{m}^{\nu,\ell}(t),x_{m}^{\nu,r}(t)]$ is the union of $k$ disjoint intervals, where the generalized $i$-characteristics are as in~\eqref{E:orderCurves} and the region $A_{\{x_{m}^{\nu,\ell/r}\}}^{s,t}$ is defined in~\eqref{E:ffdgagwagfffff}.
Then
\ba\label{E:grgeraeraeaagargrwa}
 {\left[ \wm_{i}^{\nu}(t)\right](J)} \lesssim\frac{\Ll^{1}(J)}{t-s}+\left(\mu_{\nu}^{ICS} \right)\left(A_{\{x_{m}^{\nu,\ell/r}\}}^{s,t}\right)+k\beta_{\nu} \ .
\ea
\end{lemma}
This part is an extension of~\cite[Theorem 10.3]{BressanBook} relative to genuinely nonlinear, $1d$, conservation laws.

\begin{remark}
We stress that the relevant information from Estimate~\eqref{E:grgeraeraeaagargrwa} is on the positive part of the measure, while the negative is not relevant in that estimate.
\end{remark}

\begin{proof}
For simplicity of notations, by the semigroup property assume that $s=0$; denote $t^{*}$ the final time $t$. 

Consider at time $t^{*}$ any closed interval $I=[a,b]$ of which $J$ is made of, for example $[x_{m}^{\nu,\ell}(t^{*}),x_{m}^{\nu,r}(t^{*})]$. Set
%
\[
 a(t)\doteq{}x_{m}^{\nu,\ell}(t)\ , \qquad b(t)\doteq{}x_{m}^{\nu,r}(t)
 \qquad \text{for $0\leq t\leq t^{*}$.}
\]
For $t\in[0,t^{*}]$, set $I(t)=[a(t),b(t)]$ and $z(t)=\Ll^{1}(I(t))=b(t)-a(t)$. Then by definition
\[
\dot z(t)=\dot b(t)-\dot a(t)=\widetilde{\lambda_{i}^{\nu}}(t,b(t))-\widetilde{\lambda_{i}^{\nu}}(t,a(t))\ ,
\]
where $\lambda_{i}^{\nu}$ is the speed of the $i$-th wave of $u^{\nu}$ in Definition~\ref{def:velCarAppr}.

Recall the explicit expression of the \emph{piecewise} Lipschitz continuous function $\Phi$ in~\eqref{E:PhiAlberto}.
Remind again that while in~\cite{BressanBook} the function $\Phi$ can jump only at interaction times, here $\Phi$ can jump also at update times: at interaction times as before the variation of $\Phi$ is controlled by the Glimm function~\cite[(10.51)]{BressanBook}, while at update times it is controlled by the source measure by~\cite[Remark 4.3]{ACM1}. 
As before, it satisfies $\dot\Phi\geq0$ where differentiable, which is excluding interaction and update times, and $\Phi$ controls the variation of the $i$-wave due to waves of other families in this quantitative way:
\ba\label{est:Bressanbis}
\abs{\dot z(t)+\xi(t)-\left[\wm_{i}^{\nu}(t)\right](I(t))}\leq \dot\Phi(t)z(t)+\OL(1)\varepsilon_{\nu}\ ,
\ea
where we define
\[
\xi(t)\doteq{}\pr{\widetilde{\lambda_{i}}(t,a(t ))-\widetilde{\lambda_{i}}(t,a(t)-)}+\pr{\widetilde{\lambda_{i}}(t,b(t)+)-\widetilde{\lambda_{i}}(t,b(t) )}\ .
\]
Since by its very definition $\Phi(t)$ is controlled by the total variation of $u^{\nu}(t)$, then $\int_{0}^{r}\dot\Phi $ is uniformly bounded at finite times and it can be estimated in terms only of the total variation of $u^{\nu}(t)$, the Glimm function and the source measure, in turn estimated in Lemma~\ref{L:estSource}.
.

Suppose $\left[ \wm_{i}^{\nu}(t^{*})\right](I)>0$, otherwise thesis is trivial. Two cases are possible:

\paragraph{Case 1} Suppose, in addition to $\left[ \wm_{i}^{\nu,\cont}(0)\right](I)>0$,  that
\bas
\dot z(r)+\dot\Phi(r)z(r)>\frac{\left[ \wm_{i}^{\nu}(t^{*})\right](I)}{4}
&&\forall r\in(0,t^{*})\ .
\eas
By elementary calculus, when $r\in (0,t^{*})$, this assumption implies
\bas
\ddr\pr{z(r)e^{\int_{0}^{r}\dot\Phi }}>\frac{\left[ \wm_{i}^{\nu}(t^{*})\right](I)}{4}\cdot e^{\int_{0}^{r}\dot\Phi }\geq \frac{\left[ \wm_{i}^{\nu}(t^{*})\right](I)}{4}>0\ .
\eas
Integrating the inequality in $(0,t^{*})$ we get
\bas
z(t)\exp\pr{ \int_{0}^{t^{*}}\dot\Phi(s)ds }-z(s)> \frac{\left[ \wm_{i}^{\nu}(t^{*})\right](I)}{4} \cdot t^{*}
\ .
\eas
Of course, having $z\geq0$ and $\dot\Phi\geq0$,
\bas
 \exp\pr{\int_{0}^{t^{*}}\dot\Phi(s)ds }z(t)\geq \frac{\left[ \wm_{i}^{\nu}(t^{*})\right](I)}{4}\cdot t^{*}
\ .
\eas
Taking into account that $\int_{0}^{T}\dot\Phi$ is uniformly bounded by some constant depending only on time and on the system, such inequality can be rewritten as
\ba\label{E:estCase1+}
&0<\left[ \wm_{i}^{\nu}(t^{*})\right]^{}(I)\leq C(T)\cdot \frac{\Ll^{1}(I) }{t^{*}}
&&
 C(T)\geq 4e^{\int_{0}^{T}\dot \Phi}\ .
\ea

\paragraph{Case 2} Suppose, in addition to $\left[ \wm_{i}^{\nu,\cont}(0)\right](I)>0$,  that
\bas
\dot z(r)+\dot\Phi(r)z(r)\leq\frac{\left[ \wm_{i}^{\nu}(t)\right](I)}{4}
&&\text{at some time } r\in(0, t^{*})\ .
\eas
Reordering terms in the estimate~\eqref{est:Bressanbis}, and since $-\xi(r)\geq -\OL(1)  \varepsilon_\nu$, we get thus
\bas
\dot z(r)+\dot\Phi(r)z(r)\geq \left[ \wm_{i}^{\nu}(r)\right](I(r))
- \OL(1) \varepsilon_\nu \ .
\eas
Jointly with the hypothesis of Case 2, at some $r\in(0,t^{*})$ the latter estimate yields
\ba\label{E:grggggqrgrqgrgrgqgq}
 \frac{\left[ \wm_{i}^{\nu}(t^{*})\right](I)}{4}
 \geq \left[ \wm_{i}^{\nu}(r)\right](I(r))
- \OL(1) \varepsilon_\nu \ .
\ea
\paragraph{Conclusion} 
 Suppose now $J=\cup_{\ell=1}^{k}I_{\ell}$ where $\{I_{\ell}\}_{\ell=1}^{k}$ is a disjoint family of $k$ closed intervals.
 Let $S_{1}$ be the set of indexes $\ell$ where the interval $I_{\ell}$, bounded by the two curves $x_{m}^{\nu,\ell}$ and $x_{m}^{\nu,r}$ only, fall in Case 1.
 Let $S_{2}=\{1,\dots,k\}\setminus S_{1}$ be thus the set of indexes $\ell$ where the interval $I_{\ell}$, bounded by the two curves $x_{m}^{\nu,\ell}$ and $x_{m}^{\nu,r}$ only, fall in Case 2.

Focusing only on intervals that fall in Case 1, thanks to~\eqref{E:estCase1+} on each interval and by additivity of measures, the set $J_{1}\doteq{}\cup_{\ell\in S_{1}}I_{\ell}$ satisfies
\ba\label{E:auxjdjejjejebis}
&0<-\left[ \wm_{i}^{\nu}(t^{*})\right]^{}(J_{1})\leq C(T)\cdot \frac{\Ll^{1}(J_{1}) }{t^{*}}\,,
&&
 C(T)\geq4e^{\int_{0}^{T}\dot \Phi} \ .
\ea

On the other hand the set $J_{2}\doteq{}\cup_{\ell\in S_{2}}I_{\ell}$, by additivity of measures jointly with \eqref{E:grggggqrgrqgrgrgqgq}, satisfies
\ba\label{E:robuvpbwqwgqqgqgwbis}
&\frac{\left[ \wm_{i}^{\nu}(t^{*})\right](J_{2})}{4}\geq\sum_{\ell\in S_{2}}\left[\wm_{i}^{\nu}(r_{\ell})\right](I(r_{\ell}))-\OL(1)k\varepsilon_{\nu} 
\ea
for some $r_{\ell}\in (0,t^{*})$ and for $\ell\in S_{2}$.

We now work on the sum in the right hand side to bring every addend to the final time $t=t^{*}$.
Reorder such special times so that $r_{\ell_{m}}\leq r_{\ell_{m+1}}$ for $m=1,\dots,\#S_{2}$.
By the continuous balance estimate~\eqref{E:allest} 
\bas
 \left[\wm_{i}^{\nu}(r_{\ell_{m}})\right](I(r_{\ell_{m}}))\geq  &\left[\wm_{i}^{\nu}(r_{\ell_{m+1}})\right](I(r_{\ell_{m+1}}))+
 \\
 &-\OL(1) \left(\mu_{\nu}^{ICS} \right)\left(A_{I_{\ell_{m+1}}}^{r_{\ell_{m}},r_{\ell_{m+1}}}\right)-\OL(1) \varepsilon_\nu\,,
\eas
so that, arguing similarly to the proof of the decay for the negative waves in Lemma~\ref{L:grgaaear}, going up to time $t^{*}$,
\ba\label{E:rrerefrffwqfwq}
\sum_{\ell\in S_{2}}\left[\wm_{i}^{\nu}(r_{\ell})\right](I_{\ell}(r_{\ell}))
\geq 
&\left[\wm_{i}^{\nu}(t^{*})\right](J_{2})+
\\
&\notag\qquad-\OL(1) \left(\mu_{\nu}^{ICS} \right)\left(A_{\{x_{m}^{\nu,\ell/r}\}}^{0,t^{*}}\right)-k\OL(1) \beta_\nu\,,
\ea
were we also used the relations in Remark~\ref{R:relation}, that now we apply also below.
We used that $A_{I_{\ell}}^{r_{j},r_{\ell}}\subseteq A_{I_{\ell}}^{0,r_{\ell}}\subseteq A_{\{x_{m}^{\nu,\ell/r}\}}^{0,t}$ for all $\ell\in S_{2}$ and $\mu_{\nu}^{ICJS}$ being nonnegative.
By~\eqref{E:robuvpbwqwgqqgqgwbis}-\eqref{E:rrerefrffwqfwq} thus
\ba\label{E:auxjdjejjeje333}
\frac34 {\left[ \wm_{i}^{\nu}(t^{*})\right](J_{2})}     & \leq \OL(1) \left(\mu_{\nu}^{ICS} \right)\left(A_{\{x_{m}^{\nu,\ell/r}\}}^{0,t^{*}}\right)+k\OL(1)\beta_{\nu} \ .
\ea

Joining the two cases, thus adding~\eqref{E:auxjdjejjejebis} and~\eqref{E:auxjdjejjeje333}, multiplied by $4/3$ we get the claim~\eqref{E:grgeraeraeaagargrwa}.
\end{proof}

%
%
\nomenclature{$\mathrm{ceil}$}{Smallest integer bigger than a given real number, i.e.~integer part of the number plus one}
\nomenclature{$ \delta_{hk}$}{Delta di Kronecker, which is equal to $1$ if $h=k$ and it vanishes otherwise}
\nomenclature{$\hat\lambda$}{Uniform bound for the characteristic speeds of very family}
\nomenclature{$\vSC{t}{h}\bar u$}{The viscous semigroup of the Cauchy problem Eqs.~\eqref{eq:sysnc}-\eqref{eq:inda} starting at time $h$, rather than fixing the initial time $h=0$. See~\cite{Ch1}}
\nomenclature{$\vSB{t}{h}$}{The semigroup of the Cauchy problem for the homogeneous system of Eqs.~\eqref{eq:sysnc}-\eqref{eq:inda} when $g\equiv 0$ starting at time $h$, rather than fixing the initial time $h=0$. See~\cite{BB}}
\nomenclature{$\ftS{t}{h}$}{The $\varepsilon_{\nu}$-wave-front tracking approximation of $\vSB{t}{h}$ by~\cite{AMfr}}
\nomenclature{$\Omega$}{Open, bounded, connected subset of $\real^N$ where $u$ is valued}
\nomenclature{$\dssb$}{Smallness parameter for initial datum in the Cauchy problem of the viscous system of Eqs.~\eqref{eq:sysnc} as in Theorem~\ref{T:localConv}}
\nomenclature{$ \dH$}{Smallness parameter for initial datum in the Cauchy problem of the homogeneous system as at Page~\pageref{E:rgabab} which provides a threshold for global existence and convergence of wave-front tracking approximations. See~\cite{AMfr}, where it is denoted by $\delta_{0}$}
\nomenclature{$\varepsilon$}{Positive parameter}
\nomenclature{$\nu$}{Positive integer parameter relative to subsequences of $\varepsilon_{\nu}$-wave-front tracking approximations or of $(\varepsilon_{\nu},\tau_{\nu})$-fractional-step approximations. See~\S~\ref{sec:PCA}}
\nomenclature{$\varepsilon_{\nu}$}{Positive vanishing constant, as $\nu\uparrow\infty$, in $\varepsilon_{\nu}$-wave-front tracking approximation and in $(\varepsilon_{\nu},\tau_{\nu})$-fractional-step approximation. It correspond also to maximum size of rarefactions}
\nomenclature{$\tau_{\nu}$}{Size of the time-step in $(\varepsilon_{\nu},\tau_{\nu})$-fractional-step approximations. See~\S~\ref{sec:PCA}}
\nomenclature{$u_{\nu}$, $u$}{Often, $\varepsilon_{\nu}$-wave front tracking approximation of the Cauchy problem in Equation~
\eqref{eq:sysnc}-\eqref{eq:inda}, either homogeneous or not, and its limit entropy solution, as constructed in~\cite{AMfr} and recalled in \S~\ref{Ss:frontTr}}
\nomenclature{$\ww_{\nu}$, $\ww$}{Fractional step approximation, and its limit as $\nu\uparrow\infty$, of the Cauchy problem in Eqs.~\eqref{eq:sysnc}-\eqref{eq:inda} that we review in~\S\S~\ref{sec:PCA}-\ref{s:reviewDisc}. We fix the right-continuous representative in time and space}
\nomenclature{$\ell_{g}$}{The Lipschitz constant of $g(\cdot, x,u)$ in $x,u$. See the Assumption~\textbf{(G)} at Page~\pageref{Ass:G}}
\nomenclature{$\alpha$}{See the assumption \textbf{(G)} at Page~\pageref{alpha}}
\nomenclature{$\Phi$}{Functional defined at Equation~
\eqref{eq2:RP1}}
\nomenclature{$\Qg$, $\mathcal{V}$, $\Upsilon$}{Functionals defined
in Eqs.~\eqref{eq:ip}-\eqref{eq:Ups}}
\nomenclature{$\lesssim$}{ Less or equal up to a constant which only depends only on the flux of Equation~
\eqref{eq:sysnc} and on $\dss$ or $\dH$}
\nomenclature{$\mu_{(\nu)}^{I}$, $\mu_{(\nu)}^{IC}$}{Interaction and interaction-cancellation measures~\eqref{E:interactionMeas}-\eqref{E:muICl}, either on a sequence of ($\varepsilon_{\nu}$-$\tau_{\nu}$)-fractional step approximations or a fixed limit of them}

\nomenclature{$(\beta,i)$-approximate discontinuity}{As in Definition~\ref{D:subdisc}}
\nomenclature{$\Theta$}{At most countable set containing interaction points~\eqref{E:muICl}}
\nomenclature{$\TV(v)$}{The total variation of $v:\real\to\real^m$, $m\in\nat$, which is $\sup_{K\in\nat}\sup_{x_{1}<\dots<x_{K}}\sum_{j}|v(x_{j+1})-v(x_{j})|$}


\begin{thenomenclature} 
\nomgroup{A}
  \item [{$ \delta_{hk}$}]\begingroup Delta di Kronecker, which is equal to $1$ if $h=k$ and it vanishes otherwise\nomeqref {B.32}\nompageref{34}
  \item [{$ \dH$}]\begingroup Smallness parameter for initial datum in the Cauchy problem of the homogeneous system as at Page~\pageref{E:rgabab} which provides a threshold for global existence and convergence of wave-front tracking approximations. See~\cite{AMfr}, where it is denoted by $\delta_{0}$\nomeqref {B.32}\nompageref{34}
  \item [{$ \widetilde{\lambda_{i}^{\nu}}$}]\begingroup Speed of the $i$-th wave of the $\nu$-approximate solution $u^{\nu}$ in Definition~\ref{def:velCarAppr}\nomeqref {B.7}\nompageref{22}
  \item [{$(\beta,i)$-approximate discontinuity}]\begingroup As in Definition~\ref{D:subdisc}\nomeqref {B.32}\nompageref{34}
  \item [{$[D_{x}\lambda_{i}(u)]_{i}$}]\begingroup The $i$-component of $D_{x}\lambda_{i}(u)$ in Definition~\ref{D:icomponent}\nomeqref {3.1}\nompageref{8}
  \item [{$\alpha$}]\begingroup See the assumption \textbf{(G)} at Page~\pageref{alpha}\nomeqref {B.32}\nompageref{34}
  \item [{$\beta_{\nu}$}]\begingroup See Remark~\ref{R:relation}\nomeqref {B.5}\nompageref{20}
  \item [{$\dssb$}]\begingroup Smallness parameter for initial datum in the Cauchy problem of the viscous system of Eqs.~\eqref{eq:sysnc} as in Theorem~\ref{T:localConv}\nomeqref {B.32}\nompageref{34}
  \item [{$\ell_{g}$}]\begingroup The Lipschitz constant of $g(\cdot, x,u)$ in $x,u$. See the Assumption~\textbf{(G)} at Page~\pageref{Ass:G}\nomeqref {B.32}\nompageref{34}
  \item [{$\ftS{t}{h}$}]\begingroup The $\varepsilon_{\nu}$-wave-front tracking approximation of $\vSB{t}{h}$ by~\cite{AMfr}\nomeqref {B.32}\nompageref{34}
  \item [{$\Ha^{M}$}]\begingroup The Hausdorff $M$-dimensional external measure, where $M\in\N\cup\{0\}$\nomeqref {2.2}\nompageref{5}
  \item [{$\hat\lambda$}]\begingroup Uniform bound for the characteristic speeds of very family\nomeqref {B.32}\nompageref{34}
  \item [{$\lesssim$}]\begingroup  Less or equal up to a constant which only depends only on the flux of Equation~ \eqref{eq:sysnc} and on $\dss$ or $\dH$\nomeqref {B.32}\nompageref{34}
  \item [{$\Ll^{M}$}]\begingroup The Lebesgue $M$-dimensional measure, where $M\in\N\cup\{0\}$\nomeqref {2.2}\nompageref{5}
  \item [{$\mathrm{ceil}$}]\begingroup Smallest integer bigger than a given real number, i.e.~integer part of the number plus one\nomeqref {B.32}\nompageref{34}
  \item [{$\mu^{I}$, $\mu^{IC}$}]\begingroup An interaction and an interaction-cancellation measures of $u$ in Definition~\ref{D:ICmeasures}\nomeqref {B.5}\nompageref{20}
  \item [{$\mu^{I}_{\nu}$, $\mu^{IC}_{\nu}$}]\begingroup The interaction and interaction-cancellation measures of $u^{\nu}$ in Definition~\ref{D:ICmeasuresnu}\nomeqref {B.3}\nompageref{20}
  \item [{$\mu_{(\nu)}^{I}$, $\mu_{(\nu)}^{IC}$}]\begingroup Interaction and interaction-cancellation measures~\eqref{E:interactionMeas}-\eqref{E:muICl}, either on a sequence of ($\varepsilon_{\nu}$-$\tau_{\nu}$)-fractional step approximations or a fixed limit of them\nomeqref {B.32}\nompageref{34}
  \item [{$\mu_{\nu}^{\sour}$}]\begingroup The $\nu$-approximate source measure in Definition~\ref{D:measuresnu}\nomeqref {B.7}\nompageref{22}
  \item [{$\mu_{\nu}^{ICJS}$}]\begingroup The interaction-cancellation-jump-source measure in Definition~\ref{D:measuresnu}\nomeqref {B.7}\nompageref{22}
  \item [{$\mu_{\nu}^{ICJS}$}]\begingroup The interaction-cancellation-source measure in Definition~\ref{D:measuresnu}\nomeqref {B.7}\nompageref{22}
  \item [{$\mu_{i}^{\nu,\jump} $}]\begingroup The $\nu$-jump-wave-balance measure in Definition~\ref{D:measuresnu}\nomeqref {B.7}\nompageref{22}
  \item [{$\mu_{i}^{\nu} $}]\begingroup The $\nu$-approximate wave-balance measure in Definition~\ref{D:measuresnu}\nomeqref {B.7}\nompageref{22}
  \item [{$\nu$}]\begingroup Positive integer parameter relative to subsequences of $\varepsilon_{\nu}$-wave-front tracking approximations or of $(\varepsilon_{\nu},\tau_{\nu})$-fractional-step approximations. See~\S~\ref{sec:PCA}\nomeqref {B.32}\nompageref{34}
  \item [{$\Omega$}]\begingroup Open, bounded, connected subset of $\real^N$ where $u$ is valued\nomeqref {B.32}\nompageref{34}
  \item [{$\Phi$}]\begingroup Functional defined at Equation~ \eqref{eq2:RP1}\nomeqref {B.32}\nompageref{34}
  \item [{$\Qg$, $\mathcal{V}$, $\Upsilon$}]\begingroup Functionals defined in Eqs.~\eqref{eq:ip}-\eqref{eq:Ups}\nomeqref {B.32}\nompageref{34}
  \item [{$\SBV(I;\Omega)$}]\begingroup See Definition~\ref{D:cantorPart1d}, when $\Omega\subset \R^{N}$ is open and $I= (a,b)\subseteq\R$\nomeqref {1.6}\nompageref{2}
  \item [{$\SBV_{\loc}((t_{1},t_{2})\times\R;\Omega)$}]\begingroup Subspace of $\BV_{\loc}((t_{1},t_{2})\times\R;\Omega)$, where $\Omega\subseteq \R^{M}$ is open, consisting of those functions whose derivative does not have the Cantor part, see~\cite[\S~3.9]{AFPBook}\nomeqref {2.2}\nompageref{5}
  \item [{$\tau_{\nu}$}]\begingroup Size of the time-step in $(\varepsilon_{\nu},\tau_{\nu})$-fractional-step approximations. See~\S~\ref{sec:PCA}\nomeqref {B.32}\nompageref{34}
  \item [{$\Theta$}]\begingroup At most countable set containing interaction points~\eqref{E:muICl}\nomeqref {B.32}\nompageref{34}
  \item [{$\varepsilon$}]\begingroup Positive parameter\nomeqref {B.32}\nompageref{34}
  \item [{$\varepsilon_{\nu}$}]\begingroup Positive vanishing constant, as $\nu\uparrow\infty$, in $\varepsilon_{\nu}$-wave-front tracking approximation and in $(\varepsilon_{\nu},\tau_{\nu})$-fractional-step approximation. It correspond also to maximum size of rarefactions\nomeqref {B.32}\nompageref{34}
  \item [{$\vSB{t}{h}$}]\begingroup The semigroup of the Cauchy problem for the homogeneous system of Eqs.~\eqref{eq:sysnc}-\eqref{eq:inda} when $g\equiv 0$ starting at time $h$, rather than fixing the initial time $h=0$. See~\cite{BB}\nomeqref {B.32}\nompageref{34}
  \item [{$\vSC{t}{h}\bar u$}]\begingroup The viscous semigroup of the Cauchy problem Eqs.~\eqref{eq:sysnc}-\eqref{eq:inda} starting at time $h$, rather than fixing the initial time $h=0$. See~\cite{Ch1}\nomeqref {B.32}\nompageref{34}
  \item [{$\wm_i$}]\begingroup The $i$-th wave measure in Definition~\ref{D:upsiloni}\nomeqref {2.4}\nompageref{6}
  \item [{$\wm_i^{\nu}$}]\begingroup The $i$-wave measure of $u^{\nu}$, Definition~\ref{D:approximateDiscC}\nomeqref {B.1}\nompageref{20}
  \item [{$\wm_{\beta,i}^{\nu,\cont}$}]\begingroup The approximate $(\beta,i)$-continuity measure of $u^{\nu}$, Definition~\ref{D:approximateDiscC}\nomeqref {B.1}\nompageref{20}
  \item [{$\wm_{\beta,i}^{\nu,\jump}$}]\begingroup The approximate $(\beta,i)$-jump measure of $u^{\nu}$, Definition~\ref{D:approximateDiscC}\nomeqref {B.1}\nompageref{20}
  \item [{$\ww_{\nu}$, $\ww$}]\begingroup Fractional step approximation, and its limit as $\nu\uparrow\infty$, of the Cauchy problem in Eqs.~\eqref{eq:sysnc}-\eqref{eq:inda} that we review in~\S\S~\ref{sec:PCA}-\ref{s:reviewDisc}. We fix the right-continuous representative in time and space\nomeqref {B.32}\nompageref{34}
  \item [{$u_{\nu}$, $u$}]\begingroup Often, $\varepsilon_{\nu}$-wave front tracking approximation of the Cauchy problem in Equation~ \eqref{eq:sysnc}-\eqref{eq:inda}, either homogeneous or not, and its limit entropy solution, as constructed in~\cite{AMfr} and recalled in \S~\ref{Ss:frontTr}\nomeqref {B.32}\nompageref{34}

\end{thenomenclature}

\section*{Acknowledgments}
All authors are members of the Gruppo Nazionale per l'Analisi Matematica, la Probabilit\`a e le loro Applicazioni (GNAMPA) of the Istituto Nazionale di Alta Matematica (INdAM) and are supported by the PRIN national project ``Hyperbolic Systems of Conservation Laws and Fluid Dynamics: Analysis and Applications''. This research was supported by PRIN 2020 ``Nonlinear evolution PDEs, fluid dynamics and transport equations: theoretical foundations and applications'' and PRIN PNRR P2022XJ9SX of the European Union – Next Generation EU.


\end{document}